\documentclass[11pt,a4paper,oneside]{amsart}
\usepackage{amssymb,amsmath,mathrsfs,amsthm,amsrefs,graphicx,color}
\usepackage{leftidx}

\theoremstyle{plain}
\newtheorem{theorem}{Theorem}[section]
\newtheorem{definition}[theorem]{Definition}
\newtheorem{lemma}[theorem]{Lemma}
\newtheorem{corollary}[theorem]{Corollary}
\newtheorem{proposition}[theorem]{Proposition}
\newtheorem{assumption}{Assumption}
\theoremstyle{remark}
\newtheorem{remark}[theorem]{Remark}

\allowdisplaybreaks[4]

\usepackage{hyperref}

\numberwithin{equation}{section}

\makeatletter
\@namedef{subjclassname@2020}{\textup{2020} Mathematics Subject Classification}
\makeatother

\newcommand{\C}{\mathbb{C}}
\newcommand{\R}{\mathbb{R}}
\newcommand{\Z}{\mathbb{Z}}

\renewcommand{\Im}{\operatorname{Im}}
\renewcommand{\Re}{\operatorname{Re}}
\newcommand{\I}{\infty}

\newcommand{\norm}[1]{\left\lVert #1\right\rVert}

\newcommand{\IN}{\quad\text{in }}

\def\({\left(}
\def\){\right)}
\def\<{\left\langle}
\def\>{\right\rangle}
\def\le{\leqslant}
\def\ge{\geqslant}

\DeclareMathOperator{\sign}{sign}

\DeclareMathOperator{\rank}{rank}

\DeclareMathOperator{\tr}{tr}

\newcommand{\todayd}{\the\year/\the\month/\the\day}

\theoremstyle{definition}

\newcommand{\ol}{\overline}

\begin{document}
\title[]
{On scalar-type standing-wave solutions to systems of nonlinear Schr\"odinger equations}

\author[S. Masaki]{Satoshi MASAKI}
\address[]{Division of Mathematical Science, Department of Systems Innovation, Graduate School of Engineering Science, Osaka University, Toyonaka, Osaka, 560-8531, Japan}
\email{masaki@sigmath.es.osaka-u.ac.jp}

\begin{abstract}
In this article, we study the standing-wave solutions to a class of systems of nonlinear Schr\"odinger equations.
Our target is all the standard forms of the NLS systems, with two unknowns,  that have a common linear part and cubic gauge-invariant nonlinearities and that yield a Hamiltonian with a coercive kinetic-energy part.
We give a necessary and  sufficient condition on the existence of the ground state. 
Further, we give a characterization of the shape of the ground state.
It will turn out that the ground states are scalar-type, i.e., multiples of a constant vector and a scalar function.
We further give a sufficient condition on the existence of excited states of the same form.
The stability and the instability of the ground states are also studied.
To this end, we introduce an abstract treatment on the study of scalar-type standing-wave solution that applies to a wide class of NLS systems with homogeneous energy-subcritical nonlinearity.
By the argument, some previous results are reproduced.
\end{abstract}

\maketitle

\section{Introduction}

This article studies the standing-wave solutions (solitons)
 to a class of  systems of nonlinear Schr\"odinger equations on $\R^d$.
The specific systems we have in our mind are the following:
\begin{enumerate}
\item Type 1, decoupled systems:
\begin{equation}\tag{NLS1}\label{E:nls1}
	\left\{
	\begin{aligned}
	(	i \partial_t +\Delta) u_1 
		&= \alpha |u_1|^2 u_1 ,\\
	(	i \partial_t +\Delta) u_2
		&= \beta |u_2|^2 u_2 ,
	\end{aligned}
	\right.
\end{equation}
where $\alpha, \beta \in \{-1,0,1\}$ satisfy $\alpha \ge \beta$.
\item Type 2, systems with componentwise mass-conservation:
\begin{equation}\tag{NLS2}\label{E:nls2}
	\left\{
	\begin{aligned}
	(	i \partial_t +\Delta) u_1 
		&= \alpha |u_1|^2 u_1 +  \sigma (|u_1|^2+|u_2|^2) u_1 ,
		\\
	(	i \partial_t +\Delta) u_2
		&= \beta |u_2|^2 u_2 + \sigma (|u_1|^2+|u_2|^2) u_2 ,
	\end{aligned}
	\right.
\end{equation}
where $\alpha \ge \beta$ and $\sigma \in \{-1,1\}$.
\item Type 3:
\begin{equation}\tag{NLS3}\label{E:nls3}
	\left\{
	\begin{aligned}
	(	i \partial_t +\Delta) u_1 
		&= (3 \alpha_1+\alpha_2) |u_1|^2 u_1 +(\alpha_1-\alpha_2)(2|u_2|^2 u_1+ u_2^2 \overline{u_1}) + r (|u_1|^2+|u_2|^2) u_1 ,
		\\
	(	i \partial_t +\Delta) u_2
		&= (3 \alpha_1+\alpha_2 )|u_2|^2 u_2  +(\alpha_1-\alpha_2) (2|u_1|^2 u_2+ u_1^2 \overline{u_2})+ r (|u_1|^2+|u_2|^2) u_2 ,
	\end{aligned}
	\right.
\end{equation}
where $\alpha_2\ge0$, $\alpha_1^2\neq \alpha_2^2$, $\alpha_1^2+\alpha_2^2=1$,  and $r\in\R$.
\item Type 4:
\begin{equation}\tag{NLS4}\label{E:nls4}
	\left\{
	\begin{aligned}
	(	i \partial_t + \Delta) u_1 
		&= (3 \alpha_1+\alpha_2 + 2\alpha_3) |u_1|^2 u_1 +(\alpha_1-\alpha_2) (2|u_2|^2 u_1+ u_2^2 \overline{u_1})
		\\& \quad  + r (|u_1|^2+|u_2|^2) u_1 ,
		\\
	(	i \partial_t + \Delta) u_2
		&= (3 \alpha_1+\alpha_2 - 2\alpha_3) |u_2|^2 u_2  +(\alpha_1-\alpha_2) (2|u_1|^2 u_2+ u_1^2 \overline{u_2})\\
		&\quad  + r (|u_1|^2+|u_2|^2) u_2 ,
	\end{aligned}
	\right.
\end{equation}
where
$ \alpha_2\ge 0$, $\alpha_3>0 $, $\alpha_1\neq \alpha_2$,
 $\alpha_1^2+\alpha_2^2+\alpha_3^2=1$, and $r\in\R$.
\item Type 5:
\begin{equation}\tag{NLS5}\label{E:nls5}
	\left\{
	\begin{aligned}
	(	i \partial_t + \Delta) u_1 
		&= (3 \alpha_1+\alpha_2 + 2\alpha_3 \cos \eta) |u_1|^2 u_1 +(\alpha_1-\alpha_2) (2|u_2|^2 u_1+ u_2^2 \overline{u_1})
		\\& \quad +\alpha_3 \sin \eta (2|u_1|^2 u_2 + u_1^2\ol{u_2}+|u_2|^2u_2) + r (|u_1|^2+|u_2|^2) u_1 ,
		\\
	(	i \partial_t + \Delta) u_2
		&= (3 \alpha_1+\alpha_2 - 2\alpha_3 \cos \eta) |u_2|^2 u_2  +(\alpha_1-\alpha_2) (2|u_1|^2 u_2+ u_1^2 \overline{u_2})\\
		&\quad +\alpha_3 \sin \eta (2|u_2|^2 u_1 + u_2^2\ol{u_1}+|u_1|^2u_1) + r (|u_1|^2+|u_2|^2) u_2 ,
	\end{aligned}
	\right.
\end{equation}
where
$ \alpha_2>0$, $\alpha_3>0 $,  $\alpha_1^2+\alpha_2^2+\alpha_3^2=1$, $\eta \in (0,\pi)$, and $r\in\R$.
We assume $\alpha_1>0$ if $ \eta > \pi/2$.
\end{enumerate}
Note that the systems of \eqref{E:nls3} and \eqref{E:nls4} are degenerate cases $\alpha_3=0$ and $\eta=0$
of \eqref{E:nls5}, respectively.

By a standard argument, the well-posedness of the  above systems is established in $(H^1(\R^d))^N$ for $d\le4$ (See \cite{CazBook}, for instance).
These systems possess three conserved quantities.
One is \emph{mass} (or \emph{charge})
\begin{equation}\label{E:mass}
	M(u_1,u_2) = \int_{\R^d}\( \frac12 |u_1|^2 + \frac12 |u_2|^2 \)dx
\end{equation}
and the others two are \emph{momentum}
\begin{equation}\label{E:momentum}
	P(u_1,u_2) = \Im \int_{\R^d}\( \overline{u_1}\nabla u_1 + \overline{u_2}\nabla u_2 \)dx
	=\int_{\R^d} \xi (|\hat{u_1}|^2 + |\hat{u_2}|^2) d\xi
\end{equation}
and \emph{energy}
\begin{equation}\label{E:energy}
	E(u_1,u_2) =  \int_{\R^d}\( \frac12 |\nabla u_1|^2 + \frac12 |\nabla u_2|^2 
 + \frac14 g(u_1,u_2)  \) dx,
\end{equation}
where $g$ is a suitable quartic homogeneous polynomial given by the nonlinearity.

In \cites{M,MSU1,MSU2}, the author and his collaborators give the classification of cubic systems which possess at least one mass-like conserved quantity. 
According to the classification, one sees that the above \eqref{E:nls1}--\eqref{E:nls5} are the complete list of the standard forms (or the canonical forms)
of the systems which have gauge-invariant cubic nonlinearities and which
admit the conserved energy with a coercive kinetic-energy part.
See Theorem \ref{T:reduction} for more detail.

The main purpose of the present article is to study the standing-wave solutions (solitons) to the above systems.
It will turn out that the ground state solutions, i.e., the least action solutions, to the above systems are \emph{scalar-type}.
Here, we say a vector-valued function is \emph{scalar-type} if the function is given as a multiple of a constant vector and a scalar function.
There is a number of literature on the study of standing waves of the NLS systems (see \cites{AC,MMP,Si,LW,I} and references therein).
Some of the above systems are already studied.

In this article, we introduce a unified treatment to study the \emph{scalar-type} solutions of the following
elliptic system 
\begin{equation}\tag{gE${}_\omega$} \label{E:gE}
	-\Delta u_j+ \omega u_j 
		=- F_j( u_1,u_2,\dots,u_N)  \qquad (j=1,2,\dots,N),
\end{equation}
where $x\in \R^d$,
$d\ge1$, $N\ge2$, $\omega>0$, $u_j$ are complex-valued unknowns, and $F_j$ are nonlinearities.
Let us emphasize that the parameter $\omega>0$ is independent of $j$.
This restriction is made because it is a natural requirement for the presence of standing-wave solutions to NLS systems without a constant-type linear potential such as \eqref{E:nls1}--\eqref{E:nls5} (See Remark \ref{R:sameomega}).
As $\omega$ reflects the decay speed of the components at the spatial infinity, it is quite natural to expect that
\eqref{E:gE} possesses scalar-type solutions.
Except for the equality of the parameter $\omega$, our argument applies to a wide range of systems.
Roughly speaking, we make only two assumptions on nonlinearity:
One is the existence of a simple Hamiltonian structure (i.e., the existence of a Hamiltonian with the normalized kinetic-energy part), and the other is the homogeneity of the function that determines the nonlinear part of the Hamiltonian.
The precise assumption on the nonlinearity is given below in Assumption \ref{A:1}.
We give a simple criterion on the existence/nonexistence of the ground state and its shape (Theorem \ref{T:main}).
That of scalar-type excited states is also established (Theorem \ref{T:excited}).

As is well known, solutions to the nonlinear elliptic system of the form \eqref{E:gE} 
have a connection between standing wave solutions to an NLS system.
Fix  a vector ${\bf n}=(n_1,\dots, n_N) \in \Z_+^N$ and consider the NLS system
\begin{equation}\tag{gNLS}\label{E:gNLS}
	\left\{
	\begin{aligned}
	&(	i \partial_t +n_j\Delta) u_j 
		= n_j F_j(  u_1, u_2,\dots , u_N)\qquad (j=1,2,\dots,N), \\
	&(u_1(0),u_2(0),\dots, u_N(0))=(u_{1,0},u_{2,0},\dots,u_{N,0}) \in (H^1(\R^d))^N,
	\end{aligned}
	\right.
\end{equation}
where ${\bf u}(t)=(u_1(t),u_2(t),\dots,u_N(t))$,
$(t,x) \in \R^{1+d}$, $d\ge1$, and $N\ge2$.
To link \eqref{E:gE} and \eqref{E:gNLS}, we need one more assumption on the nonlinearity (Assumption \ref{A:2}).
This is a generalized version of the gauge-invariance property. The vector ${\bf n}$ is given in the assumption.
The systems \eqref{E:nls1}--\eqref{E:nls5} are obtained with the choice $N=2$ and ${\bf n} =(1,1)$.
Here we consider general ${\bf n}$ because we would compare our result with previous results on the scalar-type solutions.
In particular, we will find that our theorems reproduce results in \cite{CO}.

\subsection{Scalar-type solutions to the nonlinear elliptic systems}

We first consider the nonlinear elliptic system \eqref{E:gE}.
Let us make the assumption on the nonlinearity clear.
\begin{assumption}\label{A:1}
The nonlinearity is given by the formula
\begin{equation}\label{D:fj}
	F_j := \tfrac2p \partial_{\overline{z_j}} g 
\end{equation}
for $j=1,2,\dots,N$ from a function $g: \C^N \to \R$ which is twice continuously differentiable as a function $\R^{2N} \to \C$
and is homogeneous of degree $p$:
\begin{equation}\label{E:g-homo}
	g(r z_1, r z_2, \dots , r z_N) = r^{p}g( z_1, z_2.\dots,z_N) 
\end{equation}
for all $( z_1, z_2, \dots, z_N) \in \mathbb{C}^N$ and $r\ge0$. Here, $p$ satisfies the energy subcriticality condition
\begin{equation}\label{E:prange}
	p \in (2,2^*),
\end{equation}
where $2^*=\infty$ for $d=1,2$ and
$2^*= 2d/(d-2)$ for $d\ge3$.
\end{assumption}
We remark that the ${z}$-derivative and the
$\overline{z}$-derivative of a function $f(z)$ are given by $\partial_{{z}} f =\tfrac12 (\partial_x - i \partial_y)\tilde{f}$ and
$\partial_{\overline{z}} f = \tfrac12(\partial_x + i \partial_y)\tilde{f}$, respectively, with the identifications $z=x+iy$ and $f(z) = \tilde{f}(x,y)$. Since $g$ is real-valued, one has $\partial_{\overline{z_j}}g = \overline{\partial_{z_j}g} $.
This gives us an alternative representation of the nonlinearity:
\begin{equation}\label{D:fj-alt}
	F_j = \tfrac1{p}(\overline{\partial_{{z_j}} g } + \partial_{\overline{z_j}} g )
	= \tfrac2p \overline{\partial_{{z_j}} g }.
\end{equation}
By differentiating the both sides of \eqref{E:g-homo} by $\overline{z_j}$ and applying the chain rule, we obtain
\begin{equation}\label{E:fj-homo}
	F_j(r  z_1, r  z_2,\dots, r  z_N) = r^{p-1} F_j ( z_1, z_2.\dots,z_N) 
\end{equation}
for all $( z_1, z_2, \dots, z_N) \in \mathbb{C}^N$ and $r\ge 0$.

For $\omega >0$,
let $S_\omega$ be the action functional on $(H^1(\R^d))^N$ given by
\begin{equation}\label{E:Action}
	S_\omega (u_1,u_2,\dots , u_N) := E(u_1,u_2,\dots , u_N) + \omega M(u_1,u_2,\dots , u_N)
\end{equation}
with
\begin{equation}\label{E:gEnergy}
	E(u_1,u_2,\dots , u_N) = \sum_{j=1}^N \int_{\R^d} \frac12|\nabla u_j|^2dx + \int_{\R^d} \frac1p g(u_1,u_2,\dots,u_N)dx
\end{equation}
and
\begin{equation}\label{E:gMass}
	M(u_1,u_2,\dots , u_N) = \sum_{j=1}^N \int_{\R^d} \frac12| u_j|^2dx,
\end{equation}
where $g$ is the function given in Assumption \ref{A:1}.
Note that ${\bf u} \in (H^1(\R^d))^N$ solves \eqref{E:gE} if and only if it is a critical point $S_\omega$,
i.e., $S'_\omega({\bf u})=0$.

\begin{definition}
For $\omega>0$, let
\[
	\mathcal{A}_{\omega} := \{ \Phi = (\phi_1,\phi_2, \dots , \phi_N) \in (H^1(\R^d))^N \ |\ \Phi \neq0,\, S'_\omega(\Phi) =0 \}
\]
be the set of solutions to the elliptic system \eqref{E:gE}
and let
\[
\mathcal{G}_\omega := \left\{ \Phi \in \mathcal{A}_\omega \ \middle|\ S_\omega (\Phi) = \inf_{\Psi \in \mathcal{A}_\omega} S_\omega (\Psi)\right\}
\]
be the set of ground states. Further, we let $\mathcal{A}= \cup_{\omega>0}\mathcal{A}_\omega$
and $\mathcal{G}= \cup_{\omega>0}\mathcal{G}_\omega$.
We call an element of $\mathcal{A} \setminus \mathcal{G}$ an excited state.
\end{definition}

Let $Q \in H^1(\R^d)$ be the positive radial solution to the elliptic equation
$-\Delta Q + Q = Q^{p-1}$ on $\R^d$.
For $\omega>0$ and $a>0$, we denote 
\[
Q_{\omega,a}:= (\omega/a)^{ 1/(p-2)} Q(\sqrt{\omega} \cdot).
\]
Note that $\omega>0$ is a scale parameter and $a>0$ is an amplitude parameter.
 Obviously, $Q_{\omega,a}$ is a positive radial solution to 
\begin{equation}\label{E:Qaeq}
-\Delta Q_{\omega,a} + \omega Q_{\omega,a} = a Q_{\omega,a}^{p-1}.
\end{equation}
Let
\[
	\partial B :=\{ (z_1,z_2,\dots,z_N) \in \C^N \ |\ |z_1|^2 + |z_2|^2+ \dots + |z_N|^2 = 1 \}
\]
be the boundary of the unit ball in $\C^{N}$.
For ${\bf w} =(w_1,w_2,\dots,w_N) \in \partial B$, $\omega>0$, 
 and $a>0$, we define
\begin{equation}
	\mathcal{R}({\bf w},\omega, a) := \{ {\bf w}Q_{\omega,a}(\cdot -y) = ( w_1 Q_{\omega,a}(\cdot -y) , w_2 Q_{\omega,a}(\cdot -y) , \dots , w_N Q_{\omega,a}(\cdot -y) ) \ |\  y \in \R^d \} . 
\end{equation}
This is a set of scalar-type functions. 

Set
\begin{equation}\label{D:gmin}
	g_{\min} = \min_{(z_1,z_2,\dots,z_N) \in \partial B} g(z_1,z_2,\dots,z_N)
\end{equation}
and
\begin{equation}\label{D:T0}
	T_0 := \{ (z_1,z_2,\dots,z_N) \in \partial B \  |\ g(z_1,z_2,\dots,z_N) =g_{\min} \}.
\end{equation}
Our main result on the ground states
is formulated as follows:
\begin{theorem}\label{T:main}
Suppose Assumption \ref{A:1}.
If $g_{\min} < 0$ then  
\[
	\mathcal{G}_\omega = \bigcup_{{\bf w} \in T_0} \mathcal{R}({\bf w},\omega,-g_{\min})
\]
for all $\omega>0$.
If $g_{\min} \ge 0$ then $\mathcal{A}=\emptyset$, i.e., there is no nontrivial solution to \eqref{E:gE} for all $\omega>0$.
\end{theorem}

We remark that the theorem tells us that $g_{\min}<0$ is a necessary and sufficient condition of the existence of the ground states.
The existence of the ground state is studied in \cites{BeLi,BrLi,St} in an abstract setting.
Our result gives a precise description of the shape of the ground states.
This kind of characterization is introduced in \cite{Co} to study some specific systems of the form \eqref{E:gE}. 

Let us next turn to the excited states.
We have the following:
\begin{theorem}\label{T:excited}
Suppose Assumption \ref{A:1}.
Let ${\bf w} \in \partial B$ and $a>0$.
$
\mathcal{R}({\bf w},\omega,a)
	\subset \mathcal{A}_\omega$ holds for all $\omega>0$ 
if and only if
${\bf w} \in \partial B$ 
 is a critical point of  $g|_{\partial B}$ and $a=-g({\bf w})>0$. 
 \end{theorem}
The theorem implies that if $g|_{\partial B}$ has a negative critical value other than its minimum value then \eqref{E:gE} admits scalar-type excited states. 
Our theorems reveal that the shape of the scalar-type solutions to \eqref{E:gE}
is well described in the language of the function $g$, which is the novelty of our theorems.

\begin{remark}
We remark that a homogeneous function satisfying \eqref{E:g-homo} is identified with a pair of the degree $p$
and the function on $\partial B$. Indeed, for a function $g$ satisfying \eqref{E:g-homo},
the pair $(p,g|_{\partial B})$ is uniquely specified obviously.
On the other hand, for the prescribed pair $(q,h)$, where $q>0$ and $h: \partial B \to \R$, the function
$\C^N \ni z\mapsto |z|^q h(\frac{z}{|z|}) \in \R$ satisfies \eqref{E:g-homo}.
The important quantity $g_{\min}$ and set $T_0$ given in \eqref{D:gmin} and \eqref{D:T0}, respectively, are defined only by the function part.
\end{remark}

\begin{remark}
If $d\ge2$ then the nonlinear elliptic equation $-\Delta Q + Q = |Q|^{p-2}Q$ admits sign-changing solutions.
If the function $g$ satisfies the relation $g(-{\bf z})=g({\bf z})$ for all ${\bf z} \in \C^N$
then 
the scalar-type functions given by the sign-changing solutions instead of the positive solution $Q$
are also contained in $\mathcal{A}$. They are all excited states. See Remark \ref{R:morees} below.
As seen in the next subsection, the function $g$ corresponding to the systems \eqref{E:nls1}--\eqref{E:nls5} satisfies the assumption.
\end{remark}

\smallbreak
It is known that the ground states are related to the optimizer of the Gagliardo-Nirenberg inequality.
This is true also in our case.
Let us introduce a Gagliardo-Nirenberg-type inequality adopted in the current context. Besides its own interest, this inequality is used below in the study of the instability of the ground state.
\begin{theorem}[A sharp Gagliardo-Nirenberg-type inequality]\label{T:sGN}
Suppose Assumption \ref{A:1}. Suppose $g_{\min}<0$.
It holds for all ${\bf u} \in (H^1(\R^d))^N$ that
\[
	-G({\bf u}) \le C_{\mathrm{GN}} M({\bf u})^{\frac{p}2-\frac{d(p-2)}{4}} H({\bf u})^{\frac{d(p-2)}4},
\]
where the constant $C_{\mathrm{GN}}$ is given by
\[
	C_{\mathrm{GN}} = (\tfrac2{p-2})^{\frac{p}2} (\tfrac2d)^{\frac{d(p-2)}4} (d- \tfrac{d-2}{2}p)^{\frac{p-2}2}  \|Q\|_2^{2-p}
		(-g_{\min}).
\]
Further, the equality holds if and only if ${\bf u}$ is a constant multiple of an element of $\mathcal{G}$.
\end{theorem}

\subsection{Linking with NLS systems and stability/instability of the ground states}

Let us now turn to the study of standing-wave solutions to \eqref{E:gNLS}.
Let us introduce the notion of an $H^1$-solution to \eqref{E:gNLS}.
\begin{definition}
We say an $N$-tuple of functions ${\bf u}(t)=(u_1(t),u_2(t),\dots,u_N(t))$ be an $H^1$-solution to \eqref{E:gNLS}
on an interval $I\subset \R$, $0\in \overline{I}$,  
if ${\bf u}$ belongs to
\[
(C(I,H^1(\R^d)) \cap L^{q_d}_{\mathrm{loc}}(I,W^{1,r_d}(\R^d)))^N
\] 
and satisfies
\[
	u_j(t) = e^{itn_j\Delta} u_{j,0} - i \int_0^t e^{i(t-s)\Delta} n_j {F}_j({\bf u}(s)) ds\qquad(j=1,2,\dotsm,N)
\]
on $I$. Here, $(q_1,r_1)=(4,\infty)$, $(q_2,r_2)=(\frac{2p}{p-2},p)$, and $(q_d,r_d)=(2,2^*)$ for $d\ge 3$.
Let $I_{\max}=(T_{\min},T_{\max})$ denote the maximal lifespan of the solution.
A soluiton extended to that on its maximal lifespan is called a maximal-lifespan solution.
We say a solution is global if $I_{\max}=\R$.
\end{definition}
The local well-posedness of \eqref{E:gNLS}  in the $H^1$-framework is established under Assumption \ref{A:1} 
by a standard fixed point argument with Strichartz' estimates (See e.g. \cite{CazBook} for details).
The solutions have the conserved \emph{energy} defined in \eqref{E:gEnergy}.
The energy conservation can be seen by means of the identity
\[
	\Re \sum_{j=1}^N {F}_j \overline{\partial_t u_j} 
	= \partial_t (\tfrac1p  {g}({\bf u})).
\]

To relate the nonlinear elliptic system \eqref{E:gE} and the nonlinear Schrodinger system,
we make one more assumption.

\begin{assumption}\label{A:2}
There exists a vector ${\bf n}=(n_1,\dots, n_N) \in \Z_+^N$ such that
the function $g$ satisfies the gauge  condition
\begin{equation}\label{E:g-gauge}
	g(e^{in_1\theta} z_1, e^{in_2\theta} z_2, \dots , e^{in_N\theta} z_N) =  g( z_1, z_2.\dots,z_N) 
\end{equation}
for all $( z_1, z_2, \dots, z_N) \in \mathbb{C}^N$ and $\theta \in \R$.
\end{assumption}

One easily sees that \eqref{E:g-gauge} yields
\begin{equation}\label{E:fj-gauge}
	F_j(e^{in_1\theta}  z_1, e^{in_2\theta}  z_2,\dots, e^{in_N\theta}  z_N) = e^{in_j\theta} F_j ( z_1, z_2.\dots,z_N) 
\end{equation}
for all $( z_1, z_2, \dots, z_N) \in \mathbb{C}^N$ and $\theta \in \R$.

If Assumption \ref{A:2} is fulfilled with the same vector as in \eqref{E:gNLS}
then an $H^1$-solution ${\bf u}(t)$ to \eqref{E:gNLS} also conserves the
 \emph{mass} (or \emph{charge}) $M({\bf u})$ defined in \eqref{E:gMass}.
Note that the conservation of
mass  is verified with the property $\sum_{j=1}^N n_j \Im  (F_j({\bf z}) \overline{z_j}) =0$ for all ${\bf z}=( z_1, z_2, \dots, z_N) \in \mathbb{C}^N$, which follows by differentiating \eqref{E:g-gauge} by $\theta$ at $\theta =0$.

In the rest of this subsection, we suppose that Assumptions \ref{A:1} and \ref{A:2} are satisfied and consider
\eqref{E:gNLS} with the same vector ${\bf n}$ given in Assumption \ref{A:2}.

\begin{theorem}[Standing wave solutions]\label{T:soliton}
Suppose Assumptions \ref{A:1} and \ref{A:2}, with the same vector ${\bf n}$ as in \eqref{E:gNLS}.
Let $\mathcal{A}_\omega$ be the set of solutions to \eqref{E:gE} with the nonlinearity ${F}_j$.
If $\Phi\in \mathcal{A}_\omega$ then
$
	(e^{i  n_1 \omega t }\phi_1,e^{i  n_2 \omega t } \phi_2,\dots ,e^{i  n_N \omega t } \phi_N)
$
is a global $H^1$-solution to \eqref{E:gNLS}. 
\end{theorem}

\begin{remark}\label{R:sameomega}
The equality of the parameter $\omega>0$ in \eqref{E:gE} corresponds to the
absence of a constant-type linear potential in \eqref{E:gNLS}.
Namely,   
consider 
\begin{equation}\tag{gNLS'}
	(	i \partial_t +n_j\Delta - n_j b_j) u_j 
		= n_j F_j(  u_1, u_2,\dots , u_N)\qquad (j=1,2,\dots,N), 
\end{equation}
where ${\bf b} = (b_1,\dots,b_N) \in \R^N$. Suppose that $F_j$ satisfies Assumption \ref{A:1} and \ref{A:2}, with the same vector ${\bf n}$ as in the equation.
Note that a suitable application of the gauge transform $u_j(t,x) \mapsto e^{in_j \theta t} u_j(t,x)$ allows us to assume $\min_{j\in[1,N]} b_j =0$.
Then, a soliton solution to the system takes the form
$
	(e^{i  n_1 \omega t }\phi_1,e^{i  n_2 \omega t } \phi_2,\dots ,e^{i  n_N \omega t } \phi_N)
$
with a solution $\Phi=(\phi_1, \dots, \phi_N) \in (H^1(\R^d))^N$ to the elliptic system
\[
 	-\Delta \phi_j + (\omega+b_j) \phi_j = -F_j (\Phi)\qquad (j=1,2,\dots,N).
\]
Thus, our assumption on the system \eqref{E:gE} reads also as ${\bf b}=0$.
\end{remark}

Let us turn to the study of the stability/instability of the ground states.
The notion of the stability of the ground state in this article is as follows: 

\begin{definition}[Orbital stability/instability]
Let $\omega>0$.
We say the family of the ground states $ \mathcal{G}_\omega$
is stable if
for any $\varepsilon>0$ there exists $\delta>0$ such that if ${\bf u}_0 \in (H^1(\R^d))^N$ satisfies
\[
	\inf_{\Phi \in \mathcal{G}_\omega} \| {\bf u}_0 - \Phi \|_{(H^1(\R^d))^N} \le \delta
\]
then the corresponding 
$H^1$-solution ${\bf u(t)}$ to \eqref{E:gNLS} exists globally in time and satisfies
\begin{equation}\label{E:instability_def}
	\sup_{t\in \R}\inf_{\Phi \in \mathcal{G}_\omega} \| {\bf u}(t) - \Phi \|_{(H^1(\R^d))^N} \le \varepsilon.
\end{equation}
We say the family of the ground states $ \mathcal{G}_\omega$
is unstable if it is not stable.
\end{definition}

By applying a standard argument, we obtain the following.

\begin{theorem}\label{T:stability}
Suppose Assumptions \ref{A:1} and \ref{A:2}, with the same vector ${\bf n}$ as in \eqref{E:gNLS}.
Assume that $g_{\min}<0$ and let $\mathcal{G}$ be the set of ground states obtained in Theorem  \ref{T:main}.
If $2<p<2+\frac4d$ then $\mathcal{G}_\omega$ is stable for all $\omega>0$.
\end{theorem}

\begin{theorem}\label{T:instability}
Suppose Assumptions \ref{A:1} and \ref{A:2}, with the same vector ${\bf n}$ as in \eqref{E:gNLS}.
Assume further that ${\bf n}=(1,\dots,1)$ when $d=1$ and $p\ge 6$ or $d=2$ and $p>6$.
Assume that $g_{\min}<0$ and let $\mathcal{G}$ be the set of ground states obtained in Theorem  \ref{T:main}.
If $2+\frac4d\le p<2^*$ is satisfied
then $\mathcal{G}_\omega$ is unstable for all $\omega>0$.
\end{theorem}

The proof of the above instability result is done by establishing that one can pick data that gives a blowup or grow-up solution in an arbitrary neighborhood of a ground state.
Here, we say a solution ${\bf u}(t)$ blows up if $T_{\max}<\infty$ and
a solution ${\bf u}(t)$ grows up if $T_{\max}=\infty$ and $\lim_{t\to\infty} \|  {\bf u}(t) \|_{(H^1(\R^d))^N}=\infty$.  
We remark that in the latter case we have 
\[
	\inf_{\Phi \in \mathcal{G}_\omega} \| {\bf u}(t) - \Phi \|_{(H^1(\R^d))^N}
	\ge \|  {\bf u}(t) \|_{(H^1(\R^d))^N} - \sup_{\Phi \in \mathcal{G}_\omega} 
	\|  \Phi \|_{(H^1(\R^d))^N}\to \infty
\]
as $t\to\infty$.
This clearly implies the failure of \eqref{E:instability_def}.

\begin{remark}
We make a comment on the restriction of ${\bf n}$ in the instability theorem.
Consider the following system
\begin{equation}\label{E:mNLS}
	(	i \partial_t + \kappa_j \Delta) u_j 
		= \mu_j {F}_j({\bf u})\qquad (j=1,2,\dots,N),
\end{equation}
where $(t,x) \in \R^{1+d}$, $d\ge1$, and $\mu_j$ is a constant. 
Here, $\kappa_j>0$ are constant. Suppose that Assumptions \ref{A:1} and \ref{A:2} are satisfied.
The case $\kappa_j=\mu_j=n_j$ corresponds to \eqref{E:gNLS}, where ${\bf n} = (n_1,\dots,n_N)$ is the vector 
given in Assumption \ref{A:2}.
Another important case is the so-called mass-resonance case $\kappa_j=n_j^{-1}$.
In this case, for any choice of $\mu_j$,
the system possesses several properties such as
the Galilean invariance property and the validity of the pseudo-conformal transform 
and the virial type identity.
We use the latter property in the proof of Theorem \ref{T:instability}.
This is why we assume ${\bf n}=(1,\dots,1)$, in which case \eqref{E:gNLS}
becomes a system with the mass-resonace.
On the other hand, there is an alternative approach for $d\ge 2$ and $p\le 6$. 
We prove a blowup or grow-up result for radial solutions to \eqref{E:gNLS}  without mass-resonance condition by means of
 a localized version of the virial identity,  
 as in \cite{DF,IKN,NP,NP2} (See Theorems \ref{T:blowup} and \ref{T:grow-up}).
\end{remark}

\subsection{Ground states and excited states for \eqref{E:nls1}--\eqref{E:nls5}}
Let us resume the study of the systems \eqref{E:nls1}--\eqref{E:nls5}, which is the main interest of the paper.
From the argument we developed, what we do for specific systems is to find critical values and critical points of the corresponding function $g|_{\partial B}$.
As seen in Theorem \ref{T:main}, the ground state exists if and only if the minimum value $g_{\min}$ is negative
and the shape of the ground state (restriced to a specific scaling) is characterized in terms of this value and the set $T_0$ of minimum points.
Further, if $g|_{\partial B}$ has a negative critical value other than its minimum value then there is a scalar-type excited states, thanks to Theorem \ref{T:excited}.
The stability and the instability of the ground state follow from Theorems \ref{T:stability} and \ref{T:instability}, respectively.

\subsubsection{Application to \eqref{E:nls1}}
To begin with, we consider \eqref{E:nls1}.
The result is obvious from that for the single cubic NLS.
Here we state it for completeness. 
One has 
\[
	g(u_1,u_2)=\alpha |u_1|^4 + \beta |u_2|^4
\]
with $\alpha,\beta \in \{-1,0,1\}$ and $\alpha \ge \beta$.
The result is as follow.
\begin{corollary}\label{C:app1}
Let $1 \le d \le 3$.
If $\beta\ge0$ then $\mathcal{A}=\emptyset$.
If $\beta=-1$ then we have the following:
\begin{enumerate}
\item For all $\omega>0$, the following scalar-type solitons exist:
\begin{itemize}
\item $A_{1,\omega}:= \bigcup_{\theta \in \R/2\pi \Z}\mathcal{R} ((0,e^{i\theta}),\omega,1) \subset \mathcal{A}_\omega$;
\item If $\alpha=-1$ then
$A_{2,\omega}:= \bigcup_{\theta \in \R/2\pi \Z}\mathcal{R} ((e^{i\theta},0),\omega,1) \subset \mathcal{A}_\omega$.
\end{itemize}
\item The ground state is given as follows:
\[
	\mathfrak{G}_\omega =
	\begin{cases}
	A_{1,\omega} & \alpha >\beta=-1, \\
	A_{1,\omega} \cup A_{2,\omega} & \alpha =\beta=-1. 
	\end{cases}
\]
Further, $\mathcal{G}_\omega$ is stable if $d=1$ and unstable if $d=2,3$.
\end{enumerate}
\end{corollary}

\subsubsection{Application to \eqref{E:nls2}}

Let us next consider \eqref{E:nls2}.
The ground states for this system is studied in \cites{LW} (See also \cites{AC,Si,CZ,MMP,I,IT,WY}).
One has 
\[
	g(u_1,u_2)=\alpha |u_1|^4 + \beta |u_2|^4+\sigma (|u_1|^2+|u_2|^2)^2
\]
with $\alpha\ge\beta$ and $\sigma\in \{\pm1\}$.

In the special case $\alpha=\beta=0$ and $1\le d \le 3$. We have $\mathcal{A}=\emptyset$ if $\sigma=1$
and 
\[
	\mathcal{G}_\omega = \bigcup_{{\bf w} \in \partial B} \mathcal{R} ({\bf w},\omega,1)
\]
if $\sigma=-1$. $\mathcal{G}_\omega$ is stable if $d=1$ and unstable if $d=2,3$.
For the other case, we obtain the following result by the analysis of $g$.
\begin{corollary}\label{C:app2}
Let $(\alpha,\beta)\neq(0,0)$ and $1 \le d \le 3$.
If $\beta >0$ and $\frac{\alpha\beta}{\alpha+ \beta} + \sigma \ge 0$ or if $\beta \le 0$ and
$\beta+\sigma \ge0$ then $\mathcal{A}=\emptyset$.
Otherwise, we have the following:
\begin{enumerate}
\item For all $\omega>0$, the following scalar-type solitons exist:
\begin{itemize}
\item If $\beta+\sigma<0$ then
$A_{3,\omega}:= \bigcup_{\theta \in \R/2\pi \Z}\mathcal{R} ((0,e^{i\theta}),\omega,-\beta-\sigma) \subset \mathcal{A}_\omega$;
\item If $\alpha+\sigma<0$ then
$A_{4,\omega}:= \bigcup_{\theta \in \R/2\pi \Z}\mathcal{R} ((e^{i\theta},0),\omega,-\alpha-\sigma) \subset \mathcal{A}_\omega$;
\item If $\beta>0$ and $\sigma=-1$ and if $\frac{\alpha\beta}{\alpha+\beta}<1$
then
\[
A_{5,\omega}:= \bigcup_{\theta_1,\theta_2 \in \R/2\pi \Z}\mathcal{R} ((\sqrt{\tfrac{\beta}{\alpha+\beta}}e^{i\theta_1},\sqrt{\tfrac{\alpha}{\alpha+\beta}} e^{i\theta_2}),\omega,1-\tfrac{\alpha \beta}{\alpha+\beta}) \subset \mathcal{A}_\omega;
\]
\item If $\alpha<0$ and if $\frac{\alpha\beta}{\alpha+\beta}<-\sigma$
then
\[
A_{6,\omega}:= \bigcup_{\theta_1,\theta_2 \in \R/2\pi \Z}\mathcal{R} ((\sqrt{\tfrac{\beta}{\alpha+\beta}}e^{i\theta_1},\sqrt{\tfrac{\alpha}{\alpha+\beta}} e^{i\theta_2}),\omega,-\sigma-\tfrac{\alpha \beta}{\alpha+\beta}) \subset \mathcal{A}_\omega.
\]
\end{itemize}
\item The ground state is given as follows:
\begin{itemize}
\item If $\beta>0$ and $\sigma=-1$ and if $\frac{\alpha\beta}{\alpha+\beta}<1$
then $\mathcal{G}_\omega=A_{5,\omega}$.
\item If $\beta\le0$ and if $\beta+\sigma<0$ then
\[
	\mathcal{G}_\omega=	\begin{cases}
	A_{3,\omega} \cup A_{4,\omega} & \alpha = \beta<0, \\
	A_{3,\omega} & otherwise.
	\end{cases}
\]
\end{itemize}
Further, $\mathcal{G}_\omega$ is stable if $d=1$ and unstable if $d=2,3$.
\end{enumerate}
\end{corollary}

\subsubsection{Application to \eqref{E:nls3}}
One has
\[
	g(u_1,u_2)=\alpha_1 |u_1^2+u_2^2|^2 + \alpha_2 |u_1^2-u_2^2|^2 -4\alpha_2 |u_1|^2|u_2|^2+ (2\alpha_1 + r) (|u_1|^2+|u_2|^2)^2,
\]
where $\alpha_2\ge0$, $\alpha_1^2\neq \alpha_2^2$, $\alpha_1^2+\alpha_2^2=1$,  and $r\in\R$.
We have the following.
\begin{corollary}\label{C:app3}
Let $1 \le d \le 3$. It holds that $g_{\min}=\min (3\alpha_1-\alpha_2+r, 2\alpha_1+r)$.
If $g_{\min}\ge 0$ then $\mathcal{A}=\emptyset$.
Otherwise, we have the following:
\begin{enumerate}
\item For all $\omega>0$, the following scalar-type solitons exist:
\begin{itemize}
\item If $3\alpha_1 + \alpha_2+ r<0$ then
\[
A_{7,\omega}:= \bigcup_{\theta \in \R/2\pi \Z}\mathcal{R} ((0,e^{i\theta}),\omega,-3\alpha_1 - \alpha_2- r) \cup\mathcal{R} ((e^{i\theta},0),\omega,-3\alpha_1 - \alpha_2- r) \subset \mathcal{A}_\omega;
\]
\item If $3\alpha_1 - \alpha_2+ r<0$ then
\[
	A_{8,\omega}:= \bigcup_{\theta \in \R/2\pi \Z,\, \sigma \in \{\pm1\}}\mathcal{R} ((2^{-1/2}e^{i\theta}, \sigma 2^{-1/2}e^{i\theta} ),\omega,-3\alpha_1 + \alpha_2- r) \subset \mathcal{A}_\omega;
\]
\item If $2\alpha_1 + r<0$ then 
\[
	A_{9,\omega}:= \bigcup_{\theta \in \R/2\pi \Z,\, \sigma \in \{\pm1\}}\mathcal{R} ((2^{-1/2}e^{i\theta}, i\sigma 2^{-1/2}e^{i\theta} ),\omega,-2\alpha_1 - r) \subset \mathcal{A}_\omega.
\]
\item
If $\alpha_2=0$ and $3\alpha_1 +  r<0$ then
\[
A_{10,\omega}:= \bigcup_{\nu, \theta\in \R/2\pi\Z} \mathcal{R} ((e^{i\theta}\cos \nu  , e^{i\theta}\sin \nu ),\omega,-3\alpha_1 - r)\subset \mathcal{A}_\omega.
\]
\end{itemize}
\item The ground state is given as follows:
\begin{itemize}
\item If $\alpha_1>\alpha_2$ then
then $\mathcal{G}_\omega=A_{9,\omega}$.
\item If $-1<\alpha_1<\alpha_2$ then $\mathcal{G}_\omega=	A_{8,\omega}$.
\item If $-1=\alpha_1<\alpha_2=0$ then $\mathcal{G}_\omega=	A_{10,\omega}$.
\end{itemize}
Further, $\mathcal{G}_\omega$ is stable if $d=1$ and unstable if $d=2,3$.
\end{enumerate}
\end{corollary}

Notice that if $\alpha_2=0$ then $A_{7,\omega} \cup A_{8,\omega} \subset A_{10,\omega}$.

\subsubsection{Application to \eqref{E:nls4}}
One has 
\begin{align*}
	g(u_1,u_2)=& \alpha_1 |u_1^2+u_2^2|^2 + \alpha_2 |u_1^2-u_2^2|^2 -4\alpha_2 |u_1|^2|u_2|^2   \\
	&+2\alpha_3  (|u_1|^4-|u_2|^4) + (2\alpha_1 + r) (|u_1|^2+|u_2|^2)^2,
\end{align*}
where
$ \alpha_2\ge 0$, $\alpha_3>0 $, $\alpha_1\neq \alpha_2$,
 $\alpha_1^2+\alpha_2^2+\alpha_3^2=1$, and $r\in\R$.

\begin{corollary}\label{C:app4}
Let $1 \le d \le 3$.
Let $\tilde{\alpha}=\max(\alpha_1+\alpha_2,2\alpha_2)$. It holds that
\[
	g_{\min}= - \tfrac{\alpha_3^2}{\max(\tilde{\alpha},\alpha_3)}  + 3\alpha_1 + \alpha_2 - \max(\tilde{\alpha},\alpha_3) +r .
\]
If $g_{\min}\ge 0$  then $\mathcal{A}=\emptyset$.
Otherwise, we have the following:
\begin{enumerate}
\item For all $\omega>0$, the following scalar-type solitons exist:
\begin{itemize}
\item If $3\alpha_1 + \alpha_2+2\alpha_3+ r<0$ then
\[
A_{11,\omega}:= \bigcup_{\theta \in \R/2\pi \Z}\mathcal{R} ((e^{i\theta},0),\omega,-3\alpha_1 - \alpha_2-2\alpha_3- r)  \subset \mathcal{A}_\omega;
\]
\item If $3\alpha_1 + \alpha_2-2\alpha_3+ r<0$ then
\[
A_{12,\omega}:= \bigcup_{\theta \in \R/2\pi \Z}\mathcal{R} ((0,e^{i\theta}),\omega,-3\alpha_1 - \alpha_2+2\alpha_3- r)  \subset \mathcal{A}_\omega;
\]
\item If $\alpha_3\le 2\alpha_2$ and $-\tfrac{\alpha_3^2}{2\alpha_2} + 3\alpha_1 -\alpha_2 + r<0$ then 
\[
	A_{13,\omega}:= \bigcup_{\theta \in \R/2\pi \Z,\, \sigma \in \{\pm1\}}\mathcal{R} \(\( \sqrt{\tfrac{2\alpha_2-\alpha_3}{4\alpha_2}}e^{i\theta}, \sigma\sqrt{\tfrac{2\alpha_2+\alpha_3}{4\alpha_2}}e^{i\theta}\),\omega, \tfrac{\alpha_3^2}{2\alpha_2} - 3\alpha_1 +\alpha_2 - r\) \subset \mathcal{A}_\omega.
\]
\item If $\alpha_3\le |\alpha_1+\alpha_2|$ and $-\tfrac{\alpha_3^2}{\alpha_1+\alpha_2} + 2\alpha_1  + r<0$ then 
\[
	A_{14,\omega}:= \bigcup_{\theta \in \R/2\pi \Z,\, \sigma \in \{\pm1\}}\mathcal{R} \(\( \sqrt{\tfrac{\alpha_1+\alpha_2-\alpha_3}{2(\alpha_1+\alpha_2)}}e^{i\theta}, i\sigma\sqrt{\tfrac{\alpha_1+\alpha_2+\alpha_3}{2(\alpha_1+\alpha_2)}}e^{i\theta} \),\omega, \tfrac{\alpha_3^2}{\alpha_1+\alpha_2} - 2\alpha_1  - r\) \subset \mathcal{A}_\omega.
\]
\end{itemize}
\item The ground state is given as follows:
\begin{itemize}
\item If $\alpha_3 \ge \tilde{\alpha}$ then
then $\mathcal{G}_\omega=A_{12,\omega}$.
\item If $\alpha_3< \tilde{\alpha}$ and $\alpha_1<\alpha_2$ then $\mathcal{G}_\omega=A_{13,\omega}$.
\item If $\alpha_3< \tilde{\alpha}$ and $\alpha_1>\alpha_2$ then $\mathcal{G}_\omega= A_{14,\omega}$.
\end{itemize}
Further, $\mathcal{G}_\omega$ is stable if $d=1$ and unstable if $d=2,3$.
\end{enumerate}
\end{corollary}
We remark that if $\alpha_3=2\alpha_2$ then $A_{13,\omega}=A_{12,\omega}$.
Similarly, if $\alpha_3=\alpha_1+\alpha_2$ then $A_{14,\omega}=A_{12,\omega}$
and if $\alpha_3=-(\alpha_1+\alpha_2)$ then $A_{14,\omega}=A_{11,\omega}$.

\subsubsection{Application to \eqref{E:nls5}}
One has 
\begin{align*}
	g(u_1,u_2)=& \alpha_1 |u_1^2+u_2^2|^2 + \alpha_2 |u_1^2-u_2^2|^2 -4\alpha_2 |u_1|^2|u_2|^2  +2\alpha_3 \cos \eta (|u_1|^4-|u_2|^4)  \\
	&+4 \alpha_3 \sin \eta (|u_1|^2+|u_2|^2) \Re (\overline{u_1}u_2) + (2\alpha_1 + r) (|u_1|^2+|u_2|^2)^2,
\end{align*}
where
$ \alpha_2>0$, $\alpha_3>0 $,  $\alpha_1^2+\alpha_2^2+\alpha_3^2=1$, $\eta \in (0,\pi)$, and $r\in\R$.
We assume $\alpha_1>0$ if $ \eta > \pi/2$.
Let $\rho\ge0$ and $\tau \in (0,\pi)$.
The solutions to the equation
\begin{equation}\label{E:equation}
	\sin 2\theta + \rho \sin (\theta-\tau) =0
\end{equation}
play a crucial role. This equation is written as a quartic equation with respect to $\cos \theta$ or to $\sin \theta$
and hence it can be solved explicitly. We do not give it here.
Instead, we use the following.
\begin{lemma}\label{L:pzero2}
Let $\rho\ge0$ and $\tau \in (0,\pi)$ and let $f_{\rho,\tau} (\theta):=\sin 2\theta + \rho \sin (\theta -\tau)$. 
The solutions to the equation \eqref{E:equation} are described as follows:
\begin{enumerate}
\item
If $\tau \in (0,\pi/2)$ then
there exists $\rho_* = \rho_* (\tau)>0$ such \eqref{E:equation}
has four solutions if $\rho<\rho_*$, three solutions if $\rho = \rho_*$, and two solutions if $\rho > \rho_*$.
The solutions are described as follows: there exist $\theta_0  : \R_{\ge0} \to [0,\tau)$, $\theta_1  : [0,\rho_*] \to [\pi/2 ,\pi)$, $\theta_2 :[0,\rho_*]\to (\pi/2,\pi]$, and $\theta_3 :\R_{\ge0}\to (\pi+\tau ,3\pi/2]$ such that
$f_{\rho,\tau} (\theta_j(\rho))=0$ and $ \theta_j(0) = \tfrac{j\pi}{2}$  for $j=0,1,2,3$
and that $\theta_1(\rho_*) = \theta_2(\rho_*)$.
Moreover, $\theta_0$ and $\theta_1$ are strictly increasing and $\theta_2$ and $\theta_3$ are strictly decreasing in $\rho$.
\item
If $\tau = \pi/2$ then
the equation has four solutions if $\rho<\rho_*(\pi/2)$ and two solutions if $\rho \ge \rho_*(\pi/2)$. The solutions are given as follows:
\begin{align*}
	\theta_0={}&\arcsin \tfrac{\rho}2\quad (\rho \le \rho_*(\pi/2)), &
	\theta_1={}&\tfrac{\pi}2, &	
	\theta_2 ={}& \pi -\arcsin \tfrac{\rho}2\quad (\rho \le \rho_*(\pi/2)) ,&
	\theta_3={}&\tfrac{3\pi}2.
\end{align*}
Further, $\rho_*(\pi/2)=2$.
\item
If $\tau \in (\pi/2,\pi)$ then
there exists $\rho_{*} = \rho_{*} (\tau)>0$ such \eqref{E:equation}
has four solutions if $\rho<\rho_{*}$, three solutions if $\rho = \rho_{*}$, and two solutions if $\rho > \rho_{*}$.
The solutions are described as follows: there exist $\theta_0  : [0,\rho_{*}]\to [0,\pi/2)$, $\theta_1  : [0,\rho_{*}] \to (0,\pi/2]$, $\theta_2 : \R_{\ge0} \to (2/\pi,\pi]$, and $\theta_3 :\R_{\ge0}\to [3\pi/2,2\pi)$ such that
$f_{\rho,\tau} (\theta_j(\rho))=0$ and $ \theta_j(0) = \tfrac{j\pi}{2}$  for $j=0,1,2,3$
and that $\theta_0(\rho_{*}) = \theta_1(\rho_{*})$.
Moreover, $\theta_0$ and $\theta_3$ are strictly increasing and $\theta_1$ and $\theta_2$ are strictly decreasing in $\rho$.
\end{enumerate}
\end{lemma}

\begin{corollary}\label{C:app5}
Let $1 \le d \le 3$. 
Let $\theta_j$ ($j=0,1,2,3$) be the solution to \eqref{E:equation} with $\rho=\alpha_3/\alpha_2$ and $\tau=\eta$,  given in Lemma \ref{L:pzero2}.
Let $\rho_*:(0,\pi)\to \R_+$ be the function given in Lemma \ref{L:pzero2}.
If
\begin{equation}\label{E:5cond}
\alpha_1>\alpha_2 \quad \text{ and }\quad 	\alpha_3^2 < \tfrac{(\alpha_1 + \alpha_2)^2(\alpha_1 - \alpha_2)^2}{\alpha_1^2 + \alpha_2^2 - 2\alpha_1 \alpha_2 \cos 2\eta}
\end{equation}
holds then
\[
	g_{\min} = - \tfrac{\alpha_3^2(\alpha_1 - \alpha_2 \cos 2 \eta)}{\alpha_1^2-\alpha_2^2}  + 2\alpha_1 +r  .
\]
If \eqref{E:5cond} does not hold then
\[
	g_{\min} = \alpha_2  \cos 2\theta_3
+2\alpha_3 \cos (\theta_3 -  \eta ) +3\alpha_1+\alpha_2 + r.
\]
If $g_{\min}\ge 0$  then $\mathcal{A}=\emptyset$.
Otherwise, we have the following:
\begin{enumerate}
\item Let $g_j=\alpha_2  \cos 2\theta_j
+2\alpha_3 \cos (\theta_j -  \eta ) +3\alpha_1+\alpha_2 + r$ for ($j=0,1,2,3$).
For all $\omega>0$, the following scalar-type solitons exist:
\begin{itemize}
\item Suppose $\alpha_3 \le \rho_*(\eta )\alpha_2$   if $\eta\in [\tfrac\pi2,\pi)$.
If $g_0<0$ then
\[
A_{15,\omega}:= \bigcup_{\theta \in \R/2\pi \Z}\mathcal{R} ((e^{i\theta} \cos \tfrac{\theta_0}2, e^{i\theta}\sin \tfrac{\theta_0}2),\omega,-g_0)  \subset \mathcal{A}_\omega;
\]
\item Suppose $\alpha_3 \le \rho_*(\eta )\alpha_2$   if $\eta\in (0,\tfrac\pi2)\cup(\tfrac\pi2,\pi)$.
If $g_1<0$ then
\[
A_{16,\omega}:= \bigcup_{\theta \in \R/2\pi \Z}\mathcal{R} ((e^{i\theta} \cos \tfrac{\theta_1}2, e^{i\theta}\sin \tfrac{\theta_1}2),\omega,-g_1)  \subset \mathcal{A}_\omega;
\]
\item Suppose $\alpha_3 \le \rho_*(\eta )\alpha_2$   if $\eta\in (0,\tfrac\pi2]$.
If $g_2<0$ then
\[
A_{17,\omega}:= \bigcup_{\theta \in \R/2\pi \Z}\mathcal{R} ((e^{i\theta} \cos \tfrac{\theta_2}2, e^{i\theta}\sin \tfrac{\theta_2}2),\omega,-g_2)  \subset \mathcal{A}_\omega;
\]
\item If $g_3<0$ then
\[
A_{18,\omega}:= \bigcup_{\theta \in \R/2\pi \Z}\mathcal{R} ((e^{i\theta} \cos \tfrac{\theta_3}2, e^{i\theta}\sin \tfrac{\theta_3}2),\omega,-g_3)  \subset \mathcal{A}_\omega;
\]
\item 
If $\alpha_1\neq\alpha_2$,
$\alpha_3^2 \le \tfrac{(\alpha_1 + \alpha_2)^2(\alpha_1 - \alpha_2)^2}{\alpha_1^2 + \alpha_2^2 - 2\alpha_1 \alpha_2 \cos 2\eta}$, and
$- \tfrac{\alpha_3^2(\alpha_1 - \alpha_2 \cos 2 \eta)}{\alpha_1^2-\alpha_2^2}  + 2\alpha_1 +r <0$ 
are satisfied then
\[
A_{19,\omega}:= \bigcup_{\theta \in \R/2\pi \Z,\,\sigma \in \{\pm 1\} } \mathcal{R} ((e^{i\theta}w_1,e^{i\theta}w_2),\omega, \tfrac{\alpha_3^2(\alpha_1 - \alpha_2 \cos 2 \eta)}{\alpha_1^2-\alpha_2^2}  - 2\alpha_1 -r )  \subset \mathcal{A}_\omega,
\]
where $w_1=\sigma\sqrt{\tfrac{\alpha_1 + \alpha_2-\alpha_3 \cos \eta }{2(\alpha_1 + \alpha_2)} }$ and
\[
 w_2= \sqrt{\tfrac{\alpha_1 + \alpha_2}{2 (\alpha_1 + \alpha_2-\alpha_3 \cos \eta)} }\(-\tfrac{\sigma \alpha_3 \sin \eta}{\alpha_1-\alpha_2}+i \sqrt{ 1- \tfrac{\alpha_3^2(\alpha_1^2 + \alpha_2^2 - 2\alpha_1 \alpha_2 \cos 2\eta)}{(\alpha_1 + \alpha_2)^2(\alpha_1 - \alpha_2)^2} }\).
\]
\end{itemize}
\item The ground state is given as follows:
\begin{itemize}
\item If \eqref{E:5cond} holds 
then $\mathcal{G}_\omega=A_{19,\omega}$.
\item If \eqref{E:5cond} fails then $\mathcal{G}_\omega=A_{18,\omega}$.
\end{itemize}
Further, $\mathcal{G}_\omega$ is stable if $d=1$ and unstable if $d=2,3$.
\end{enumerate}
\end{corollary}

\begin{remark}
 A simple change of variable shows that Corollary \ref{C:app5} holds without the assumption $\alpha_1^2+\alpha_2^2+\alpha_3^2 =1$  by replacing $r$ with $r\sqrt{\alpha_1^2+\alpha_2^2+\alpha_3^2}$.
Under this generalization, let us fix $\alpha_1,\alpha_2$ so that $\alpha_1+\alpha_2\neq0$ and regard $\alpha_3\in (0,\infty)$ as a parameter. (We let $r$ be as in \eqref{E:nls5} and regard it as also a fixed parameter.)
Then, in the case $\eta \in (0,\pi/2)$ we have the following bifurcation diagram for sufficiently large $-r>0$:
if $\alpha_3>0$ is small then there are six scalar-type solitons.
Note that there are two branches, $\sigma=\pm1$, in $A_{19,\omega}$.
At $\alpha_3=0$, the equation becomes \eqref{E:nls3}. In the limit case, $A_{15,\omega}\cup A_{17,\omega}$ correspond to $A_{7,\omega}$,
and $A_{16,\omega}\cup A_{18,\omega}$ to $A_{8,\omega}$, and $A_{19,\omega}$ to $A_{9,\omega}$.
Let us consider (large) $\alpha_3>0$.
$A_{16,\omega}$ and $A_{17,\omega}$ become the same at $\alpha_3=\rho_*(\eta) \alpha_2$ and cease to exist
for $\alpha_3>\rho_*(\eta) \alpha_2$.
$A_{15,\omega}$ exists as long as $g_1<0$. Note that $\lim_{\alpha_3\to\infty }g_1= \infty$ holds and hence that the branch does not exist for all $\alpha_3>0$.
The branch given by $A_{18,\omega}$ exists for all $\alpha_3>0$.
The both two branches in
$A_{19,\omega}$ merge into $A_{18,\omega}$ at $\alpha_3 = (\tfrac{(\alpha_1 + \alpha_2)^2(\alpha_1 - \alpha_2)^2}{\alpha_1^2 + \alpha_2^2 - 2\alpha_1 \alpha_2 \cos 2\eta})^{1/2}$.
Before the merging, either $A_{18,\omega}$ or $A_{19,\omega}$ is the ground states.
$A_{19,\omega}$ is the ground states if and only if $\alpha_1>\alpha_2$.
After the merging, $A_{18,\omega}$ is the ground state.
One can obtain a similar diagram for $\eta \ge \pi/2$.
A pitchfork bifurcation among $A_{15,\omega}$, $A_{16,\omega}$, and $A_{17,\omega}$
is found in the case $\eta=\pi/2$.
We omit the details.
Let us observe another limit case $\eta= 0$.
Since $\theta_j = j\pi/2$ for $j=0,2$, we have
$A_{15,\omega} = A_{11,\omega}$, $A_{17,\omega} = A_{12,\omega}$, and $A_{19,\omega} = A_{14,\omega}$.
Similarly $A_{16,\omega}$ and $A_{18,\omega}$ correspond to the branches $\sigma=1$ and $\sigma=-1$ in
$A_{13,\omega}$, respectively.
\end{remark}

\subsubsection{Comments}

Let us discuss the benefit of the study of the standing-wave solutions in the present article to the study of the classification of the nonlinear systems in \cites{MSU1,MSU2,M}.

\begin{remark}
The standard forms in \eqref{E:nls1}--\eqref{E:nls5} are chosen by looking at the matrix-vector form of a system introduced in \cites{MSU2,M} (See the proof of Theorem \ref{T:reduction}).
The results in the present article would suggest a different choice of the standard forms.
We say a vector-valued function is \emph{semi-trivial} if all but one component is zero (as a function).
As seen above, the ground states for \eqref{E:nls2}--\eqref{E:nls5} are not always semi-trivial.
However, one finds that for some of them it is merely a matter of the choice of the standard form of the systems.
Namely, by applying the suitable further change of variable with a \emph{real} matrix $M \in GL_2(\R)$, we can turn them into other representatives which have the same mass and similar energy and of which ground states are semi-trivial.
Indeed, if we apply a change of variable
\[
	\begin{pmatrix} v_1 \\ v_2 \end{pmatrix}
	=
	\begin{pmatrix} \frac1{\sqrt2} & -\frac1{\sqrt2} \\ \frac1{\sqrt2} & \frac1{\sqrt2} \end{pmatrix}
	\begin{pmatrix} u_1 \\ u_2 \end{pmatrix}
\]
to \eqref{E:nls3}
then the system for $(v_1,v_2)$ is again the same \eqref{E:nls3} but the sign of $\alpha_2$ is opposite.
By Corollary \ref{C:app3}, one sees that if $-1<\alpha_1<\alpha_2$ then 
\[
	\mathcal{G}_\omega =\bigcup_{\theta \in \R/2\pi \Z,\, \sigma \in \{\pm1\}}\mathcal{R} ((2^{-1/2}e^{i\theta}, \sigma 2^{-1/2}e^{i\theta} ),\omega,\gamma)
\]
for all $\omega>0$, where $\gamma:=-3\alpha_1+2\alpha_2+r>0$.
In this case, since
\begin{align*}
	&\begin{pmatrix} \frac1{\sqrt2} & -\frac1{\sqrt2} \\ \frac1{\sqrt2} & \frac1{\sqrt2} \end{pmatrix} \mathcal{R} ((2^{-1/2}e^{i\theta},  2^{-1/2}e^{i\theta} ),\omega,\gamma) = \mathcal{R} ((0,e^{i\theta}),\omega,\gamma),\\
	&\begin{pmatrix} \frac1{\sqrt2} & -\frac1{\sqrt2} \\ \frac1{\sqrt2} & \frac1{\sqrt2} \end{pmatrix} \mathcal{R} ((2^{-1/2}e^{i\theta}, - 2^{-1/2}e^{i\theta} ),\omega,\gamma) = \mathcal{R} ((e^{i\theta},0),\omega,\gamma),
\end{align*}
the transformed system has the \emph{semi-trivial} ground states.
Further, in view of the identity
\[
		\begin{pmatrix} \cos \nu & \sin \nu \\ -\sin \nu & \cos\nu  \end{pmatrix} \mathcal{R}((e^{i\theta}\cos \nu, e^{i\theta} \sin \nu ), \omega, \gamma) = \mathcal{R} ((e^{i\theta},0),\omega,\gamma)
\]
for any  $\nu\in\R/2\pi \Z$, $\omega>0$, and $\gamma>0$, we can transform \eqref{E:nls5} so that the ground states become semi-trivial ones if \eqref{E:5cond} fails.
One option to choose the standard form would be to make the structure of the ground states as simple as possible,
in this way.
\end{remark}
\begin{remark}
As we discuss in Appendix A, there is no overlapping in the list \eqref{E:nls1}--\eqref{E:nls5}
in such a sense that neither system in the list cannot be transformed into another by a change of variables of the form
 \[
	\begin{pmatrix} v_1 \\ v_2 \end{pmatrix}
	=
M	\begin{pmatrix} u_1 \\ u_2 \end{pmatrix}, \quad M \in GL_2(\R).
\]
However, one finds that the
 change of variables of this form with \emph{unitary matrices}
produces a redundancy.
(Notice that the unitary property up to constant multiplication is necessary to keep the form of the mass and the kinetic-energy part of the Hamiltonian.)
Let us see one such example. Consider \eqref{E:nls3} with $(\alpha_1,\alpha_2)=(-1,0)$.
We first introduce the change of variable with a simple unitary matrix:
\[
	\begin{pmatrix} v_1 \\ v_2 \end{pmatrix}
	=
	\begin{pmatrix} 1 & 0 \\ 0 & i \end{pmatrix}
	\begin{pmatrix} u_1 \\ u_2 \end{pmatrix}.
\]
Then the system for $(v_1,v_2)$ takes the form
\begin{equation*}
	\left\{
	\begin{aligned}
	(	i \partial_t +\Delta) v_1 
		&=  |v_1|^2 v_1 +(2|v_2|^2 v_1+ v_2^2 \overline{v_1}) + (r-4) (|v_1|^2+|v_2|^2) v_1 ,
		\\
	(	i \partial_t +\Delta) v_2
		&= |v_2|^2 v_2  + (2|v_1|^2 v_2+ v_1^2 \overline{v_2})+ (r-4) (|v_1|^2+|v_2|^2) v_2 ,
	\end{aligned}
	\right.
\end{equation*}
which is \eqref{E:nls3} in the exceptional case $(\alpha_1,\alpha_2)=(1/2,-1/2)$.
By further introducing 
\[
	\begin{pmatrix} w_1 \\ w_2 \end{pmatrix}
	=
	\begin{pmatrix} \frac1{\sqrt2} & -\frac1{\sqrt2} \\ \frac1{\sqrt2} & \frac1{\sqrt2} \end{pmatrix}
	\begin{pmatrix} v_1 \\ v_2 \end{pmatrix},
\]
we can change the sign of $\alpha_2$, yielding
\begin{equation*}
	\left\{
	\begin{aligned}
	(	i \partial_t +\Delta) w_1 
		&=  2|w_1|^2w_1  + (r-4) (|w_1|^2+|w_2|^2) w_1 ,
		\\
	(	i \partial_t +\Delta) w_2
		&= 2|w_2|^2 w_2  +  (r-4) (|w_1|^2+|w_2|^2) w_2 .
	\end{aligned}
	\right.
\end{equation*}
This is transformed into \eqref{E:nls1} with $\alpha=\beta=1$ if $r=4$ and into \eqref{E:nls2}
with $\alpha=\beta=2/|r-4|$ and $\sigma = \sign (r-4)$ if $r\neq4$.
There might exist further reductions by a unitary matrices.
\end{remark}

\begin{remark}
Our analysis of the ground state here (or, more precisely, of the study on the position of the minimum points of the function $g|_{\partial B}$)
gives a necessary condition on the existence of the change of variable with a unitary matrix which causes
redundancy in the list \eqref{E:nls1}--\eqref{E:nls5}.
If one system is transformed into another by changing variables with a unitary matrix, then all the soliton solutions for the original system are mapped to those for the transformed system. Hence the structure of the set of the ground states for these two systems must be identical.
(Note that also that of the excited states must be identical. However, we do not have a complete characterization of $\mathcal{A}\setminus \mathcal{G}$. So, it is less useful, for now.)
One tool for the investigation of the identity of the structure is the following:
For $z=(z_1,z_2), w=(w_1,w_2) \in \C^2\setminus{(0,0)}$. we define an \emph{angle} of these two vectors by
\[
	\angle(z, w) = \arccos \frac{|(z,w)_{\C^2}|}{\sqrt{(z,z)_{\C^2}}\sqrt{(w,w)_{\C^2}} },\quad
	(z,w)_{\C^2} := z_1 \overline{w_1} + z_2 \overline{w_2}.
\]
One easily sees that this is an invariant quantity under the change with  a unitary matrix, i.e.,
$\angle (Uz, Uw)=\angle (z, w)$ holds for any unitary matrix $U \in GL_2(\C)$ and $z,w \in \C^2$.
Now, suppose that $\phi_j \in \mathcal{R}((e^{i\theta_j}\cos \nu_j,e^{i\theta_j+i\zeta_j}\sin \nu_j),\omega_j ,a_j)$ $(j=1,2)$. Then, one finds that
\[
	\cos( \angle(\phi_1(x), \phi_2(x))) = 
	 \tfrac12(1+\cos 2\nu_1 \cos 2\nu_2 + \sin 2\nu_1 \sin 2\nu_2 \cos (\zeta_1-\zeta_2))
\]
for all $x\in \R^d$.
Hence, if both $\phi_1$ and $\phi_2$ are \emph{semi-trivial} ground states, that is, if $\nu_1,\nu_2 \in \{0,\pi/2\}$
then one finds that $\cos \angle(\phi_1(x), \phi_2(x))=\delta_{\nu_1,\nu_2}$.
This fact is useful to consider the structure of the ground state.
For instance, if a picked system admits a pair of ground states such that the angle between these two is not $0$ nor $1$
then the system never be transformed (by a change with a unitary matrix) into a system that admits only semi-trivial ground states,
such as \eqref{E:nls1}, \eqref{E:nls2} with $\beta\le0$ and $(\alpha,\beta)\neq (0,0)$, or \eqref{E:nls4} with $\max(\alpha_1,\alpha_2)\le \alpha_3 -\alpha_2$.
We remark that such a pair of ground states can be picked, for instance, in the case \eqref{E:nls2} with $\beta>0$, \eqref{E:nls4} with $\max(\alpha_1,\alpha_2)>\alpha_3-\alpha_2$, and \eqref{E:nls5} with the assumption \eqref{E:5cond}.
Notice that the argument requires $ g_{\min} <0$ for the existence of the ground states.
However, the assumption can be removed by looking at the set $T_0$  of minimum points instead of the ground states.
\end{remark}

\subsection{Comparison with a result by Colin and Ohta}

Our theorems reproduce several previous results on scalar-type solitons. 
Consider the following elliptic system:
\begin{equation}\label{E:CO}
	\left\{
	\begin{aligned}
	&-\Delta u_1 + \omega u_1 = \kappa |u_1|u_1 + \gamma \overline{u_1} u_2, \\
	&-\Delta u_2 + \omega u_2 =  |u_2|u_2 + \tfrac{\gamma}2  u_1^2,
	\end{aligned}
	\right.
\end{equation}
where $d\le5$ and $\kappa \in \R$ and $\gamma>0$ are parameters.
This is of the form \eqref{E:gE} and Assumption \ref{A:1} is satisfied with  $N=2$ and
\[
	g(z_1,z_2) = - \kappa |z_1|^3 - |z_2|^3 - \tfrac{3\gamma}2 \Re (\overline{z_1}^2 z_2),
\]
which satisfies \eqref{E:g-homo} with $p=3$.
The restriction $d\le 5$ comes from \eqref{E:prange} with $p=3$.
The ground state of this system is studied in \cite{CO} (see also \cite{HOT}*{Section 7} for a study of a similar model).

We remark that $g$ satisfies \eqref{E:g-gauge} with $(n_1,n_2)=(1,2)$. Hence, in view of Theorem \ref{T:soliton},
solutions to \eqref{E:CO}
give soliton solutions to the NLS system
\begin{equation}\label{E:CONLS}
	\left\{
	\begin{aligned}
	&i\partial_t u_1 +\Delta u_1 = -\kappa |u_1|u_1 - \gamma \overline{u_1} u_2, \\
	&i\partial_t u_2 +2\Delta u_2 = -2 |u_2|u_2 - \gamma  u_1^2.
	\end{aligned}
	\right.
\end{equation}

Let 
\[
	\kappa_c(\gamma) = \tfrac12 (\gamma+2) \sqrt{\gamma-1}
\]
for $\gamma\in (0,1]$.  Let us introduce 
\begin{align*}
	J_1:={}& \{  (\gamma,\kappa)\in \R_+ \times \R  \ |\  \gamma>1 \}
	\cup \{  (\gamma,\kappa)\in \R_+ \times \R  \ |\  \gamma \le 1, \, \kappa \ge \sqrt{ 2\gamma (1-\gamma)} \} \setminus \{(1,0)\}, \\
	J_2 : ={}& \{  (\gamma,\kappa)\in \R_+ \times \R  \ |\  \gamma < 1, \, \kappa \ge \sqrt{ 2\gamma (1-\gamma)} \}, \\
	J_3 : ={}& \{  (\gamma,\kappa)\in \R_+ \times \R  \ |\  \kappa > 2^{-\frac12} \gamma^{\frac32} \}
\end{align*}
and define 
\[
	\nu_1 := \arctan \( \tfrac{\gamma}{ \kappa + \sqrt{\kappa^2 + 2\gamma (\gamma-1)}} \) \in (0,\tfrac{\pi}2)
\]
for $(\gamma,\kappa) \in J_1$,
\[
	\nu_2 := \arctan \( \tfrac{\gamma}{ \kappa - \sqrt{\kappa^2 + 2\gamma (\gamma-1)}} \) \in (0,\tfrac{\pi}2)
\]
for $(\gamma,\kappa) \in J_2$, and
\[
	\nu_3 := \arctan \( \tfrac{\gamma}{  \sqrt{\kappa^2 + 2\gamma (\gamma+1)} -\kappa } \) \in (0,\tfrac{\pi}2)
\]
for $(\gamma,\kappa) \in J_3$.

\begin{theorem}\label{T:CO}
Let $1 \le d \le 5$.

(1) For all $\omega>0$, the system \eqref{E:CO} admits following scalar-type solutions:
\begin{itemize}
\item For any $\gamma>0$ and $\kappa\in\R$, 
\[
	A_{0,\omega}:= \bigcup_{\theta \in \R/2\pi \Z} \mathcal{R}( (0, e^{2i\theta}) , \omega, 1) \subset \mathcal{A}_\omega.
\]
\item For any $(\gamma,\tau) \in J_m$, 
\[
	A_{m,\omega}:= \bigcup_{\theta \in \R/2\pi \Z} \mathcal{R}( (e^{i\theta}\cos \nu_m, e^{2i\theta}\sin \nu_m ) , \omega, - g(\cos \nu_m,\sin \nu_m)) \subset \mathcal{A}_\omega,
\]
where $m=1,2,3$.
\end{itemize}

(2) The ground state is given as follows:
\begin{itemize}
\item If $\gamma>1$ or if $\gamma \in (0,1]$ and $\kappa > \kappa_c(\gamma)$ then 
$\mathcal{G}_\omega = A_{1,\omega}$.
\item If $\gamma \in (0,1)$ and $\kappa = \kappa_c(\gamma)$ then 
$
	\mathcal{G}_\omega = 
	A_{0,\omega} \cup A_{1,\omega}.
$
\item If $\gamma \in (0,1]$ and $\kappa < \kappa_c(\gamma)$ or if $(\gamma,\kappa)=(1,0)$ then 
$
	\mathcal{G}_\omega = A_{0,\omega}.
$
\end{itemize}
Further $\mathcal{G}_\omega$ is stable if $d\le3$ and unstable if $d=4,5$.
\end{theorem}
\begin{remark}
The curve $(0,1) \ni \gamma \mapsto \kappa_c (\gamma) $ is characterized by the identity $g(\cos \nu_1,\sin \nu_1)=-1$
on $J_2$.
If $\gamma \in (0,1)$ and
$\kappa=\sqrt{2\gamma(1-\gamma)}$ then $A_{1,\omega}=A_{2,\omega}$.
\end{remark}
\begin{remark}
In \cite{CO}, $d\le 3$ is considered. Further, the excited state $A_{3,\omega}$ is not studied.
The second assertion is a reproduction of \cite{CO}*{Theorem 6} (when $d\le 3$).
\end{remark}

The rest of the paper is organized as follows:
In Section \ref{S:main}, we shall prove Theorems \ref{T:main} and \ref{T:sGN}.
Section \ref{S:excited} is devoted to the proof of Theorem  \ref{T:excited}.
Then, we turn to the study of the stability/instability of the ground states. Theorem \ref{T:stability} 
is established in Section \ref{S:stability}. Theorem \ref{T:instability} is shown in Section \ref{S:instability}
after proving a preliminary result in Section \ref{S:Pstruct}.
We then move to the study of specific systems.
In Section \ref{S:Cs}, we prove Corollaries \ref{C:app1}, \ref{C:app2}, \ref{C:app3}, \ref{C:app4}, and \ref{C:app5}. 
One treats Theorem \ref{T:CO} in Section \ref{S:CO}.
Finally, we discuss the derivation of the systems \eqref{E:nls1}--\eqref{E:nls5} in Appendix \ref{S:A}.

\section{Proof of Theorem \ref{T:main}}\label{S:main}
Let us prove Theorem \ref{T:main}. The basic strategy is quite standard. 
We minimize the action functional on the Nehari manifold (See \cites{WilBook,O}, for instance).
We give a detailed proof for completeness.
A new ingredient is a variant of rearrangement (See the proof of Lemma \ref{L:domega} below).
It is inspired by a treatment of homogeneous functions in \cites{MM,MMU}.
We remark that a similar treatment is used in \cites{Co,Co2}. 

We let $s_p:=\frac{d}2-\frac{d}p$ and
$s_c:=\frac{d}2-\frac{2}{p-2}$. Note that the condition \eqref{E:prange} reads as $s_p \in (0,\min(1,\frac{d}2))$ or $s_c \in (-\infty,\min(1,\frac{d}2))$.
Let us introduce few more functionals. Let
\[
	H({\bf u}) =  \frac12 \sum_{j=1}^N \norm{ \nabla u_j }_{L^2}^2 , \quad
	G({\bf u}) = \frac1p \int_{\mathbb{R}^d}  g({\bf u}) dx.
\]
Note that the energy functional is written as $E=H+G$.
\begin{definition}
For $\omega>0$, let us introduce  a functional
\[
	K_{\omega}({\bf u}) := 2H({\bf u}) + 2\omega M({\bf u}) + pG({\bf u}).
\]
and a set
\[
	\mathcal{K}_\omega :=\{ {\bf u} \in H^1(\R^d)^N \ |\ {\bf u} \neq0 ,\,K_\omega ({\bf u})=0\}.
\]
\end{definition}
Note that the functional $K_\omega$ has an alternative expression
\[
	K_\omega ({\bf u}) = \left.\frac{d}{d a} S_\omega (a {\bf u}) \right\rvert_{a=1}.
\]
Hence, if ${\bf u}$ is a critical point of $S_\omega$ then $K_\omega({\bf u})=0$.
This shows 
\begin{equation}\label{E:AK}
\mathcal{A}_\omega \subset \mathcal{K}_\omega.
\end{equation}

If $g_{\min}\ge0$ then $G\ge0$ holds and so $K_\omega ({\bf u})>0$ for ${\bf u} \neq0$.  
This implies $\mathcal{K}_\omega =\mathcal{A}_\omega= \emptyset$ for all $\omega>0$. 
Namely, \eqref{E:gE} does no have any nontrivial
$H^1$ solution.

Let us assume $g_{\min}<0$ in what follows.
Let us observe the following:
\begin{lemma}\label{L:1}
$\mathcal{K}_\omega$ is not empty.
\end{lemma}
\begin{proof}
Since $g_{\min}<0$, there exists ${\bf w} \in \partial B$ such that $g({\bf w})<0$.
Hence, one has
\[
	G({\bf w}Q) = \frac1p g({\bf w}) \int_{\R^d} Q(x)^p dx <0.
\]
Note that
${K}_\omega(c{\bf w} Q) = c^2 (\|\nabla Q\|_{L^2}^2 + \omega \|Q\|_{L^2}^2 ) + c^p pG({\bf w}Q)$. 
As $p>2$, there exists $c>0$ such that ${K}_\omega(c{\bf w} Q)=0$.
\end{proof}
We consider the following minimization problem:
\begin{equation}\label{D:domega}
	\mathfrak{I}(\omega) := \inf \{ S_\omega({\bf u} ) \ |\ 
{\bf u} \in \mathcal{K}_\omega  \}.
\end{equation}

\begin{proposition}\label{L:2}
For $\omega>0$, $\mathfrak{I}(\omega) \in (0,\infty)$.
\end{proposition}
\begin{proof}
The finiteness follows from Lemma \ref{L:1}.
Pick ${\bf u}=(u_1,\dots,u_N) \in H^1(\R^d)^N$.
Let us show that $\mathfrak{I}(\omega)$ is positive.
Let $\rho(x) = (\sum_{j=1}^N|u_j|^2)^{1/2}\ge0$.
If $\rho(x)>0$ then
\[
	-\tfrac1p g({\bf u}(x)) = - \tfrac1p \rho(x)^p g(\tfrac{{\bf u}(x)}{\rho(x)})
	\le - \tfrac{g_{\min}}p \rho(x)^p.
\]
Hence, using the fact that $p \in (2,2^*)$, we obtain
\[
	-G({\bf u}) \le  \tfrac{|g_{\min}|}p \|u_j\|_{L^p_x \ell^2_j}^p \lesssim_N M({\bf u})^{\frac{p}2(1-s_p)} H({\bf u})^{\frac{p}{2}s_p} \lesssim (H({\bf u}) + \omega M({\bf u}))^{\frac{p}2}.
\]
Therefore, there exists a constant $C>0$ such that
if ${\bf u}\in K_\omega$ then one has
\[
	H({\bf u}) + \omega M({\bf u}) = -\tfrac{p}2G({\bf u}) \le C (H({\bf u}) + \omega M({\bf u}))^{\frac{p}2}.
\]
Thus, $S_\omega({\bf u})=(1-\frac2p)(H({\bf u}) + \omega M({\bf u})) \gtrsim 1$ for any ${\bf u}\in K_\omega$.
We obtain the conclusion.
\end{proof}

There is another characterization of $\mathfrak{I}(\omega).$
\begin{lemma}\label{L:3}
It holds that
$\mathfrak{I}(\omega) 
	= (1-\tfrac2p)\inf \{ H ({\bf u}) +   \omega M ({\bf u}) \ |\ {\bf u} \neq0,\,
	K_\omega ({\bf u}) \le 0\}$.
\end{lemma}
\begin{proof}
Denote the right hand side $\tilde{\mathfrak{I}}(\omega)$.
Pick a nonzero ${\bf u} \in H^1(\R^d)^N$ so that
$K_\omega ({\bf u})=0$.
Then, it holds that
\[
S_\omega ({\bf u}) = (1-\tfrac2p)(H ({\bf u}) +   \omega M ({\bf u})) \ge \tilde{\mathfrak{I}}(\omega).
\]
Taking infimum with respect to such ${\bf u}$, we obtain $\mathfrak{I}(\omega) \ge \tilde{\mathfrak{I}}(\omega)$.
On the other hand, fix a nonzero ${\bf v} \in H^1(\R^d)^N$ so that
$K_\omega ({\bf v})\le 0$. Then, looking at the curve $\R_+ \ni c \mapsto K_\omega (c{\bf v}) \in \R$, one sees that there exists $c_0 \in (0,1]$ such that 
$K_\omega (c_0 {\bf v})=0$. It holds that
\[
	H ({\bf v}) +   \omega M ({\bf v}) \ge H (c_0 {\bf v}) +   \omega M (c_0 {\bf v})
	= \tfrac{p}{p-2}S_\omega (c_0 {\bf v}) \ge 
	\tfrac{p}{p-2} \mathfrak{I}(\omega).
\]
Taking infimum with respect to such ${\bf v}$, we obtain $\tilde{\mathfrak{I}}(\omega) \ge \mathfrak{I}(\omega)$.
Thus, the equality holds.
\end{proof}

\begin{lemma}\label{L:GK}
If ${\bf u} \in H^1(\R^d)^N $ satisfies $K_\omega ({\bf u})<0$ then
$-\frac{p-2}2 G({\bf u})>\mathfrak{I}(\omega).$
\end{lemma}
\begin{proof}
As in the previous lemma, consider the curve $\R_+ \ni c \mapsto K_\omega (c{\bf u}) \in \R$.
If $K_\omega ({\bf u})<0$ then
there exists $c_0 \in (0,1)$ such that 
$K_\omega (c_0{\bf u})=0$.
Since ${\bf u}\neq0$,
\[
	-G({\bf u}) > -G(c_0{\bf u}) = \tfrac2p(H(c_0{\bf u}) + \omega M(c_0{\bf u}))
	\ge \tfrac2p (1-\tfrac2p)^{-1} \mathfrak{I}(\omega)
\]
follows from the preceding lemma.
\end{proof}

Let  $\mathcal{M}_\omega$ be the set of minimizers to $\mathfrak{I}(\omega)$:
\begin{equation}\label{D:Momega}
	\mathcal{M}_\omega := \{ \Phi \in \mathcal{K}_\omega \ |\ 
S_\omega (\Phi) = \mathfrak{I}(\omega) \}.
\end{equation}
We next show that a minimizer exist, i.e., $\mathcal{M}_\omega\neq\emptyset$.
There is a direct approach to the existence based on Theorem \ref{T:excited} and the rearrangement argument used in the proof of Lemma \ref{L:domega}, below. However, we prove this by a standard compactness argument since we use it in the proof of the stability result (Theorem \ref{T:stability}).

\begin{lemma}\label{L:compactness}
$\mathcal{M}_\omega$ is not empty. Namely, there exists a minimizer to $\mathfrak{I}(\omega)$.
\end{lemma}
\begin{proof}
One can choose a minimizing sequence $\{{\bf u}_n\}_n =\{(u_{1,n}, u_{2,n},\dots,u_{N,n} )\}_n \subset (H^1(\R^d))^N $
such that
$K_\omega ({\bf u}_{n}) =0$ and
$
S_\omega ({\bf u}_{n}) \in [\mathfrak{I}(\omega),\mathfrak{I}(\omega) + \tfrac1n].
$
Notice that $K_\omega ({\bf u}_{n}) =0$ gives us
\[
	H ({\bf u}_{n}) + \omega M ({\bf u}_{n}) = \tfrac{p}{p-2} S_\omega ({\bf u}_{n})
	\le \tfrac{p}{p-2} (\mathfrak{I}(\omega)+\tfrac1n) \le \tfrac{p}{p-2} (\mathfrak{I}(\omega)+1).
\]
Hence, $\{{\bf u}_n\}_n$ is a bounded sequence in $(H^1 (\R^d) )^N$.

By extracting a subsequence if necessary, we suppose that the limits
$M_*:=\lim_{n\to\infty} M({\bf u}_{n})$,  $H_*:=\lim_{n\to\infty} H ({\bf u}_{n})$, and $G_*:=\lim_{n\to\infty} G({\bf u}_{n})$
exist. Then, 
\[
	0= \lim_{n\to\infty} K_\omega ({\bf u}_{n}) = 2H_* + 2 \omega M_*+ pG_*
\]
and
\[
	\mathfrak{I}(\omega) = H_* +  \omega M_*+ G_* = (1-\tfrac2p)(H_* +  \omega M_*)
	=(1-\tfrac{p}2)G_*>0
\]
hold.

We apply the profile decomposition to $\{{\bf u}_{n}\}_n$ with errors in $(L^p (\R^d))^N$.
Then, up to a subsequence, there exist $J_0 \in \mathbb{N}_0 \cup \{\infty\}$
and, for $1 \le j \le J_0$,
$\Phi_j=(\phi_{1,j}, \phi_{2,j}, \dots, \phi_{N,j}) \in H^1 (\R^d)^N$, $y_n^j \in \R^d$, and ${\bf r}_n^J=(r_{1,n}^J,r_{2,n}^J,\dots,r_{N,n}^J) \in H^1(\R^d)^N$ such that,
for each $J\le J_0$,
\[
	u_{k,n} = \sum_{j=1}^J 
	\phi_{k,j}(\cdot - {y_n^j}) + r_{k,n}^J \quad (k=1,2,\dots,N)
\]
for $n\ge1$. Further, for any $1 \le j_1 < j_2 \le J_0$,
\[
	\lim_{n\to\infty} |y_n^{j_1} - y_n^{j_2}| = \infty.
\]
One has
\[
	 (r_{1,n}^J , r_{2,n}^J , \dots, r_{N,n}^J )(\cdot + {y_n^j}) \rightharpoonup (0,0,\dots,0) \IN H^1(\R^d)^N
\]
as $n\to\infty$ for any $j \le J$ and
\[
	\lim_{J\to J_0} \varlimsup_{n\to\infty} \sum_{k=1}^N \|r_{k,n}^J\|_{L^p(\R^d)} =0.
\]
We have the decoupling inequalities
\[
	M_* \ge \sum_{j=1}^{J_0} M(\Phi_{j})
\quad
\text{and}
\quad
	H_* \ge \sum_{j=1}^{J_0} H(\Phi_{j}).
\]
Further,
\[
	\varlimsup_{n\to\infty} M({\bf r}_n^J) \le M_*, \quad \varlimsup_{n\to\infty} H({\bf r}_n^J) \le H_*
\]
for all $J\ge1$.
Let us claim
\begin{equation}\label{E:compactnesspf1}
	G_* = \sum_{j=1}^{J_0} G(\Phi_{j}).
\end{equation}
Indeed, for each fixed $J \le J_0$ finite, one sees from \eqref{E:fj-homo} that 
\[
	\left|G({\bf u}_{n}) - G\(\sum_{j=1}^J  \Phi_{j}(\cdot - {y_n^j})\)\right| = \left|-\int_0^1 \partial_\theta G({\bf u}_n- \theta {\bf r}_n^J) d\theta \right| \lesssim (\|{\bf u}_n\|_{L^p}+\|{\bf r}_n^J\|_{L^p})^{p-1} \|{\bf r}_n^J\|_{L^p}.
\]
Notice that the implicit constant is independent of $J$.
Further, by the mutual orthogonality of $\{y_n^j\}_n$, one finds
\[
	\lim_{n\to\infty} G\(\sum_{j=1}^J  \Phi_{j}(\cdot - {y_n^j})\) = \sum_{j=1}^{J} G(\Phi_{j}).
\]
This is justified, for instance, by approximating each $\phi_{k,j}$ by functions with a compact support in the $(L^p(\R^d))^N$-topology.
Hence, combining these two estimates and using the Gagliardo-Nirenberg inequality, one obtains
\[
	\varlimsup_{n\to\infty}
	\left|G({\bf u}_{n}) - \sum_{j=1}^{J} G(\Phi_{j})\right|
	\lesssim (M_* + H_*)^{\frac{p-1}2} \varlimsup_{n\to\infty} \|{\bf r}_n^J\|_{L^p}
	\to 0
\]
as $J\to J_0$.
This shows the claim.

If $J_0=0$ then $G_*=0$ follows from \eqref{E:compactnesspf1}. Hence, $\mathfrak{I}(\omega)=0$. This contradicts with $\mathfrak{I}(\omega)>0$.

If $J_0\ge2 $ then $0<H(\Phi_{j}) +\omega M(\Phi_{j}) < H_* + \omega M_* = \frac{p}{p-2}\mathfrak{I}(\omega) $ for all $j \in [1,J_0]$.
Hence, by means of Lemma \ref{L:3}, we have
$K_\omega (\Phi_j) >0$ for all $j \in [1,J_0]$.
Then,
\[
	0 = 2H_* + 2\omega M_* + pG_* \ge 2\sum_{j=1}^{J_0} H(\Phi_{j}) +2 \omega \sum_{j=1}^{J_0} M(\Phi_{j}) + p \sum_{j=1}^{J_0} G(\Phi_{j}) =\sum_{j=1}^{J_0} K_\omega (\Phi_j) >0,
\]
which is a contradiction.

Thus, we have $J_0=1$. Then,
in one hand, we have  $H(\Phi_{1}) + \omega M(\Phi_{1}) \le H_* + \omega M_*  $.
On the other hand, since
$0 = 2H_* + 2\omega M_* + pG_* \ge K_\omega(\Phi_{1})$,
we see from Lemma \ref{L:3} that
$H_* + \omega M_*= \frac{p}{p-2} \mathfrak{I}(\omega)  \le H(\Phi_{1}) + \omega M(\Phi_{1})$.
Thus, $H(\Phi_{1}) +\omega M(\Phi_{1}) = H_* + \omega M_*$.
This also shows $K_\omega(\Phi_{1}) = 2H_* + 2\omega M_* + pG_*  =0$
and $S_\omega (\Phi_1) = H_* + \omega M_* + G_*  = \mathfrak{I}(\omega)$.
Thus, $\Phi_{1}$ is a minimizer to $\mathfrak{I}(\omega)$. 
\end{proof}

\begin{lemma}\label{L:MequalG}
$\mathcal{M}_\omega = \mathcal{G}_\omega$. In particular, $\mathcal{M}_\omega \subset (C^2(\R^d) \cap (\cap_{2\le q < \infty} W^{3,q}(\R^d)))^N$.
\end{lemma}
\begin{proof}
Suppose that $\Phi \in \mathcal{M}_\omega$.
Since $\Phi \in \mathcal{M}_\omega$ minimizes $S_\omega$ under the constraint $K_\omega =0$, one sees from
Lagrange's multiplier theorem that there exists $\mu \in \R$ such that
$S_\omega'(\Phi) + \mu K'_\omega (\Phi) =0$. 
Then,
\[
	0=K_\omega (\Phi) =\left.\frac{d}{d a} S_\omega (a \Phi) \right\rvert_{a=1}
	= (S'_\omega (\Phi),\Phi)_{(L^2(\R^d))^N} = - \mu 
	(K'_\omega (\Phi),\Phi)_{(L^2(\R^d))^N}.
\]
Since 
\begin{align*}
	(K'_\omega (\Phi),\Phi)_{L^2(\R^d)^N} = {}&
	\left.\frac{d}{d a} K_\omega (a {\bf u}) \right\rvert_{a=1}\\
	={}& 4H(\Phi) + 4\omega M(\Phi) + p^2G(\Phi)\\
	={}&2(2-p)(H(\Phi) + \omega M(\Phi))>0,
\end{align*}
we see that $\mu=0$. Hence, $S'(\Phi)=0$. This implies $\Phi \in \mathcal{A}_\omega$.
Further, recalling \eqref{E:AK}, one obtains $S_\omega (\Phi) = \mathfrak{I}(\omega) \le S_\omega (\Psi)$
for all $\Psi \in \mathcal{A}_\omega$.
This implies that $\Psi \in \mathcal{G}_\omega$.
Thus $\mathcal{M}_\omega \subset \mathcal{G}_\omega$.

We remark that the above argument shows the identity
\[
	\inf_{\Psi \in \mathcal{A}_\omega} S_\omega (\Psi) = \mathfrak{I}(\omega).
\]
This immediately shows the other relation $\mathcal{G}_\omega \subset \mathcal{M}_\omega$.
Indeed, any $\Phi \in \mathcal{G}_\omega$ satisfies
\[
	S_\omega (\Phi) =\inf_{\Psi \in \mathcal{A}_\omega} S_\omega (\Psi) = \mathfrak{I}(\omega)
\]
and $\Phi \in \mathcal{A}_\omega \subset \mathcal{K}_\omega$.

The regularity property $\mathcal{A}_\omega \subset (C^2(\R^d) \cap (\cap_{2\le q < \infty} W^{3,q}(\R^d)))^N$
follows by a standard argument (See \cite{CazBook}*{Theorem 8.1.1}, for instance).
\end{proof}

The following lemma completes the proof of Theorem \ref{T:main}.
%
%
\begin{lemma}\label{L:domega}
$\mathcal{M}_\omega =  \bigcup_{{\bf w} \in T_0} \mathcal{R}({\bf w},\omega,-g_{\min})$.
Further, 
$\mathfrak{I}(\omega) = \tfrac1{2(1-s_c)} { \|Q\|_{L^2}^2 }{(-g_{\min})^{s_c-\frac{d}2} }  \omega^{1-s_c} $.
\end{lemma}

The proof is essentially the same as in \cite{Co}*{Theorem 1}. 
Here we reorganize the argument.

\begin{proof}
We divide the proof into three steps.

\smallskip

{\bf Step 1}.
We shall prove $\mathcal{M}_\omega \subset  \bigcup_{{\bf w} \in T_0} \mathcal{R}({\bf w},\omega,-g_{\min})$ in the first two steps.
Fix $\omega>0$.
Pick  ${\bf u}= (u_1,u_2,\dots,u_N) \in \mathcal{M}_\omega$. 
Let us first establish
\begin{equation}\label{E:shapepf21}
	\Im \sum_{j=1}^N\overline{u_j}\nabla u_j  = 0, \quad
	u_{j_1} \nabla u_{j_2} - u_{j_2} \nabla u_{j_1} = 0
\end{equation}
on $\R^d$ for all $j_1,j_2 \in[1,N]$.
Pick ${\bf w} =(w_1,w_2,\dots,w_N) \in T_0$ and
define $\tilde{\bf u}=(\tilde{u}_1, \tilde{u}_2, \dots , \tilde{u}_N) \in (H^1(\R^d))^N$ by
\[
	\tilde{u}_j(x) = w_j \rho(x) \quad (j=1,2,\dots,N), \quad \rho(x)=\(\sum_{j=1}^N |u_j(x)|^2 \)^{1/2}.
\]
Then, one has 
\[
	\sum_{j=1}^N |\tilde{u}_j(x)|^2  = \rho(x)^2 = \sum_{j=1}^N |u_j(x)|^2.
\]
This shows $M(\tilde{\bf u}) = M({\bf u})$.
Further, thanks to the homogeneity of $g$, one has for all $x\in \R^d$ such that $\rho(x)>0$,
\begin{equation}\label{E:shapepf215}
\begin{aligned}
	g({\bf u}(x) )  = \rho(x)^p  g \(\frac{{\bf u}(x)}{\rho(x)}  \)
	\ge{} \rho(x)^p g_{\min} 
	={} \rho(x)^p g ({\bf w}) 
	= g(\tilde{u}(x) ),
\end{aligned}
\end{equation}
where we have used the fact that ${\bf u}/\rho \in \partial B$ to obtain the inequality.
Hence, $G({\bf u}) \ge G (\tilde{\bf u})$.
Moreover,  one has
\begin{align*}
	|\nabla \rho| ={}&  \tfrac1{\rho} \left|\Re {\textstyle \sum_{j=1}^N} \overline{u_j}\nabla u_j \right|\\
	\le{}& \tfrac1{\rho} \left| {\textstyle \sum_{j=1}^N} \overline{u_j}\nabla u_j \right|
	=  \sqrt{ {\textstyle \sum_{j=1}^N}|\nabla u_j|^2 - {(2\rho^2)}^{-1} {\textstyle \sum_{j_1\neq j_2}}|u_{j_1} \nabla u_{j_2} - u_{j_2} \nabla u_{j_1}|^2}.
\end{align*}
Hence, one sees that $H(\tilde{\bf u}) \le H({\bf u})$ (See, e.g. \cite{LLBook}*{Theorem 7.8}, for rigorous treatment).
In view of the above inequality, $H(\tilde{\bf u}) = H({\bf u})$ gives us the desired conclusion \eqref{E:shapepf21}.

Combining the above estimates, we have $K_\omega(\tilde{\bf u}) \le K_\omega({\bf u})=0$ and  hence
\[
	\tfrac{p-2}{p} \mathfrak{I}(\omega) \le  H(\tilde{\bf u}) + \omega M(\tilde{\bf u})
	\le H({\bf u}) +  \omega M({\bf u}) = \tfrac{p-2}{p} \mathfrak{I}(\omega)
\]
by definition of Lemma \ref{L:3}.
Together with $M(\tilde{\bf u}) = M({\bf u})$, one obtains
$H(\tilde{\bf u})=H({\bf u})$ as desired. 

Moreover, this implies $\tilde{\bf u}\in \mathcal{M}_\omega$ and hence $K_\omega(\tilde{\bf u})=0$.
Therefore, $G({\bf u})=G(\tilde{\bf u})$.
Hence, by the virtue of \eqref{E:shapepf215}, we obtain 
\begin{equation}\label{E:shapepf2175}
	\tfrac{\bf u(x)}{\rho(x)} \in T_0
\end{equation}
as long as $\rho(x)>0$.
\smallskip

{\bf Step 2.} We next show that ${\bf u} \in \mathcal{R}({\bf w},\omega,-g_{\min})$
for some ${\bf w} \in T_0$.
For $j\in[1,N]$,
let $\Gamma_j$ be the set of all connected components of $\{ x \in \R^d \ |\ u_j(x) \neq0 \}$.
Note that $\Gamma_j$ is well-defined since $u_j \in C^2(\R^d)$.
Note that $\cup_{j=1}^N \Gamma_j$ is not empty since ${\bf u}\neq0$.

Pick $j_0 \in [1,N]$ such that $\Gamma_{j_0}\neq \emptyset$.
For simplicity, we assume $j_0=1$.
Pick $\Omega \in \Gamma_{1}$. 
Thanks to the latter identity of \eqref{E:shapepf21}, for each $j \in [2,N]$, there exists a constant $\Lambda_{j}\in \C $ such that $u_{j}=\Lambda_{j} u_{1}$ on $\Omega$ (See the proof of Theorem 7.8 of \cite{LLBook}, for instance).
Substituting this relation to the former identity of \eqref{E:shapepf21}, we obtain $\Im (\overline{u_{1}}\nabla u_{1})=0$ on $\Omega$, which implies that there exist $a\in \R$ and a nonnegative function $\phi(x)$ such that
$u_{1} = e^{ia} \phi$ on $\Omega$. Hence, we obtain
\begin{equation}\label{E:shapepf2}
	u_j =	w_{0,j} r_0 \phi
\end{equation}
for $j\in[1,N]$, where $r_0=(\sum_{j=1}^N|\Lambda_j|^2)^{1/2}$ and
$w_{0,j}:= e^{ia} r_0^{-1} \Lambda_j \in \C$ with the convention $\Lambda_1=1$.
Further, we denote ${\bf w}_0=(w_{0,1},w_{0,2},\dots,w_{0,N}) \in \partial B$.
By \eqref{E:shapepf2175}, ${\bf w}_0 \in T_0$.
By means of \eqref{E:fj-gauge}, we have
\[
	F_{1} ({\bf u}) = \phi^{p-1} r_0^{p-1} F_{1}({\bf w}_0)
\]
on $\Omega$. Since ${\bf u} \in A_\omega$, we have $-\Delta u_{1} + \omega u_{1} = - F_{1} ({\bf u})$. Substituting the above formula, we obtain
\[
	-\Delta \phi + \omega \phi =\lambda \phi^{p-1} ,\quad
	\lambda:=- e^{-ia} r_0^{p-1} F_{1}({\bf w}_0)
\]
on $\Omega$. Since $\phi$ is positive on $\Omega$,
one sees from the equation that $\lambda \in \R$.
Let $\psi$ be the zero extension of $\phi|_\Omega$ to $\R^d$.
Since $\phi$ is positive on $\Omega$ and vanishes on the boundary of $\Omega$,
$\psi$ satisfies  $-\Delta \psi + \omega \psi =\lambda \psi^{p-1}$ on $\R^d$.
Hence, $\psi=\lambda (-\Delta+\omega)^{-1} \psi^{p-1}$ holds.
Since $\psi\neq0$, one has $\lambda\neq0$.
Further,
since the Bessel potential, the integral kernel of $(-\Delta+\omega)^{-1} $, is positive everywhere but the origin and since $\psi\ge0$, we see that $(-\Delta+\omega)^{-1} \psi^{p-1}$ is positive everywhere.
This shows $\lambda>0$ and $\psi$ is also positive everywhere.
By definition of $\psi$, this implies that $\phi$ is also positive on $\R^d$, i.e., $\Omega=\R^d$.

Thus the expression \eqref{E:shapepf2} is valid on $\R^d$.
Since $\phi$ is a positive solution to
$	-\Delta \phi +\omega \phi =\lambda \phi^{p-1}$ ($\lambda>0$) on $\R^d$,
if $d\ge2$ then
it is radially symmetric up to space translation (see \cite{GNN}) and hence 
$\phi = Q_{\omega,\lambda}  (\cdot -y)$ for some $y \in \R^d$ (see \cite{Kwong}).
This shows 
\begin{equation}\label{E:shapepf25}
	u_j(x) =	w_{0,j} r_0  Q_{\omega,\lambda}  (x -y)
	= w_{0,j}	Q_{\omega, r_0^{-2}\lambda}  (x -y)
\end{equation}
Hence, we conclude that
\[
	{\bf u} \in  \mathcal{R}({\bf w}_0, \omega, r_0^{-2}\lambda) \subset \bigcup_{{\bf w} \in T_0} \mathcal{R}({\bf w},\omega, r_0^{-2}\lambda).
\]

To complete the proof of this step, let us prove that $r_0^{-2}\lambda = -g_{\min}$.
In one hand, by using ${\bf w}_0 \in \partial B$ and
the fact that $Q$ satisfies $\|\nabla Q\|_{L^2}^2 = \frac{s_p}{1-s_p}\|Q\|_{L^2}^2$,
one has
\[
	2(H({\bf u})+ \omega M({\bf u}) )
	=\tfrac{1}{(1-s_p)(r_0^{-2}\lambda)^{2/(p-2)} } \omega^{1-s_c} \|Q\|_{L^2}^2.
\]
On the other hand, noting that ${\bf w}_0 \in T_0$ and $\|Q\|_{L^p}^p = \frac{1}{1-s_p}\|Q\|_{L^2}^2$, one verifies that
\[
	pG({\bf u}) = g_{\min} \int_{\R^d} |Q_{\omega, r_0^{-2}\lambda}|^p dx 
	=\tfrac{g_{\min}}{r_0^{-2}\lambda}\tfrac{1}{(1-s_p)(r_0^{-2}\lambda)^{2/(p-2)} } \omega^{1-s_c} \|Q\|_{L^2}^2.
\]
Hence, we obtain the desired identity
from $K_\omega({\bf u})=0$.
\small

{\bf Step 3}. We complete the proof.
As $\mathcal{M}_\omega \subset \bigcup_{{\bf w} \in T_0} \mathcal{R}({\bf w},\omega,-g_{\min})$,
we see that 
\[
	\mathfrak{I}(\omega) =S_\omega ({\bf u})= (1-\tfrac2p)(H({\bf u})+\omega M({\bf u}))= 
	\tfrac{1 }{2(1-s_c)} \|Q\|_{L^2}^2 (-g_{\min})^{s_c-\frac{d}2}  \omega^{1-s_c}
\]
by substituting one element in ${\bf u} \in \mathcal{M}_\omega$.
One then sees that any element in $\bigcup_{{\bf w} \in T_0} \mathcal{R}({\bf w},\omega,-g_{\min}) $
gives the same value of $S_\omega $ and belongs to $\mathcal{K}_\omega$
and hence that any element in $\bigcup_{{\bf w} \in T_0} \mathcal{R}({\bf w},\omega,-g_{\min}) $
is a minimizer to $\mathfrak{I}(\omega)$.
Thus, we obtain the opposite inclusion relation
$\bigcup_{{\bf w} \in T_0} \mathcal{R}({\bf w},\omega,-g_{\min}) \subset \mathcal{M}_\omega
$.
\end{proof}

We finish this section with the proof of the sharp Gagliardo-Nirenberg-type inequality.
\begin{proof}[Proof of Theorem \ref{T:sGN}]
The inequality is obvious when ${\bf u}=0$. Hence, we let ${\bf u}\neq0$.
It suffices to prove that the best constant is attained by elements in $\mathcal{G}$.
Let us define
\[
	\tilde{C}_{\mathrm{GN}}:= \sup \left\{ \tfrac{-G({\bf u})}{M({\bf u})^{\frac{p(1-s_p)}2} H({\bf u})^{\frac{ps_p}2}}\ \middle| \ {\bf u} \in (H^1(\R^d))^N \setminus \{0\}\right\}.
\]
We will show that $\tilde{C}_{\mathrm{GN}} = C_{\mathrm{GN}}$.

Let us begin with the proof of
$\tilde{C}_{\mathrm{GN}} \le C_{\mathrm{GN}}$.
Arguing as in the proof of Lemma \ref{L:2}, one sees that $\tilde{C}_{\mathrm{GN}}$ is finite.
Further, mimicking the proof of Lemma \ref{L:1}, one finds that $\tilde{C}_{\mathrm{GN}}>0$.
Hence, for any $\varepsilon>0$ there exists a nonzero ${\bf v} \in (H^1(\R^d))^N$ such that
\[
	0< \tilde{C}_{\mathrm{GN}} -\varepsilon \le \tfrac{-G({\bf v})}{M({\bf v})^{\frac{p(1-s_p)}2} H({\bf v})^{\frac{ps_p}2}} = \tfrac{-G({\bf v})}{M({\bf v})^{\frac{2(1-s_c)}{d-2s_c}} H({\bf v})^{\frac{d}{d-2s_c}}}.
\]
Now, we take $c_0>0$ so that $H(c_0 {\bf v} )=-\frac{ d }{d-2s_c}G(c_0 {\bf v} ) $.  This is possible because the left hand  side is of the form $a c_0^2$ and the right hand side is of the form $b c_0^p$  for some constants $a,b>0$.
We further let $\tilde{\bf v}_\lambda = c_0 \lambda^{\frac{d}2-s_c} {\bf v} (\lambda \cdot)$ with $\lambda>0$ to be chosen later.
It follows that
\[
	M(\tilde{\bf v}_\lambda) = \lambda^{-2s_c} M(c_0{\bf v}),\quad
	H(\tilde{\bf v}_\lambda) = \lambda^{2(1-s_c)} H(c_0{\bf v}),\quad
	G(\tilde{\bf v}_\lambda) = \lambda^{2(1-s_c)} G(c_0{\bf v}).
\]
These imply that $H(\tilde{\bf v}_\lambda )=-\frac{d}{d-2s_c} G(\tilde{\bf v}_\lambda) $ for any $\lambda>0$.
Further, since
\[
	K_1 (\tilde{\bf v}_\lambda) = 2\lambda^{-2s_c} (-\tfrac{2(1-s_c)}{d}\lambda^{2} H(c_0{\bf v}) +  M(c_0{\bf v})),
\]
there exists unique $\lambda>0$ such that $K_1(\tilde{\bf v}_\lambda)=0$. 
We fix this $\lambda$.
Then, $H(\tilde{\bf v}_\lambda )=-\frac{d}{d-2s_c} G(\tilde{\bf v}_\lambda) $ and $K_1(\tilde{\bf v}_\lambda)=0$ give us
\[
	S_1 (\tilde{\bf v}_\lambda ) = \tfrac2d H(\tilde{\bf v}_\lambda )
	=\tfrac{1}{1-s_c}M(\tilde{\bf v}_\lambda ) =- \tfrac2{d-2s_c} G(\tilde{\bf v}_\lambda ).
\]
Further, one has $S_1 (\tilde{\bf v}_\lambda ) \ge \mathfrak{I}(1)$ by definition of $\mathfrak{I}(1)$.
Using these relations, one has
\begin{align*}
	\tilde{C}_{\mathrm{GN}} -\varepsilon \le{}&
		\tfrac{-G({\bf v})}{M({\bf v})^{\frac{2(1-s_c)}{d-2s_c}} H({\bf v})^{\frac{d}{d-2s_c}}} \\
		={}&	\tfrac{-G(c_0{\bf v})}{M(c_0{\bf v})^{\frac{2(1-s_c)}{d-2s_c}} H(c_0{\bf v})^{\frac{d}{d-2s_c}}}\\
		={}&	\tfrac{-G(\tilde{\bf v}_\lambda)}{M(\tilde{\bf v}_\lambda)^{\frac{2(1-s_c)}{d-2s_c}} H(\tilde{\bf v}_\lambda)^{\frac{d}{d-2s_c}}}\\
		={}& \tfrac{d-2s_c}2 (1-s_c)^{-\frac{2(1-s_c)}{d-2s_c}} (\tfrac{d}2)^{-\frac{d}{d-2s_c}}S_1(\tilde{\bf v}_\lambda)^{-\frac{2}{d-2s_c}}\\
		\le{}& \tfrac{d-2s_c}2 (1-s_c)^{-\frac{2(1-s_c)}{d-2s_c}} (\tfrac{d}2)^{-\frac{d}{d-2s_c}}\mathfrak{I}(1)^{-\frac{2}{d-2s_c}}\\
		={}&(\tfrac2{p-2})^{\frac{p}2} (\tfrac2d)^{\frac{d(p-2)}4} (d- \tfrac{d-2}{2}p)^{\frac{p-2}2}  \|Q\|_2^{2-p}
		(-g_{\min})=C_{\mathrm{GN}}.
\end{align*}
Since $\varepsilon>0$ is arbitrary, we obtain $\tilde{C}_{\mathrm{GN}} \le C_{\mathrm{GN}}$.

Let us prove the other inequality.
Pick $\omega>0$ and $\Phi \in \mathcal{G}_\omega$.
Since $\Phi $ attains $\mathfrak{I}(\omega)$, one has
\begin{align*}
	C_{\mathrm{GN}}
	={}& \tfrac{d-2s_c}2 (1-s_c)^{-\frac{2(1-s_c)}{d-2s_c}} (\tfrac{d}2)^{-\frac{d}{d-2s_c}}
	\mathfrak{I}(\omega)^{-\frac{2}{d-2s_c}}\omega^{\frac{2(1-s_c)}{d-2s_c}} \\
	={}& \tfrac{d-2s_c}2 (1-s_c)^{-\frac{2(1-s_c)}{d-2s_c}} (\tfrac{d}2)^{-\frac{d}{d-2s_c}}
	S_\omega(\Phi)^{-\frac{2}{d-2s_c}} \omega^{\frac{2(1-s_c)}{d-2s_c}}.
\end{align*}
Note that $H(\Phi )=-\frac{d}{d-2s_c}G(\Phi) $ by Pohozaev's identity.
Further, $K_\omega(\Phi)=0$.
Hence, one has
\[
	S_\omega (\Phi) = \tfrac2d H(\Phi)
	=\tfrac{1}{1-s_c} \omega M(\Phi ) =- \tfrac2{d-2s_c} G(\Phi)
\]
and so
\begin{align*}
\tfrac{d-2s_c}2 (1-s_c)^{-\frac{2(1-s_c)}{d-2s_c}} (\tfrac{d}2)^{-\frac{d}{d-2s_c}}
	S_\omega(\Phi)^{-\frac{p-2}{2}} \omega^{\frac{2(1-s_c)}{d-2s_c}}
	{}& = \tfrac{(\frac{d-2s_c}{2}S_\omega(\Phi))}{(\frac{1-s_c}{\omega}S_\omega(\Phi))^{\frac{2(1-s_c)}{d-2s_c}} (\frac{d}2 S_\omega(\Phi))^{\frac{d}{d-2s_c}}}\\
	{}& =\tfrac{-G(\Phi)}{M(\Phi)^{\frac{2(1-s_c)}{d-2s_c}} H(\Phi)^{\frac{d}{d-2s_c}}} 
	 \le \tilde{C}_{\mathrm{GN}}.
\end{align*}
Hence, one obtains $C_{\mathrm{GN}}\le \tilde{C}_{\mathrm{GN}}$.

Thus, we obtain $C_{\mathrm{GN}}= \tilde{C}_{\mathrm{GN}}$.
The argument in the proof of $C_{\mathrm{GN}}\le \tilde{C}_{\mathrm{GN}}$ also shows that
any element in $\mathcal{G}$ is an optimizer to $\tilde{C}_{\mathrm{GN}}$.
Further, mimicking the proof of $\tilde{C}_{\mathrm{GN}} \le C_{\mathrm{GN}}$, one deduces that for any minimizer $\Psi$ to $\tilde{C}_{\mathrm{GN}}$ there exists $c_0>0$ and $\lambda$ such that $\tilde{\Psi} = c_0 \lambda^{\frac{d}2-s_c} {\Psi} (\lambda \cdot)$
satisfies
\[
	K_1 (\tilde{\Psi} )=0, \quad K_1 (\tilde{\Psi} ) = \mathfrak{I}(1).
\]
This implies that $\tilde{\Psi}$ attains $\mathfrak{I}(1)$ and hence $\tilde{\Psi}\in \mathcal{G}_1$ thanks to Lemma \ref{L:MequalG}.
By the explicit formula of $\mathcal{G}_\omega$, one sees that $c_0\Psi \in \mathcal{G}_{\lambda^{-2}} \subset \mathcal{G}$.
\end{proof}

\section{Proof of Theorem \ref{T:excited}}\label{S:excited}

In this section, we prove Theorem \ref{T:excited}.

\begin{proof}[Proof of Theorem \ref{T:excited}]
We first prove the if part.
Let us first claim that the identities
\begin{equation}\label{E:excitepf1}
	(\partial_{z_j}g) ({\bf w}) = \tfrac{p}2\overline{w_j}g({\bf w}),\quad
	(\partial_{\overline{z_j}}g) ({\bf w}) = \tfrac{p}2 {w_j}g({\bf w})
\end{equation}
hold for $j=1,\dots,N$ at any critical point ${\bf w} = (w_1,\dots, w_2) \in \partial B$ of $g|_{\partial B}$.
One can regard
 ${\bf w} \in \partial B$ as a critical point of $g:\C^N \to \R$ under the constraint
$\sum_{j=1}^N|z_j|^2-1=0$.
Hence, by Lagrange's multiplier theorem, there exists a constant $\lambda$ such that
\begin{align*}
	\partial_{z_j} \(g({\bf z})+\lambda \({\textstyle\sum_{j=1}^N}|z_j|^2 -1\)\middle)\right|_{{\bf z}={\bf w}} &=0,&
	\partial_{\overline{z_j}} \(g({\bf z})+\lambda \({\textstyle\sum_{j=1}^N}|z_j|^2 -1\)\middle)\right|_{{\bf z}={\bf w}}& =0,
\end{align*}
for $j=1,\dots,N$. These read as
\[
	(\partial_{z_j}g) ({\bf w}) = -\lambda \overline{w_j},\quad
	(\partial_{\overline{z_j}}g) ({\bf w}) = -\lambda {w_j}
\]
for $j=1,\dots,N$. 
Hence,
\[
	\sum_{j=1}^N (w_j(\partial_{z_j}g) ({\bf w}) +\overline{w_j} (\partial_{\overline{z_j}}g) ({\bf w}))
	= -2\lambda \sum_{j=1}^N|w_j|^2  = -2\lambda.
\]
On the other hand,
\begin{align*}
		\sum_{j=1}^N (w_j(\partial_{z_j}g) ({\bf w}) +\overline{w_j} (\partial_{\overline{z_j}}g) ({\bf w}))
		=\left. \frac{d}{d h} g(h{\bf w}) \right|_{h=1} 
		=g({\bf w})\left. \frac{d}{d h} h^p \right|_{h=1} 
	= p g({\bf w}).	
\end{align*}
Hence, $\lambda = -\tfrac{p}2 g({\bf w})$ follows. The claim is shown.

Now, let us check that $\mathcal{R}({\bf w} ,\omega, -g({\bf w})) \subset \mathcal{A}_\omega$.
For $j=1,\dots,N$, we have
\[
		-\Delta (w_j Q_{\omega,-g({\bf w})}) + \omega w_j Q_{\omega, -g({\bf w})}
		= - g({\bf w}) w_j  Q_{\omega,-g({\bf w})}^{p-1}.
\]
By the above claim and the relation between the function $g$ and the nonlinearity $F_j$, we have
\[
	- g({\bf w}) w_j =- \tfrac2p (\partial_{\overline{z_j}}g) ({\bf w})
	=-F_j({\bf w}).
\]
Therefore, using the homogeneity of $F_j$, one has
\[
	 - g({\bf w}) w_j  Q_{\omega, -g({\bf w})}^{p-1}
	 = - F_j({\bf w})  Q_{\omega, -g({\bf w})}^{p-1}
	 =-F_j({\bf w} Q_{\omega, -g({\bf w})})
\]
for $j=1,\dots, N$.
Hence, ${\bf w} Q_{\omega,-g({\bf w})}$ solves \eqref{E:gE}. The if part is established.

Let us proceed to the only-if part.
Suppose that $\mathcal{R}({\bf w} ,\omega, a) \subset \mathcal{A}_\omega$.
Then, the latter half of the proof of the if part shows that the identity 
$ a w_j Q_{\omega, a}^{p-1} =-F_j({\bf w}) Q_{\omega, a}^{p-1} $ for all $j$, which implies that
$(\partial_{\overline{z_j}} g)({\bf w}) = - \frac{p}2 a w_j$ for all $j$.
Then, by mimicking the argument in the proof of the if part,
we see that \eqref{E:excitepf1} is valid.
In particular, $a=-g({\bf w})>0$.

Let us prove that ${\bf w}$ is a critical point of $g|_{\partial B}$.
One has $\partial_{x_j} g ({\bf w})= p (\Re w_j) g({\bf w})$ and $\partial_{y_j} g ({\bf w})= p (\Im w_j) g({\bf w})$.
Since ${\bf w}\neq0$, we may suppose that  $\Re w_{1}>0$ without loss of generality. 
Then, the real part of the first component of ${\bf z} \in \partial B \subset \C^N = \R^{2N}$ is given by
$x_1 =x_1(y_1,x_2,y_2,\dots,x_N,y_N):= (1-y_1^2 -\sum_{j=2}^N (x_j^2 +y_j^2))^{1/2}$ around ${\bf z}={\bf w}$. Then, it is easy to see that the point $(\Im w_1, \Re w_2, \Im w_2, \dots, \Re w_N, \Im w_N)\in \R^{2N-1}$ is a critical point of
the function
\[
	h(y_1,x_2,y_2,\dots,x_N,y_N) := g(x_1(y_1,x_2,y_2,\dots,x_N,y_N),y_1,x_2,y_2,\dots,x_N,y_N).
\]
This implies that ${\bf z}={\bf w}$ is a critical point of $g|_{\partial B}$.
\end{proof}

\begin{remark}\label{R:morees}
If $g$ satisfies $g(-{\bf z})=g({\bf z})$ for all ${\bf z} \in \C^N$ then one has
$F_j (-{\bf z}) = - F_j ({\bf z})$
for all ${\bf z} \in \C^N$. This implies that $F_j (r{\bf z})=|r|^{p-2}rF_j ({\bf z})$ holds for any $r\in \R$ and ${\bf z}\in \C^N$.
In this case, the above proof works even if we replace the positive solution $Q$ to $-\Delta Q + Q = Q^{p-1}$
with a real-valued solution to
$-\Delta \tilde{Q} + \tilde{Q} = |\tilde{Q}|^{p-2}\tilde{Q}$.
It is known that if $d\ge2$ then there exist infinitely many
 sign-changing solutions to the elliptic equation (See \cites{BeLi2,BGK} and references therein).
\end{remark}

\section{Proof of Theorem \ref{T:stability}}\label{S:stability}

We follow the argument by Shatah \cite{Sh} (See also \cite{O}).
For $\omega>0$, let
\[
	\mathcal{P}^\pm_\omega := \{ {\bf u} \in (H^1(\R^d))^N \ |\ S_\omega ({\bf u})<\mathfrak{I}(\omega),\, \pm (-\tfrac{p-2}2 G({\bf u}))> \mathfrak{I}(\omega)  \}.
\]
\begin{lemma}\label{L:Ppm}
The sets $\mathcal{P}^\pm_\omega$ are invariant under the \eqref{E:gNLS}-flow.
\end{lemma}
\begin{proof}
Let us only consider $\mathcal{P}^+_\omega$.
Fix ${\bf u}_0 \in  \mathcal{P}^+_\omega$ and let ${\bf u}(t)$ be the corresponding solution to \eqref{E:gNLS}.
Suppose for contradiction that ${\bf u} ( t) \not \in \mathcal{P}^+_\omega$ for some $t \neq0$.
Since $S_\omega ({\bf u}(t)) = S_\omega ({\bf u}_0)<\mathfrak{I}(\omega)$  ($\forall t\in I_{\max}$) follows from the conservation laws,
there exists $t_0 $ such that
$-\tfrac{p-2}2 G({\bf u}(t_0))= \mathfrak{I}(\omega) $.
Then, we see from Lemma \ref{L:GK} that $K_\omega ({\bf u}(t_0)) \ge 0$.
Hence,
\[
	S_\omega ({\bf u}(t_0)) = \tfrac12 K_\omega({\bf u}(t_0)) -\tfrac{p-2}2 G({\bf u}(t_0)) \ge \mathfrak{I}(\omega),
\]
which contradicts with $S_\omega ({\bf u}(t_0)) = S_\omega ({\bf u}_0)<\mathfrak{I}(\omega)$.
\end{proof}
Recall that the explicit form of $\mathfrak{I}(\omega)$ is given in Lemma \ref{L:domega}.
We remark that the mass-subcritical assumption $p<2+\frac4d$, which reads as $s_c<0$, assures that
$\mathfrak{I}''(\omega)>0$ for $\omega>0$.
\begin{lemma}\label{L:stability_key}
For each $\omega_0>0$ there exists $\varepsilon_0=\varepsilon_0(\omega_0)>0$ such that the following property holds true:
For any $\varepsilon \in (0,\varepsilon_0)$ there exists $\delta>0$ such that if ${\bf u} \in (H^1(\R^d))^N$ satisfies
$\| {\bf u}_0 - \Phi \|_{(H^1(\R^d))^N}<\delta$ for some $\Phi \in \mathcal{G}_{\omega_0}$ then 
${\bf u} \in \mathcal{P}_{\omega_0-\varepsilon}^+ \cap \mathcal{P}_{\omega_0+\varepsilon}^-$.
\end{lemma}

\begin{proof}
Since $\mathfrak{I}(\omega) $ is of the form $\mathfrak{I}(\omega) = c \omega^{q} $ with $c>0$ and $q>2$, we have
\[
	\mathfrak{I}(\omega_0 - \varepsilon) < \mathfrak{I}(\omega_0) < \mathfrak{I}(\omega_0 + \varepsilon)
\]
for any $\varepsilon \in (0, \mathfrak{I}(\omega_0)).$
Further, since $\Phi \in \mathcal{G}_{\omega_0} \subset \mathcal{K}_{\omega_0}$, one has
\[
\mathfrak{I}(\omega_0) = S_{\omega_0}(\Phi) = - \tfrac{p-2}{2} G(\Phi).
\]
Thus, for each $\varepsilon>0$ there exists $\delta>0$ such that $\| {\bf u} - \Phi \|_{(H^1(\R^d))^N} \le \delta$
implies that 
\[
	\mathfrak{I}(\omega_0 - \varepsilon) < - \tfrac{p-2}{2} G({\bf u}) < \mathfrak{I}(\omega_0 + \varepsilon).
\]

Now, let us claim that  $S_{\omega_0\pm 2\varepsilon} ({\bf u})< \mathfrak{I}(\omega_0 \pm \varepsilon)$.
A computation shows that $\mathfrak{I}'(\omega_0)=M(\Phi)$.
Hence,
\[
	S_{\omega_0\pm \varepsilon} (\Phi) = S_{\omega_0} (\Phi) \pm \varepsilon M(\Phi)
	= \mathfrak{I}(\omega_0)  \pm \varepsilon \mathfrak{I}'(\omega_0).
\]
By the Taylor expansion, there exists $\omega_\pm \in (\omega_0-\varepsilon, \omega_0+\varepsilon)$ such that
\[
	d (\omega_0 \pm \varepsilon) = \mathfrak{I}(\omega_0) \pm \varepsilon \mathfrak{I}'(\omega_0) + \varepsilon^2 \tfrac12 \mathfrak{I}''(\omega_\pm).
\]
Due to the continuity of $\mathfrak{I}''$ and the positivity  $\mathfrak{I}''>0$,
by letting $\varepsilon_0$ smaller if necessary, we have $\inf_{\omega \in [\omega_0-\varepsilon_0, \omega_0+\varepsilon_0]}  \mathfrak{I}''(\omega) \ge \tfrac12 \mathfrak{I}''(\omega_0)$.
Combining these estimates, we have
\[
	\mathfrak{I} (\omega_0 \pm \varepsilon) \ge S_{\omega_0\pm \varepsilon} (\Phi) + \tfrac{\varepsilon^2}{4} \mathfrak{I}''(\omega_0).
\]
Hence, if $\delta$ is chosen so small that
\[
	|S_{\omega_0\pm \varepsilon} (\Phi)-S_{\omega_0\pm \varepsilon} ({\bf u})| \le \tfrac{\varepsilon^2}8 \mathfrak{I}''(\omega_0)
\]
then we obtain the desired conclusion.
We remark that $\delta$ does not depend on the specific choice of $\Phi \in \mathcal{G}_{\omega_0}$ but only on $\omega_0$. This is due to the fact that $H(\Phi)$ and $M(\Phi)$ depend only on $\omega_0$.
\end{proof}

Now, let us complete the proof of the theorem.

\begin{proof}[Proof of Theorem \ref{T:stability}]
Suppose that the result fails. 
Then, there exist $\omega_0>0$ and
$\varepsilon_0>0$ such that for each $m$ there exist
${\bf u}_{0,m} \in (H^1(\R^d))^N$ and $\Phi_m \in \mathcal{G}_{\omega_0}$
such that
\begin{equation}\label{E:stpf01}
	 \| {\bf u}_{0,m} - \Phi_m \|_{(H^1(\R^d))^N} \le \tfrac1m
\end{equation}
and
\begin{equation}\label{E:stpf02}
	\sup_{t\in \R}  \inf_{\Phi \in \mathcal{G}_{\omega_0}} \| {\bf u}_{m}(t) - \Phi \|_{(H^1(\R^d))^N} \ge \varepsilon_0,
\end{equation}
where ${\bf u}_m(t)$ is a global solution to \eqref{E:gNLS} such that ${\bf u}_m(0) = {\bf u}_{0,m}$.
We let $t_m \in \R$ be the time such that
\[
	\inf_{\Phi \in \mathcal{G}_{\omega_0}} \| {\bf u}_{m}(t_m) - \Phi \|_{(H^1(\R^d))^N} = \tfrac{\varepsilon_0}2.
\]
By Lemma \ref{L:stability_key} and \eqref{E:stpf01}, there exists a series $\{\varepsilon_n\}_n\subset \R_+$, $\varepsilon_n\to0$, such that ${\bf u}_{0,n} \in \mathcal{P}^+_{\omega_0-\varepsilon_n} \cap \mathcal{P}^-_{\omega_0+\varepsilon_n}$ holds for large $n$. 
Then, one sees from Lemma \ref{L:Ppm} that ${\bf u}_{n}(t_n) \in \mathcal{P}^+_{\omega_0-\varepsilon_n} \cap \mathcal{P}^-_{\omega_0+\varepsilon_n}$ for large $n$ and so that
\begin{equation}\label{E:stpf1}
	-\tfrac{p-2}{2} G({\bf u}_n(t_n)) \to \mathfrak{I}(\omega_0)
\end{equation}
as $n\to\infty$.
Further, by the conservation laws of \eqref{E:gNLS} and \eqref{E:stpf01},
\begin{equation}\label{E:stpf2}
	|S_{\omega_0} ({\bf u}_{n}(t_n)) - \mathfrak{I}(\omega_0)|=| S_{\omega_0} ({\bf u}_{0,n})- S_{\omega_0} (\Phi_n) |
	\to 0 
\end{equation}
as $n\to\infty$. \eqref{E:stpf1} and \eqref{E:stpf2} imply that
\begin{equation}\label{E:stpf3}
	K_{\omega_0} ({\bf u}_{n}(t_n)) \to 0
\end{equation}
as $n\to\infty$.
Let $c_n >0$ be the constant such that  
\begin{equation}\label{E:stpf4}
K_{\omega_0} (c_n {\bf u}_{n}(t_n))=0. 
\end{equation}
It follows from \eqref{E:stpf3} that $c_n \to 1$ as $n\to\infty$.
Hence, together with \eqref{E:stpf2}, one finds
\begin{equation}\label{E:stpf5}
	S_{\omega_0} (c_n {\bf u}_{n}(t_n)) \to \mathfrak{I}(\omega_0).
\end{equation}
Now, due to \eqref{E:stpf4} and \eqref{E:stpf5}, we deduce that $\{ c_n {\bf u}_n(t_n)\}_n$ is a minimizing sequence
of $\mathfrak{I}(\omega_0)$.
Then, mimicking the proof of Lemma \ref{L:compactness}, one sees that, up to a subsequence, there exist $\Phi_\infty \in \mathcal{M}_{\omega_0} = \mathcal{G}_{\omega_0}$ and $y_n\in \R^d$ such that
$
	c_n {\bf u}_n(t_n) = \Phi_\infty(\cdot - y_n) +o(1)
$
strongly in $(H^1(\R^d))^N$  as $n\to \infty$.
This contradicts with \eqref{E:stpf02}.
\end{proof}

\section{The Potential-well structure in the mass-supercritical case}\label{S:Pstruct}

In this section, we establish another variational characterization of  the ground states in the mass-supercritical
case $p>2+4/d$.
The characterization is useful in the proof of the instability.
We do not need Assumption \ref{A:2} nor any restriction on $d$ and $p$ in this section.
Recall that $s_p:=\frac{d}2-\frac{d}p$ and $s_c:=\frac{d}2-\frac{2}{p-2}$. Note that $p>2+4/d$ corresponds to $s_c>0$.

We introduce a functional
\[
	V({\bf u}) = 2 H({\bf u}) + \tfrac{2d}{d-2s_c} G({\bf u}),
\]
which is the $L^2$-scaling derivative of $E({\bf u})$ or of $S_\omega ({\bf u})$.
Further, we introduce
\[
	\mathfrak{I}_2(\omega) := \inf \{ \tfrac{2s_c}{d}H({\bf u}) + \omega M({\bf u}) \ |\ {\bf u}\neq0, \,  V({\bf u}) \le 0\}
\]
for $\omega>0$
and
\[
	\mathfrak{I}_3 := \inf \{ H({\bf u}) M({\bf u})^{\frac{1-s_c}{s_c}}  \ |\ {\bf u} \neq0,\, V({\bf u})\le0 \}.
\]
We remark that if ${\bf u}\neq0$ satisfies $G({\bf u})<0$ then $V(c {\bf u})<0$ for large $c>0$.
Further, the assumption $p>2+\frac4d$ assures $s_c>0$.
Hence, $\mathfrak{I}_2(\omega)$ and $\mathfrak{I}_3$ are well-defined and are positive and finite.

The main result of this section is the following.
\begin{theorem}[Potential well structure]\label{T:Pstruct}
Suppose that $p>2+\frac4d$. 
Let $\Phi \in \mathcal{G}$.
For $\delta\in(0,1)$ there exists $\tilde{\delta}=\tilde{\delta}(d,p,\delta)>0$ such that
if  ${\bf u} \in (H^1(\R^d))^N$ is nonzero and satisfies
\[
	E({\bf u}) M({\bf u})^{\frac{1-s_c}{s_c}} \le (1-\delta)E(\Phi) M(\Phi)^{\frac{1-s_c}{s_c}}
\]
then the followings are true:
\begin{enumerate}
\item $V({\bf u})\neq0$;
\item 
if $V({\bf u})>0$ then 
	$H({\bf u}) M({\bf u})^{\frac{1-s_c}{s_c}} <  \mathfrak{I}_3$ and $
	V({\bf u}) \ge \tilde{\delta} H({\bf u})$;
\item if  $V({\bf u})<0$ then
$H({\bf u}) M({\bf u})^{\frac{1-s_c}{s_c}} >  \mathfrak{I}_3$ and
	$V({\bf u}) M({\bf u})^{\frac{1-s_c}{s_c}} \le -\tilde{\delta} \mathfrak{I}_3.$
\end{enumerate}
\end{theorem}

\begin{remark}
Let $\Phi \in \mathcal{G}$.
A computation shows that
\begin{equation*}
E(\Phi) M(\Phi)^{\frac{1-s_c}{s_c}}
=	  \tfrac{s_c}{1-s_c} (\tfrac12(-g_{\min})^{s_c-\frac{d}2} \|Q\|_{L^2}^2)^{\frac{1}{s_c}}.
\end{equation*}
In particular, this quantity is independent of the choice of $\Phi$.
Further, it will turn out that
\[
	\mathfrak{I}_3=H(\Phi) M(\Phi)^{\frac{1-s_c}{s_c}}= \tfrac{d}{2(1-s_c)} (\tfrac12(-g_{\min})^{s_c-\frac{d}2} \|Q\|_{L^2}^2)^{\frac{1}{s_c}}.
\]
\end{remark}

To prove the theorem, let us begin with the study of $\mathfrak{I}_2(\omega)$.
\begin{lemma}
For $\omega>0$, $\mathfrak{I}_2(\omega) = \mathfrak{I}(\omega)$. Further, the minimizer to $\mathfrak{I}_2(\omega)$ is $\mathcal{G}_\omega$.
\end{lemma}
\begin{proof}
Let us first claim that
\begin{equation}\label{E:d2pf1}
	\mathfrak{I}_2(\omega) = \inf \{ S_\omega ({\bf u}) \ |\ {\bf u}\neq0, \,  V({\bf u}) = 0\}.
\end{equation}
Note that $\tfrac{2s_c}{d}H({\bf u}) + \omega M({\bf u})=S_\omega ({\bf u}) - \frac{d-2s_c}{2d}V ({\bf u})$.
Hence,
\begin{align*}
	\mathfrak{I}_2(\omega) ={}& \inf \{ S_\omega ({\bf u})- \tfrac{d-2s_c}{2d}V ({\bf u}) \ |\ {\bf u}\neq0, \,  V({\bf u}) \le 0\}\\
	\le {}& \inf \{ S_\omega ({\bf u})- \tfrac{d-2s_c}{2d}V ({\bf u}) \ |\ {\bf u}\neq0, \,  V({\bf u}) = 0\}
	= \inf \{ S_\omega ({\bf u}) \ |\ {\bf u}\neq0, \,  V({\bf u}) = 0\}.
\end{align*}
We prove the opposite inequality.
Pick ${\bf u}\neq0$ such that $V({\bf u}) \le 0$. There exists $c_0 \in (0,1]$ such that $V(c_0 {\bf u})=0$. Hence,
\[
	\tfrac{2s_c}{d} H({\bf u}) + \omega M({\bf u}) \ge 
	\tfrac{2s_c}{d} H(c_0{\bf u}) + \omega M(c_0{\bf u}) = S_\omega(c_0{\bf u}) \ge 
	\inf \{ S_\omega ({\bf u}) \ |\ {\bf u}\neq0, \,  V({\bf u}) = 0\}.
\]
Taking the infimum with respect to ${\bf u}$, we obtain the desired inequality.

We next claim
\begin{equation}\label{E:d2pf2}
	\inf \{ S_\omega ({\bf u}) \ |\ {\bf u}\neq0, \,  V({\bf u}) = 0\}
	=\inf \{ S_\omega ({\bf u}) \ |\ {\bf u}\neq0, \,  V({\bf u}) = 0,\, H({\bf u}) = \tfrac{d}{2(1-s_c)}\omega M({\bf u})\}.
\end{equation}
The inequality ``$\le$'' is obvious by definition. We prove the opposite inequality.
Pick ${\bf u}\neq0$ such that $V({\bf u}) = 0$. Define ${\bf u}_\lambda := \lambda^{\frac{d}2-s_c} {\bf u}(\lambda \cdot)$.
Then, one sees that $V({\bf u}_\lambda) = 0$ for any $\lambda>0$. Further,
\[
	S_\omega ({\bf u}_\lambda) = \lambda^{2(1-s_c)}  \tfrac{2s_c}{d}H({\bf u})  + \lambda^{-2s_c} \omega M({\bf u}).
\]
Let $\lambda_0>0$ be the minimum point of the function $\lambda \mapsto S_\omega ({\bf u}_\lambda)$.
A computation shows that $H({\bf u}_{\lambda_0}) = \tfrac{d}{2(1-s_c)}\omega M({\bf u}_{\lambda_0})$.
Hence,
\[
	S_\omega ({\bf u}) \ge S_\omega ({\bf u}_{\lambda_0}) \ge \inf \{ S_\omega ({\bf u}) \ |\ {\bf u}\neq0, \,  V({\bf u}) = 0,\, H({\bf u}) = \tfrac{d}{2(1-s_c)}\omega M({\bf u})\}.
\]
Taking the infimum with respect to ${\bf u}$, we obtain the desired inequality.

Let us now prove 
\begin{equation}\label{E:d2pf3}
	\inf \{ S_\omega ({\bf u}) \ |\ {\bf u}\neq0, \,  V({\bf u}) = 0,\, H({\bf u}) = \tfrac{d}{2(1-s_c)}\omega M({\bf u})\}
	= \mathfrak{I}(\omega).
\end{equation}
We first note that the constraint $H({\bf u}) = \tfrac{d}{2(1-s_c)}\omega M({\bf u})$ implies $K_\omega ({\bf u})=\frac1{s_p}V({\bf u})$. Hence,
\begin{multline*}
	\inf \{ S_\omega ({\bf u}) \ |\ {\bf u}\neq0, \,  V({\bf u}) = 0,\, H({\bf u}) = \tfrac{d}{2(1-s_c)}\omega M({\bf u})\}\\
	=\inf \{ S_\omega ({\bf u}) \ |\ {\bf u}\neq0, \,  K_\omega({\bf u}) = 0,\, H({\bf u}) = \tfrac{d}{2(1-s_c)}\omega M({\bf u})\}
	\ge \mathfrak{I}(\omega).
\end{multline*}
Hence, the inequality ``$\ge$'' holds.
To see the other inequality, it suffices to see that a minimizer $\Phi \in \mathcal{G}_\omega$ to $\mathfrak{I}(\omega)$
satisfies $H(\Phi)= \tfrac{d}{2(1-s_c)}\omega M(\Phi)$.

Combining \eqref{E:d2pf1}, \eqref{E:d2pf2}, and \eqref{E:d2pf3}, we obtain the desired identity $\mathfrak{I}_2(\omega)=\mathfrak{I}(\omega)$.
The identity shows that any element in $\mathcal{G}_\omega$ is a minimizer of $\mathfrak{I}_2(\omega)$.
On the other hand, we see from the proofs of \eqref{E:d2pf1}, \eqref{E:d2pf2}, and \eqref{E:d2pf3} that
a minimizer ${\bf v}$ of $\mathfrak{I}_2(\omega)$ satisfies $V({\bf v}) = K_\omega({\bf v})=0$ and $\mathfrak{I}(\omega)= S_{\omega}({\bf v})$. Hence, ${\bf v} \in \mathcal{M}_\omega = \mathcal{G}_\omega$.
\end{proof}
Now, we turn to the study of $\mathfrak{I}_3$.
\begin{lemma}
$\mathfrak{I}_3=\tfrac{d}{2(1-s_c)} (\tfrac12(-g_{\min})^{s_c-\frac{d}{2}} \|Q\|_{L^2}^2)^{\frac{1}{s_c}}$. The set of the minimizers to $\mathfrak{I}_3$ is $\mathcal{G}$.
\end{lemma}
\begin{proof}
By substituting an element $\Phi \in \mathcal{G}$, we obtain
\[
	\mathfrak{I}_3 \le H(\Phi) M(\Phi)^{\frac{1-s_c}{s_c}}
	= \tfrac{d}{2(1-s_c)} (\tfrac12(-g_{\min})^{s_c-\frac{d}{2}} \|Q\|_{L^2}^2)^{\frac{1}{s_c}}.
\]
Let us prove the opposite inequality.
For any $\varepsilon>0$, there exists $\Psi \neq 0$ such that  $V(\Psi)\le 0$ and
\[
	H(\Psi) M(\Psi)^{\frac{1-s_c}{s_c}} \le \mathfrak{I}_3 + \varepsilon.
\]
One may suppose that $V(\Psi)=0$ by replacing $\Psi$ with $c\Psi$ so that
$V(c\Psi)=0$ with a suitable $c\in (0,1)$ if $V(\Psi)<0$.
For $\lambda>0$, we let $\Psi_\lambda = \lambda^{\frac{d}2 - s_c} \Psi (\lambda \cdot)$.
Then, one has $V(\Psi_\lambda)=0$ and  
\[
 H(\Psi_\lambda) M(\Psi_\lambda)^{\frac{1-s_c}{s_c}}= H(\Psi) M(\Psi)^{\frac{1-s_c}{s_c}}
 \]
for any $\lambda>0$.
We now choose $\lambda$ so that the identity
\[
	H(\Psi_\lambda) = \tfrac{d}{2(1-s_c)}  M(\Psi_\lambda)
\]
holds. In this case, we have $\tfrac{2s_c}{d}H(\Psi_\lambda) +  M(\Psi_\lambda)=\frac2dH(\Psi_\lambda)$
and hence
\begin{align*}
	(\mathfrak{I}_3+\varepsilon)^{s_c} \ge{}& H(\Psi_\lambda)^{s_c} M(\Psi_\lambda)^{1-s_c}\\
	={}&  (\tfrac{d}{2(1-s_c)} )^{s_c-1} \tfrac{d}2(\tfrac{2s_c}{d} H(\Psi_\lambda) +  M(\Psi_\lambda))\\
	\ge{}& (\tfrac{d}{2(1-s_c)} )^{s_c-1} \tfrac{d}2 \mathfrak{I}_2(1)\\
	={}&  (\tfrac{d}{2(1-s_c)})^{s_c} \tfrac12(-g_{\min})^{s_c-\frac{d}2} \|Q\|_{L^2}^2,
\end{align*}
where we have used the explicit value of $\mathfrak{I}_2(1)=\mathfrak{I}(1)$ given in Lemma \ref{L:domega} to obtain the last line.
Since $\varepsilon>0$ is arbitrary, we obtain the desired inequality.

The minimizer to $\mathfrak{I}_3$ is easily obtained. We omit the details.
\end{proof}

\begin{proof}[Proof of Theorem \ref{T:Pstruct}]
We first remark that 
\begin{equation}\label{E:Pstructpf0}
E(\Phi) M(\Phi)^{\frac{1-s_c}{s_c}}
={} \tfrac{2s_c}{d} \mathfrak{I}_3 
={}	  \tfrac{s_c}{1-s_c} (\tfrac12(-g_{\min})^{s_c-\frac{d}{2}} \|Q\|_{L^2}^2)^{\frac{1}{s_c}}
\end{equation}
holds for any $\Phi \in \mathcal{G}$. In particular, the value is independent of the choice of $\Phi$.

(1) 
To prove the first assertion, it suffices to show that $V({\bf u})=0$ and ${\bf u}\neq0$ imply
\begin{equation}\label{E:Pstructpf1}
E({\bf u}) M({\bf u})^{\frac{1-s_c}{s_c}} \ge E(\Phi) M(\Phi)^{\frac{1-s_c}{s_c}}.
\end{equation}
Pick nonzero ${\bf u} \in (H^1(\R^d))^N$ such that $V ({\bf u})=0$.
Then, one has
$H({\bf u}) M({\bf u})^{\frac{1-s_c}{s_c}} \ge  \mathfrak{I}_3$ by definition of $\mathfrak{I}_3$.
Further, $V ({\bf u})=0$ gives us $E ({\bf u})=\frac{2s_c}{d}H ({\bf u})$.
Thus, by means of the first identity of \eqref{E:Pstructpf0}, we obtain \eqref{E:Pstructpf1}.

(2) $V ({\bf u})>0$ is equivalent to $E ({\bf u})>\frac{2s_c}{d} H ({\bf u})$. Hence,
\[
	H({\bf u}) M({\bf u})^{\frac{1-s_c}{s_c}} < \tfrac{d}{2s_c} E ({\bf u}) M({\bf u})^{\frac{1-s_c}{s_c}} \le 
	(1-\delta)\mathfrak{I}_3
\]
by assumption and \eqref{E:Pstructpf0}.
Since $\Phi \in \mathcal{G}$ attains the sharp Gagliardo-Nirenberg inequality (Theorem \ref{T:sGN}), 
we have
\begin{align*}
	-G({\bf u}) \le{}& \tfrac{-G(\Phi)}{H(\Phi)^{ps_p/2} M(\Phi)^{p(1-s_p)/2}}
	H({\bf u})^{\frac{ps_p}{2}} M({\bf u})^{\frac{p(1-s_p)}{2}}\\
	= &{} \tfrac{d-2s_c}{d} \(\tfrac{H({\bf u}) M({\bf u})^{\frac{1-s_c}{s_c}}}{H(\Phi) M(\Phi)^{\frac{1-s_c}{s_c }}}\)^{\frac{ps_p}{2}-1}
	H({\bf u})\\
	<&{} \tfrac{d-2s_c}{d}(1-\delta)^{\frac{2s_c}{d-2s_c}} H({\bf u}).
\end{align*}
where we have used $V(\Phi)=0$ to obtain the second line.
Hence,  there exists $\tilde{\delta}>0$ such that
\[
	V ({\bf u})=2H ({\bf u})+\tfrac{2d}{d-2s_c} G({\bf u}) > 2(1-(1-\delta)^{\frac{2s_c}{d-2s_c}}) H({\bf u})
	\ge \tilde{\delta} H({\bf u})
\]
as desired.

(3)  $V ({\bf u})<0$ implies that
$
	H({\bf u}) M({\bf u})^{\frac{1-s_c}{s_c}} \ge  \mathfrak{I}_3
$
and 
\begin{align*}
	H({\bf u}) M({\bf u})^{\frac{1-s_c}{s_c}} < \tfrac{d}{d-2s_c } (-G({\bf u})) M({\bf u})^{\frac{1-s_c}{s_c}}
	 = \tfrac{d }{2s_c} (2E({\bf u})-V({\bf u})) M({\bf u})^{\frac{1-s_c}{s_c}}.
\end{align*}
Combining these two inequalities and using the assumption, one obtains
\begin{align*}
	\mathfrak{I}_3 <{}&
	- \tfrac{d }{2s_c} V({\bf u}) M({\bf u})^{\frac{1-s_c}{s_c}}
		+\tfrac{d }{s_c} E({\bf u}) M({\bf u})^{\frac{1-s_c}{s_c}}\\
		\le {}& 
		- \tfrac{d }{2s_c} V({\bf u}) M({\bf u})^{\frac{1-s_c}{s_c}}
		+(1-\delta) \mathfrak{I}_3,
\end{align*}
which gives us the desired upper bound.
\end{proof}

\section{Blowup results and instability of ground states}\label{S:instability}

In this section, we prove instability result (Theorem \ref{T:instability}).
We split the proof into two parts: the case ${\bf n}=(1,\dots,1)$ and the case $d\ge2$ and $p\le 6$. 
Notice that $2^*\le 6$ for $d\ge 3$. Hence, $p\le 6$ is weaker than the energy-subcritical condition \eqref{E:prange}
for $d\ge3$.
\subsection{The case ${\bf n}=(1,\dots,1)$}
Let us first consider the case ${\bf n}=(1,\dots,1)$.
We need the assumption only for $d=1$ or $d=2$ and $p>6$.
However, the proof here works for all $d\ge1$ and $p\in (2,2^*)$.
We exploit the pseudo-conformal transformation in the
mass-critical case $p=2+\frac4d$ to obtain an explicit blowup solution.
The key ingredient in the case $2+\frac4d<p<2^*$ is the virial identity.

\begin{proof}[Proof of Theorem \ref{T:instability} when ${\bf n}=(1,\dots,1)$]
Fix $\omega>0$.
Pick $\Phi \in \mathcal{G}_{\omega}$.
Note that ${\bf u}(t):=e^{i\omega t}\Phi$ is a solution to \eqref{E:gNLS}.
We apply the pseudo-conformal transform to this solution:
For any $b>0$, we define ${\bf v}(t) $ by the formula
\[
	{\bf v}(t,x) := (1-b^{-2} t)^{-\frac{d}2} {\bf u}\(\tfrac{t}{1-b^{-2}t},\tfrac{x}{1-b^{-2}t} \)
	e^{-i \frac{|x|^2}{4(b^2-t)}}.
\]
Then, ${\bf v}(t)$ is also a solution to \eqref{E:gNLS}. Further, ${\bf v}(t)$ blows up at $t=b^2$.
Further, one sees that
\[
	\|{\bf u}(0) -{\bf v}(0) \|_{(H^1(\R^d))^N}
	=\| (e^{-i {|x|^2}/{4b^2}}-1)\Phi \|_{(H^1(\R^d))^N} \to 0
\]
as $b\to\infty$. 

Let us proceed to the case $p>2+\frac4d$.
Pick $\omega>0$ and
$\Phi \in  \mathcal{G}_\omega$.
Let $\Phi_c := c^{\frac{d}2} \Phi (c \cdot)$ and
 define a function $f: \R_+ \to \R$
by
\[
	f(c) := E(\Phi_c) M(\Phi_c)^{\frac{1-s_c}{s_c}} = M(\Phi)^{\frac{1-s_c}{s_c}} (c^2 H(\Phi) + c^{\frac{2d}{d-2s_c}} G(\Phi))
\]
One see that
$
	V(\Phi_c)
	=M(\Phi)^{-\frac{1-s_c}{s_c}} cf'(c).
$
Further,
by using $-\frac{d}{d-2s_c}G(\Phi)=H(\Phi)=\frac{d}{2(1-s_c)} \omega M(\Phi)$,
 one obtains
\[
	f(c) = M(\Phi)^{\frac{1}{s_c}} \omega\tfrac{d}{2(1-s_c)} ( c^2- \tfrac{d-2s_c}{d}c^{\frac{2d}{d-2s_c}}).
\]
One verifies that $f$ takes its minimum at $c=1$
and that $f'(c)<0$ for $c>1$.

We now let ${\bf u}_0:=\Phi_c$ for $c>1$. We remark that
$\| {\bf u}_0- \Phi \|_{(H^1(\R^d))^N} \to 0$
as $c\downarrow 1$. Fix $c>1$.
Then, $f(c)<f(1)$ implies that there exists $\delta\in(0,1)$ such that
\[
	E({\bf u}_0) M({\bf u}_0)^{\frac{1-s_c}{s_c}} \le (1-\delta)E(\Phi) M(\Phi)^{\frac{1-s_c}{s_c}}
\]
holds. Further, $f'(c)<0$ give us $V({\bf u}_0) < 0$.
Hence, we deduce from Theorem \ref{T:Pstruct} that
\[
	V({\bf u}_0) \le - \tilde\delta \mathfrak{I}_3 M({\bf u}_0)^{-\frac{1-s_c}{s_c}}
\]
Let ${\bf u}(t)$ be the $H^1$-solution to \eqref{E:gNLS} under ${\bf u}(0)={\bf u}_0$.
Since the mass and the energy are conserved under the \eqref{E:gNLS}-flow, we have
\[
	E({\bf u}(t)) M({\bf u}(t))^{\frac{1-s_c}{s_c}} \le (1-\delta)E(\Phi) M(\Phi)^{\frac{1-s_c}{s_c}}
\]
and
\[
	V({\bf u}(t)) \le - \tilde\delta \mathfrak{I}_3 M({\bf u}_0)^{-\frac{1-s_c}{s_c}}=: - \delta'<0
\]
for all $t\in I_{\max}$.

As $|x|{\bf u}_0 \in (L^2(\R^d))^N$, we see that $|x|{\bf u}(t) \in (L^2(\R^d))^N$ for all $t\in I_{\max}$ by the standard argument (see \cite{CazBook}).
By the virial identity, one sees that
\[
	\frac{d^2}{dt^2}\sum_{j=1}^N \int_{\R^d} |x|^2  |u_j(x)|^2 dx
	= 8 V({\bf u}(t))\le - 8 \delta'.
\]
In one hand, we have $\sum_{j=1}^N \int_{\R^d} |x|^2  |u_j(x)|^2 dx \ge 0$ for all $t\in I_{\max}$.
On the other hand, we see that its second derivative is bounded by a negative constant from above.
Thees two yield a contradiction if ${\bf u}(t)$ exists globally for $t>0$ or $t<0$.
 \end{proof}

\subsection{The case $d\ge 2$ and $p\le 6$}
Let us turn to the proof for the case $d\ge 2$ and $p\le 6$.
In this case, we do not need any assumption on the vector ${\bf n}$.
\subsubsection{Blowup results for radial solutions}
We obtain blowup-type results for radial solutions to \eqref{E:gNLS} by using
the argument in \cite{IKN} (See also \cite{DF,NP,NP2}).
For the mass-supercritical case $p>2+\frac4d$, we have the following.
\begin{theorem}\label{T:blowup}
Let $d\ge2$, $2+\frac4d< p<2^*$, and $p \le 6$. 
Suppose Assumptions \ref{A:1} and \ref{A:2}, with the same vector ${\bf n}$ as in \eqref{E:gNLS}.
Assume that $g_{\min}<0$.
If a radial $H^1$-solution ${\bf u}(t) $ on $I_{\max}\ni 0$ to \eqref{E:gNLS} satisfies
\[
	\sup_{t\in [0, T_{\max})} V({\bf u}(t)) <0
\]
then $T_{\max}<\infty$.
\end{theorem}
As for the mass-critical case $p=2+\frac4d$, we have a blowup or grow-up result.
We remark that $s_c=0$ and $V({\bf u})=2E({\bf u})$ hold in this case.

\begin{theorem}\label{T:grow-up}
Let $d\ge2$ and $p=2+\frac4d$. 
Suppose Assumptions \ref{A:1} and \ref{A:2}, with the same vector ${\bf n}$ as in \eqref{E:gNLS}.
Assume that $g_{\min}<0$.
If a radial $H^1$-solution ${\bf u}(t) $ on $I_{\max}\ni 0$ to \eqref{E:gNLS} satisfies
$E({\bf u}) <0$
then 
\begin{equation}\label{E:grow-up}
	\varlimsup_{t\uparrow T_{\max}} H({\bf u}(t)) = \infty.
\end{equation}
\end{theorem}
A standard blowup alternative and the mass conservation law show that \eqref{E:grow-up} is a necessary condition for $T_{\max}<\infty$.
In this sense, Theorem \ref{T:grow-up} is weaker than a blowup result.

For the proof, we introduce
\[
	J(t) := 2\int_{\R^d} (\nabla \chi)(x) \sum_{j=1}^N  n_j^{-1}  \Im (\overline{u_j} \nabla u_j)(t,x) dx
\]
with a real-valued function $\chi \in C_0^\infty (\R^d)$ to be specified later. 
\begin{lemma}[Localized virial identity]\label{L:lvi}
For an $H^1$-solution ${\bf u}(t)$ to \eqref{E:gNLS}, one has
\[
	J'(t) = 4 \sum_{k,m=1}^d \int_{\R^d} [\partial_{k} \partial_{m} \chi] \sum_{j=1}^N\Re (\overline{\partial_{k} u_j} \partial_{m} u_j) dx - \int_{\R^d} [\Delta \Delta \chi] \sum_{j=1}^N |u_j|^2 dx + \tfrac{p-2}{p} \int_{\R^d}
	[\Delta \chi] g({\bf u}) dx.
\]
\end{lemma}
\begin{proof}
Due to a direct computation.
\end{proof}

Let $R>0$ and choose $\chi$ as follows:
\[
	\chi(x) = R^2\chi_0(R^{-1} |x|), 
\]
where $\chi_0 \in C^\infty (\R)$ satisfies 
$\chi_0 (r) = r^2 $ for $r\le 1$, $\chi_0'(r) \le 2r$ for $1\le r \le 2$, $\chi_0'(r) =0$ for $r\ge 2$, and $\chi_0''(r)\le 2$ for all $r\in \R$. Note that $\chi(x) = |x|^2$ for $|x| \le R$ and hence $\Delta \chi = 2d$ 
for $|x|\le R$.
Further, we have the bound
$ \| \Delta \chi\|_{L^\infty}+ R^2 \|\Delta \Delta \chi\|_{L^\infty} \lesssim_{\chi_0} 1  $
for any $R>0$.

We first claim the following.
\begin{lemma}\label{L:lvi2}
Let ${\bf u}(t)$ be a radial $H^1$-solution to \eqref{E:gNLS} on $I \subset \R$.
There exist positive constants $C_1$ and $C_2$ depending on $d,p,\chi_0, g|_{\partial B}$ such that
\[
	J'(t) \le 8V({\bf u}(t)) + R^{-2} C_1 M({\bf u}) + C_2 R^{-\frac{(d-1)(p-2)}{2}} M({\bf u})^{\frac{p+2}4} H({\bf u}(t))^{\frac{p-2}4}
\]
holds for all $t\in I$ and $R>0$.
\end{lemma}
\begin{proof}
Since ${\bf u}(t)$ is a radial solution, one sees from Lemma \ref{L:lvi} that
\[
	J'(t) = 8V({\bf u}(t)) + \mathcal{R}_1 + \mathcal{R}_2 + \mathcal{R}_3 ,
\]
where
\[
	\mathcal{R}_1 = 4 \int_{\R^d} (\chi_0''(R^{-2}|x|) -2) \sum_{j=1}^N |\nabla u_j|^2 dx , \quad
	\mathcal{R}_2 = - \int_{\R^d} [\Delta \Delta \chi] \sum_{j=1}^N |u_j|^2 dx ,
\]
and
\[
	\mathcal{R}_3 = \tfrac{p-2}{p} \int_{\R^d}
	[\Delta \chi - 2d] g({\bf u}) dx .
\]
By the choice of $\chi_0$, one has $\mathcal{R}_1 \le 0$. Thanks to the bound on $\Delta\Delta\chi$, one finds
\[
	|\mathcal{R}_2| \le 2\|\Delta\Delta \chi\|_{L^\infty} M({\bf u}) \lesssim_{\chi_0}  R^{-2} M({\bf u}).
\]
Let $\rho:= (\sum_{j=1}^2 |u_j|^2)^{1/2}$.
Recalling the properties of $\Delta\chi$, one has
\[
	|\mathcal{R}_3|   \lesssim_{\chi_0}  
	\int_{|x|\ge R}  |g({\bf u})| dx
	\le
	\sup_{{\bf z}\in \partial B} |g({\bf z})|
	\int_{|x|\ge R}  \rho(t,x)^p dx
	\lesssim_{g|_{\partial B}}   R^{-\frac{(d-1)(p-2)}{2}}M({\bf u})  \| |x|^{\frac{d-1}2} \rho \|_{L^\infty}^{p-2}.
\]
By the radial Sobolev inequality (see, e.g. \cite{CazBook}*{Lemma 1.7.3}, for instance), one has
$
	\| |x|^{\frac{d-1}2} \rho \|_{L^\infty} \lesssim M({\bf u})^\frac14 H({\bf u})^\frac14.
$
Combining these estimates, we obtain the desired inequality.
\end{proof}

Let us now turn to the proof of the blowup/grow-up theorems.
\begin{proof}[Proof of Theorem \ref{T:blowup}]
Let us consider the case $p<6$. 
We prove by contradiction.
Suppose that ${\bf u}(t)$ is a radial $H^1$-solution to \eqref{E:gNLS} such that $T_{\max}=\infty$ and
$
	-\delta := \sup_{t\ge 0} V({\bf u}(t))  <0.
$

By Lemma \ref{L:lvi2}, one has
\begin{equation}\label{E:blowup1pf0}
	J'(t) \le 8V({\bf u}(t)) + R^{-2} C_1 M({\bf u}) + C_2 R^{-\frac{(d-1)(p-2)}{2}} M({\bf u})^{\frac{p+2}4} H({\bf u}(t))^{\frac{p-2}4}
\end{equation}
for all $t\ge 0$.
Note that the assumption $p< 6$ is equivalent to $\frac{p-2}4<1$.
We see from
 Young's inequality that
\[
	C_2  M({\bf u})^{\frac{p+2}4} H({\bf u}(t))^{\frac{p-2}4}
	\le \tfrac{4s_c }{d-2s_c}  H({\bf u}(t)) + C (C_2  M({\bf u})^{\frac{p+2}4})^{\frac4{6-p}}.
\]
By the identity $\frac{4s_c}{d-2s_c}H= \frac{2d}{d-2s_c} E- V$, one obtains
\[
	8V({\bf u})+\tfrac{4s_c }{d-2s_c}R^{-\frac{(d-1)(p-2)}{2}} H({\bf u}) = (8-R^{-\frac{(d-1)(p-2)}{2}})V ({\bf u})+ \tfrac{2d}{d-2s_c}R^{-\frac{(d-1)(p-2)}{2}} E({\bf u}).
\]
Plugging these inequalities to \eqref{E:blowup1pf0} and using the assumption, one obtains
\[
	J'(t) \le -(8-R^{-\frac{(d-1)(p-2)}{2}})\delta + R^{-2} C_1 M({\bf u}) + R^{-\frac{(d-1)(p-2)}{2}}(\tfrac{2d}{d-2s_c} E({\bf u}) +  \tilde{C}_2 M({\bf u})^{\frac{p+2}{6-p}}).
\]
Recall that $M({\bf u})$ and $E({\bf u})$ are conserved quantities.
We fix $R$ sufficiently large so that
\begin{equation}\label{E:blowup1pf1}
J'(t) \le -4\delta
\end{equation}
holds for $t\ge 0$.

By means of \eqref{E:blowup1pf1}, there exists $t_0\ge0$ such that 
$2 \delta t \le - J(t) \lesssim_{\chi_0} R M({\bf u})^\frac12 H({\bf u}(t))^\frac12$
for $t\ge t_0$, which implies that
\begin{equation}\label{E:blowup1pf2}
	H({\bf u}(t)) \gtrsim t^2
\end{equation}
for $t\ge t_0$. By the identity $\frac{4s_c}{d-2s_c}H= \frac{2d}{d-2s_c} E- V$,
the inequality \eqref{E:blowup1pf0} reads also as
\[
	J'(t) \le -\tfrac{32s_c}{d-2s_c}H({\bf u}(t))+ \tfrac{16d}{d-2s_c}E({\bf u}(t)) + R^{-2} C_1 M({\bf u}) + C_2 R^{-\frac{(d-1)(p-2)}{2}} M({\bf u})^{\frac{p+2}4} H({\bf u}(t))^{\frac{p-2}4}.
\]
Another use of Young's inequality gives us
\[
	J'(t) \le -\tfrac{16s_c}{d-2s_c}H({\bf u}(t))+ \tfrac{16d}{d-2s_c}E({\bf u}(t)) + R^{-2} C_1 M({\bf u}) +  C_3 R^{-\frac{(d-1)(p-2)}{2}} M({\bf u})^{\frac{p+2}{6-p}}
\]
by letting $R$ larger if necessary.
Since we have \eqref{E:blowup1pf2}, there exists $t_1\ge t_0$ such that
$J(t_1)\le0$ and
$
	J'(t) \le -\tfrac{8s_c}{d-2s_c}H({\bf u}(t))
$
for all $t\ge t_1$. Set $\xi (t) := \int_{t_1}^t H({\bf u}(s))ds$. One has
\[
	\tfrac{8s_c}{d-2s_c}\xi(t) \le -J(t) + J(t_1) \le - J(t) \le CR M({\bf u})^\frac12 (\xi'(t))^\frac12.
\]
Hence, one has an ordinary differential inequality
$(-\tfrac1{\xi(t)})' \ge A $
on $t\ge t_1$
with some positive constant $A$.
By an integration, we find 
	$A(t-t_1) \le \tfrac1{\xi(t_1)}-\tfrac1{\xi(t)} \le \tfrac1{\xi(t_1)}$,
which gives us a contradiction by letting $t$ large.

The case $p=6$ is handled similarly. Since $\frac{p-2}4=1$,
we obtain \eqref{E:blowup1pf1} from \eqref{E:blowup1pf0} 
without Young's inequality.
\end{proof}

\begin{proof}[Proof of Theorem \ref{T:grow-up}]
The standard blowup alternative argument yields \eqref{E:grow-up} if $T_{\max}<\infty$.
Hence, we consider the case $T_{\max}=\infty$.
We prove by contradiction.
Suppose that ${\bf u}(t)$ is a radial $H^1$-solution to \eqref{E:gNLS} such that $T_{\max}=\infty$,
$E({\bf u}) <0$, and
\[
	 H_{\max}:=\sup_{t\ge0} H({\bf u}(t)) <\infty.
\]
By Lemma \ref{L:lvi2}, one has
\[
	J'(t) \le 16E ({\bf u}) + R^{-2} C_1 M({\bf u}) + C_2 R^{-\frac{(d-1)(p-2)}{2}} M({\bf u})^{\frac{p+2}4} H_{\max}^{\frac{p-2}4}
\]
for $t\ge 0$.
Letting $R$ large, we obtain
$
	J'(t) \le 8E ({\bf u})<0
$
for $t\ge0$. Hence, there exists $t_0\ge0$ such that 
\[
	4|E ({\bf u})| t \le -J(t) \lesssim_{\chi_0} R M({\bf u})^\frac12 H_{\max}^\frac12
\]
for $t\ge t_0$.
This yields a contradiction by letting $t\to\infty$.
\end{proof}

\subsubsection{Proof of the instability result}
\begin{proof}[Proof of Theorem \ref{T:instability} when $d\ge2$ and $p\le6$]
Fix $\omega>0$. Pick $\Phi \in \mathcal{G}_{\omega}$. By the space-translation symmetry of \eqref{E:gNLS},
we may suppose that $\Phi$ is radially symmetric without loss of generality.
We take ${\bf u}_0=c \Phi$ with $c>1$. It is easy to see that
$\| u_0 - \Phi \|_{H^1} \to 0$ as $c\downarrow 1$.
Let ${\bf u}(t)$ be the corresponding maximal-lifespan $H^1$-solution. Since ${\bf u_0}$ is radial, so is ${\bf u}(t)$ for all $t \in I_{\max}$.

When $p>2+\frac4d$, mimicking the argument in the proof of the case ${\bf n}=(1,\dots,1)$,
we see that ${\bf u}(t)$ satisfies
\[
	V({\bf u}(t)) \le - \tilde\delta \mathfrak{I}_3 M({\bf u}_0)^{-\frac{1-s_c}{s_c}} = : -\delta'<0
\]
for $t \in I_{\max}$. Theorem \ref{T:blowup} then implies $T_{\max}<\infty$.

When $p=2+\frac4d$, one sees that  $E({\bf u}_0)=E(c\Phi) = c^2 H(\Phi) + c^p G(\Phi) = c^2(1-c^{p-2})H(\Phi)<0$.
Hence, in light of Theorem \ref{T:grow-up}, we see that ${\bf u}(t)$ satisfies
$\varlimsup_{t\uparrow T_{\max}} H({\bf u}(t))=\infty$.
\end{proof}

\section{Analysis of specific systems}\label{S:Cs}

In this section, we prove Corollaries \ref{C:app1}, \ref{C:app2}, \ref{C:app3}, \ref{C:app4}, and \ref{C:app5}.
In view of Theorems \ref{T:main}, \ref{T:excited}, \ref{T:stability}, and \ref{T:instability}, what we do for a specific system is to evaluate $g_{\min}$ and find the set $T_0$ defined in \eqref{D:gmin} and \eqref{D:T0}, respectively.
\subsection{System \eqref{E:nls1}}
Recall that
\[
	g (u_1,u_2) = \alpha |u_1|^4 + \beta |u_2|^4,
\]
where $\alpha, \beta \in \{-1,0,1\}$ satisfy $\alpha \ge \beta$.
It is easy to see that $g_{\min}=\frac12$ if $\beta=1$
and $g_{\min}=\beta$ if $\beta=0,-1$.
A negative critical point exists if and only if $\beta=-1$. 
When $\beta=-1$, one sees that  $T_0=\cup_{\theta}\{ (0,e^{i\theta})\}$ if $\alpha>\beta$
and $T_0=\cup_{\theta}\{ (0,e^{i\theta}), (e^{i\theta},0)\}$ if $\alpha=\beta$.
Thus, we obtain Corollary \ref{C:app1} from Theorems \ref{T:main}, \ref{T:excited}, \ref{T:stability}, and \ref{T:instability}.

\subsection{System \eqref{E:nls2}}
Recall that
\[
	g(u_1,u_2)=\alpha |u_1|^4 + \beta |u_2|^4+\sigma (|u_1|^2+|u_2|^2)^2,
\]
where $\alpha \ge \beta$ and $\sigma \in \{-1,1\}$.
Since the case $\alpha=\beta=0$ is trivial, we consider the other case.
Corollary \ref{C:app2} follows from the following proposition:
\begin{proposition}
Suppose that $(\alpha,\beta)\neq(0,0)$.
\begin{enumerate}
\item  The set of critical points of $g|_{\partial B}$ are $T_{0,1}\cup T_{0,2}$ if $\alpha\beta \le 0$
and $T_{0,1}\cup T_{0,2}\cup T_{0,3}$ if $\alpha\beta>0$,
where
\[
	T_{0,1} := \bigcup_{\theta}\{ (0,e^{i\theta})\}, \quad 
	T_{0,2} := \bigcup_{\theta}\{ (e^{i\theta},0)\},
\]
and
\[
	T_{0,3} := \bigcup_{\theta_1,\theta_2}\{ (\sqrt{\tfrac{\beta}{\alpha+\beta}}e^{i\theta_1}, \sqrt{\tfrac{\alpha}{\alpha+\beta}}e^{i\theta_2})\}.
\]
\item $g_{\min}$ and $T_0$ are given as follows:
\begin{itemize}
\item If $\beta \le 0$ then $g_{\min} = \beta+ \sigma$ and $T_0=T_{0,1} \cup T_{0,2}$ if $\alpha=\beta$
and $T_0=T_{0,1}$ otherwise.
\item If $\beta>0$ then $g_{\min} = \frac{\alpha\beta }{\alpha+\beta} + \sigma$ and $T_0=T_{0,3}$.
\end{itemize}
\end{enumerate}
\end{proposition}

\begin{proof}
Regarding $g$ as a function $\R^4\to\R$ by 
$z_1=x_1+iy_1$ and $z_2=x_2+iy_2$, we have
\[
	\nabla_{x_1.y_1,x_2,y_2} g =
	4(\alpha(x_1^2+y_1^2)x_1+\sigma x_1, \alpha(x_1^2+y_1^2)x_2+\sigma x_2,
	\beta(x_2^2+y_2^2)x_2+\sigma x_2,\beta(x_2^2+y_2^2)y_2+\sigma y_2)
\]
for $(z_1,z_2) \in \partial B$.
By Lagrange's multiplier theorem, $(x_1,y_1,x_2,y_2)$ is a critical point if and only if
the vector $\nabla_{x_1.y_1,x_2,y_2} g$ is a multiple of $(x_1,y_1,x_2,y_2)$. 
This occur if and only if $x_1^2+y_1^2=0$, $x_2^2+y_2^2=0$, or $\alpha (x_1^2+y_1^2) = \beta (x_2^2+y_2^2)$ holds.
Thus, we obtain critical points $(0,e^{i\theta})$, $(e^{i\theta},0)$, and
$
	(\sqrt{\tfrac{\beta}{\alpha+\beta}}e^{i\theta_1}, \sqrt{\tfrac{\alpha}{\alpha+\beta}}e^{i\theta_2}),
$
where $\theta,\theta_1,\theta_2 \in \R$.
The third occurs only when $\alpha \beta >0$, that is, when $\beta>0$ or $\alpha<0$.
Note that $g(0,e^{i\theta})=\beta+\sigma$, $g(e^{i\theta},0)=\alpha+\beta$,
and
\[
	g(\sqrt{\tfrac{\beta}{\alpha+\beta}}e^{i\theta_1}, \sqrt{\tfrac{\alpha}{\alpha+\beta}}e^{i\theta_2}) = \tfrac{\alpha \beta}{\alpha+\beta}  + \sigma.
\]

To characterize the ground the state let us find the minimum of these critical values.
It is obvious that $\alpha+\sigma \ge \beta + \sigma$. 
Hence, if $\alpha \beta \le 0$ then 
$g_{\min}=\beta + \sigma$ and $T_0=T_{0,1}$. Note that $(\alpha,\beta)\neq(0,0)$ implies $\beta<\alpha$.

Next consider the case $\beta>0$. Since $\beta \ge \tfrac{\alpha \beta}{\alpha+\beta} $, one obtains $g_{\min} = \tfrac{\alpha \beta}{\alpha+\beta}  + \sigma$ and $T_0=T_{0,3}$.

Let us finally consider the case $\alpha<0$. Since $\beta \le \tfrac{\alpha \beta}{\alpha+\beta} $, we have
$g_{\min} = \beta + \sigma$
in this case.
Further, $T_0=T_{0,1}$
if $\alpha>\beta$ and $T_0=T_{0,1}\cup T_{0,2}$ if $\alpha=\beta$.
\end{proof}

\subsection{System \eqref{E:nls3}}
Recall that
\[
	g(u_1,u_2)=\alpha_1 |u_1^2+u_2^2|^2 + \alpha_2 |u_1^2-u_2^2|^2 -4\alpha_2 |u_1|^2|u_2|^2+ (2\alpha_1 + r) (|u_1|^2+|u_2|^2)^2,
\]
where $\alpha_2\ge0$, $\alpha_1^2+\alpha_2^2=1$, $\alpha_1^2\neq \alpha_2^2$, and $r\in\R$.

To find a minimum value, It is useful to introduce 
\begin{equation}\label{D:h}
	h(\nu,\zeta) := g(\cos \nu, e^{i\zeta} \sin \nu).
\end{equation}
\begin{lemma}\label{L:h}
Suppose that $g$ satisfies \eqref{E:g-gauge} with $N=2$ and $(n_1,n_2)=(1,1)$.
It holds that
\[
	g_{\min} =  \min_{\nu \in \R / \pi \Z,\, \zeta\in \R/ 2\pi \Z } h( \nu,\zeta ).
\]
Let  $\nu_0 \not\in \frac\pi2 \Z/\pi \Z$, $\zeta_0 \in \R/2\pi \Z$, and $\theta \in \R/2\pi \Z$.
$ (e^{i\theta}\cos \nu_0, e^{i(\theta+\zeta_0)}\sin \nu_0 )$ is a critical point of $g|_{\partial B}$
if and only if $\partial_\nu h(\nu_0,\zeta_0) = \partial_ \zeta h(\nu_0,\zeta_0) = 0$.
\end{lemma}

Corollary \ref{C:app3} follows from our abstract theorems  and the following proposition:
\begin{proposition}
\begin{enumerate}
\item If $\alpha_2>0$ then the set of critical points of $g|_{\partial B}$ is 
$T_{0,1} \cup T_{0,2} \cup T_{0,3}$, where
\[
	T_{0,1}:=\bigcup_{\theta\in \R/2\pi\Z}\{ (e^{i\theta}, 0 ) ,\, (0, e^{i\theta} ) \},
\]
\[
	T_{0,2}:=\bigcup_{\theta\in \R/2\pi\Z,\, \sigma\in \{\pm1\}}\{ (2^{-1/2}e^{i\theta}, \sigma 2^{-1/2}e^{i\theta} ) \},
\]
and
\[
	T_{0,3}:=\bigcup_{\theta\in \R/2\pi\Z,\, \sigma\in \{\pm1\}}\{ (2^{-1/2}e^{i\theta}, i \sigma 2^{-1/2}e^{i\theta} ) \}.
\]
Further, for $j=1,2,3$, we have $g({\bf w}) = g_j$ for all ${\bf w} \in T_{0,j}$, where
\[
	g_1 = 3\alpha_1 + \alpha_2+ r, \qquad 
	g_2 = 3\alpha_1 - \alpha_2+ r, \qquad
	g_3 = 2\alpha_1 + r.
\]
If $\alpha_2=0$ then the set of critical points of $g|_{\partial B}$ is 
$T_{0,3} \cup T_{0,4}$, where
\[
	T_{0,4}:=\bigcup_{\nu, \theta\in \R/2\pi\Z}\{ (e^{i\theta}\cos \nu  , e^{i\theta}\sin \nu  ) \}.
\]
Further, we have $g({\bf w}) = g_4$ for all ${\bf w} \in T_{0,4}$, where $g_4 = 3\alpha_1 + r$.
\item $g_{\min}$ and $T_0$ are given as follows: 
\begin{itemize}
\item If $\alpha_1>\alpha_2$ then $g_{\min} = g_3$ and $T_0=T_{0,3}$.
\item If $-1<\alpha_1<\alpha_2$ then $g_{\min}=g_2$ and $T_0=T_{0,2}$.
\item If $(\alpha_1,\alpha_2)=(-1,0)$ then
$g_{\min}=g_4$ and $T_0=T_{0,4}$.
\end{itemize}
\end{enumerate}
\end{proposition}
One can summarize as $g_{\min} = \min (\alpha_1,\alpha_2) + 2\alpha_1- \alpha_2 + r$.
\begin{proof}
Regarding $g$ as a function $\R^4\to\R$ by 
$z_1=x_1+iy_1$ and $z_2=x_2+iy_2$, we have
\[
	\nabla_{x_1.y_1,x_2,y_2} g =4(\Re F_1(z_1,z_2), \Im F_1(z_1,z_2), \Re F_2(z_1,z_2), \Im F_2(z_1,z_2)).
\]

If $|z_1|=1$ and $z_2=0$ then one finds
$(F_1(z_1,0),F_2(z_1,0)) = ( 3\alpha_1 + \alpha_2 + r )(z_1,0)$.
Since $(F_1(z_1,0),F_2(z_1,0))$ is a constant multiple of $(z_1,z_2)$, we see from Lagrange's multiplier theorem
that $(z_1,0)$ is a critical point of $g|_{\partial B}$. The critical value is
$
	g(z_1,0) = 3\alpha_1 + \alpha_2 +r.
$

If $|z_2|=1$ and $z_1=0$ then
 $(F_1(0,z_2),F_2(0,z_2))  = ( 3\alpha_1 + \alpha_2 + r )(0,z_2)$.
Hence, $(0,z_2)$ is a critical point of $g|_{\partial B}$.
The critical value is
$g(0,z_2) = 3\alpha_1 + \alpha_2 +r$.

To find other critical points, let us consider
the function $h$ defined in \eqref{D:h}.
It takes the form
\[
	h(\nu,\zeta)=
	 2^{-1}(\alpha_1 - \alpha_2) \sin^2 2\nu \cos2\zeta
	-2^{-1}(\alpha_1 + 3\alpha_2) \sin^2 2\nu 
	 +3\alpha_1+\alpha_2  + r,
\]
where $\nu \in (0,\pi/2)\cup(\pi/2,\pi)$ and $\zeta \in [0,\pi)$.
Note that
\[
	\partial_\nu h(\nu,\zeta)=
	 2 \sin 2\nu \cos 2\nu ((\alpha_1 - \alpha_2) \cos2\zeta
	-(\alpha_1 + 3\alpha_2)) 
\]
and
\[
	\partial_\zeta h(\nu,\zeta)=-(\alpha_1 - \alpha_2) \sin^2 2\nu \sin 2\zeta.
\]
Since $\sin 2\nu\neq0$, $\partial_\zeta h(\nu,\zeta)=0$ implies $\sin 2\zeta =0$ and hence $\cos 2\zeta = \pm1$.
If $\cos 2\zeta =-1$ then $\partial_\nu h(\nu,\zeta)=0$ implies $\cos 2\nu =0$.
If $\cos 2\zeta =1$ then $\partial_\nu h(\nu,\zeta)=0$ implies $\cos 2\nu =0$ or $\alpha_2=0$.
Summarizing the above, we see that the set of
the critical points of $h$ in $( (0,\pi/2)\cup(\pi/2,\pi)) \times [0,2\pi)$ is given as
\[	
	\begin{cases}
	\{ \tfrac\pi4, \tfrac{3\pi}4 \} \times \{ 0, \tfrac\pi2 \}  &\text{if } \alpha_2 >0 , \\
	(\{ \tfrac\pi4, \tfrac{3\pi}4 \} \times\{ \tfrac\pi2 \} ) \cup (( (0,\tfrac\pi2)\cup(\tfrac\pi2,\pi))  \times \{ 0 \}) &\text{if } \alpha_2 =0  .
	\end{cases}
\]
The values of $g_j$ ($j=1,2,3,4$) are easily found.

Let us find $g_{\min}$.
If $\alpha_1>\alpha_2$ then $g_1 \ge g_2 > g_3$.
Hence, $g_{\min} =g_3= 2\alpha_1 + r$.
Similarly, if $\alpha_1 < \alpha_2$ and $\alpha_2>0$, i.e., if $-1<\alpha_1<\alpha_2$ then we have $\min (g_1,g_3)> g_2$.
If $\alpha_1 < \alpha_2=0$, i.e., $(\alpha_1,\alpha_2)=(-1,0)$ then we have $g_3 > g_4$.
Thus, we obtain the result.
\end{proof}

\subsection{System \eqref{E:nls4}}
Recall that 
\begin{align*}
	g(u_1,u_2)=& \alpha_1 |u_1^2+u_2^2|^2 + \alpha_2 |u_1^2-u_2^2|^2 -4\alpha_2 |u_1|^2|u_2|^2   \\
	&+2\alpha_3  (|u_1|^4-|u_2|^4) + (2\alpha_1 + r) (|u_1|^2+|u_2|^2)^2,
\end{align*}
where
$ \alpha_2\ge 0$, $\alpha_3>0 $,  $\alpha_1^2+\alpha_2^2+\alpha_3^2=1$, $\alpha_1\neq\alpha_2$, and $r\in\R$.
Corollary \ref{C:app4} follows from the following proposition:
\begin{proposition}\label{P:app4}
\begin{enumerate}
\item The set of critical points of $g|_{\partial B}$ is given as follows
\[
	\begin{cases}
	T_{0,1} \cup T_{0,2}  & \text{if }  \max (2\alpha_2, |\alpha_1+\alpha_2|) \le \alpha_3, \\
	T_{0,1} \cup T_{0,2} \cup T_{0,3} &\text{if }  |\alpha_1+\alpha_2|\le \alpha_3 < 2\alpha_2,\\
	T_{0,1} \cup T_{0,2} \cup T_{0,4} &\text{if } 2\alpha_2 \le \alpha_3 < |\alpha_1+\alpha_2| ,\\
	T_{0,1} \cup T_{0,2} \cup T_{0,3} \cup T_{0,4} &\text{if } \alpha_3 < \min (2\alpha_2, |\alpha_1+\alpha_2|),
	\end{cases}
\]
where
\[
	T_{0,1}:=\bigcup_{\theta\in \R/2\pi\Z}\{ (e^{i\theta}, 0 ) \},\qquad
	T_{0,2}:=\bigcup_{\theta\in \R/2\pi\Z}\{ (0, e^{i\theta} ) \},
\]
\[
	T_{0,3}:=\bigcup_{\theta \in \R/2\pi\Z,\,\sigma \in \{\pm 1\}} \left\{ ( \sqrt{\tfrac{2\alpha_2-\alpha_3}{4\alpha_2}}e^{i\theta}, \sigma\sqrt{\tfrac{2\alpha_2+\alpha_3}{4\alpha_2}}e^{i\theta}) \right\},
\]
and
\[
	T_{0,4} := \bigcup_{\theta \in \R/2\pi\Z,\,\sigma \in \{\pm 1\}} \left\{ ( \sqrt{\tfrac{\alpha_1+\alpha_2-\alpha_3}{2(\alpha_1+\alpha_2)}}e^{i\theta}, i\sigma\sqrt{\tfrac{\alpha_1+\alpha_2+\alpha_3}{2(\alpha_1+\alpha_2)}}e^{i\theta}) \right\}.
\]
Further, for $j=1,2,3$, we have $g({\bf w}) = g_j$ for all ${\bf w} \in T_{0,j}$, where
\begin{align*}
	g_1 ={}& 3\alpha_1 + \alpha_2 +2\alpha_3 + r, &
	g_2 ={}& 3\alpha_1 + \alpha_2 -2\alpha_3 + r , \\
	g_3 ={}& -\tfrac{\alpha_3^2}{2\alpha_2} + 3\alpha_1 -\alpha_2 + r, &
	g_4:={}& -\tfrac{\alpha_3^2}{\alpha_1+ \alpha_2} + 2\alpha_1 + r.
\end{align*}
\item $g_{\min}$ and $T_0$ is given as follows: 
\[
	g_{\min} =
	\begin{cases}	
	g_2 & \text{if }\alpha _3 \ge \tilde{\alpha} ,\\
	g_3 & \text{if }\alpha_3< \tilde{\alpha}= 2\alpha_2,\\
	g_4 & \text{if }\alpha_3< \tilde{\alpha}=\alpha_1+ \alpha_2
	\end{cases}
\]
 and
\[
	T_0 =
	\begin{cases}	
	T_{0,2} & \text{if }\alpha _3 \ge \tilde{\alpha} ,\\
	T_{0,3} & \text{if }\alpha_3< \tilde{\alpha}= 2\alpha_2,\\
	T_{0,4} & \text{if }\alpha_3< \tilde{\alpha}=\alpha_1+ \alpha_2,
	\end{cases}
\]
where $\tilde{\alpha}=\max(\alpha_1+\alpha_2,2\alpha_2)= \alpha_2 + \max(\alpha_1,\alpha_2)$.
\end{enumerate}
\end{proposition}

\begin{proof}
Regarding $g$ as a function $\R^4\to\R$ by 
$z_1=x_1+iy_1$ and $z_2=x_2+iy_2$, we have
\[
	\nabla_{x_1.y_1,x_2,y_2} g =4(\Re F_1(z_1,z_2), \Im F_1(z_1,z_2), \Re F_2(z_1,z_2), \Im F_2(z_1,z_2)).
\]

If $|z_1|=1$ and $z_2=0$ then
$(F_1(z_1,0),F_2(z_1,0)) = ( 3\alpha_1 + \alpha_2 +2\alpha_3 + r )(z_1,0)$.
Hence, $(z_1,0)$ is a critical point of $g|_{\partial B}$. The critical value is $g(z_1,0) = g_1$.

If $|z_2|=1$ and $z_1=0$ then
$(F_1(0,z_2) ,F_2(0,z_2)) = ( 3\alpha_1 + \alpha_2 -2\alpha_3 + r )(0,z_2)$.
Hence, $(0,z_2)$ is a critical point of $g|_{\partial B}$.
The critical value is $g(0,z_2) = g_2$.

To find the other critical points, let us consider
the function $h$ defined in \eqref{D:h}.
It takes the form
\begin{align*}
	h(\nu,\zeta) 
	={}&2^{-1}(\alpha_1 - \alpha_2)  \sin^2 2\nu \cos2\zeta-
	2^{-1}(\alpha_1 + 3\alpha_2)  \sin^2 2\nu 
	+2\alpha_3  \cos 2\nu +3\alpha_1+\alpha_2+r ,
\end{align*}
where $\nu \in (0,\pi/2)\cup (\pi/2,\pi)$ and $\zeta \in [0,\pi)$.
Note that
\[
	\partial_\nu h(\nu,\zeta)=
	 2 \sin 2\nu (\cos 2\nu ((\alpha_1 - \alpha_2) \cos2\zeta
	-(\alpha_1 + 3\alpha_2)) -2\alpha_3)
\]
and
\[
	\partial_\zeta h(\nu,\zeta)=-(\alpha_1 - \alpha_2) \sin^2 2\nu \sin 2\zeta.
\]
Since $\sin 2\nu\neq0$, $\partial_\zeta h(\nu,\zeta)=0$ implies $\sin 2\zeta =0$ and hence $\cos 2\zeta = \pm1$.

If $\cos 2\zeta =1$ then $\partial_\nu h(\nu,\zeta)=0$ implies $\alpha_2>0$ and $\cos 2\nu =-\frac{\alpha_3}{2\alpha_2}$.
We remark that such $\nu$ exists in $(0,\pi/2)\cup (\pi/2,\pi)$
if $\alpha_3 < 2\alpha_2$.
The critical value in this case is $g_3$
and the set of the corresponding critical points is $T_{0,3}$.

If $\cos 2\zeta =-1$ then $\partial_\nu h(\nu,\zeta)=0$ implies $\alpha_1 +\alpha_2 \neq0$ and
$\cos 2\nu = -\frac{\alpha_3}{ \alpha_1 +\alpha_2 }$.
We remark that such $\nu$ exists in $(0,\pi/2)\cup (\pi/2,\pi)$ if the modulus of the right hand side is less than one, that is,
if $\alpha_3 < |\alpha_1+\alpha_2|$.
The critical value in this case is $g_4$
and the set of the corresponding critical points is $T_{0,4}$.

Let us find $g_{\min}$ and $T_0$. One has $g_1>g_2$.
Hence, if $\max (2\alpha_2, |\alpha_1+\alpha_2|) \le \alpha_3$ then $g_{\min}=g_2$ and $T_0=T_{0,2}$.
Further, it follows that $g_2 > g_3$ for $ \alpha_3 < 2\alpha_2 $ since
$g_2 > g_3 \Leftrightarrow (2\alpha_2- \alpha_3)^2  > 0$.
Similarly, $g_2 > g_4$ holds if $\alpha_3 < \alpha_1+\alpha_2  $ and
$g_4 >g_2$ holds if $\alpha_3 < -(\alpha_1+\alpha_2)  $.
Finally, under the assumption $\alpha_3 < \min (2\alpha_2,\alpha_1+\alpha_2)$, $g_3>g_4$ holds if and only if $\alpha_1>\alpha_2$.
One can summarize as in the statement.
\end{proof}

\subsection{System \eqref{E:nls5}}
In this case, we have
\begin{align*}
	g(u_1,u_2)=& \alpha_1 |u_1^2+u_2^2|^2 + \alpha_2 |u_1^2-u_2^2|^2 -4\alpha_2 |u_1|^2|u_2|^2  +2\alpha_3 \cos \eta (|u_1|^4-|u_2|^4) \\
	& +4 \alpha_3 \sin \eta (|u_1|^2+|u_2|^2) \Re (\overline{u_1}u_2) + (2\alpha_1 + r) (|u_1|^2+|u_2|^2)^2,
\end{align*}
where
$ \alpha_2>0$, $\alpha_3>0 $,  $\alpha_1^2+\alpha_2^2+\alpha_3^2=1$, $\eta \in (0,\pi)$, and $r\in\R$.
We put an additional assumption $\alpha_1 >0$ if $\eta \in (\pi/2,\pi)$.
For $\alpha_2,\alpha_3>0$ and $k\in (0,\pi/2]$, recall that
$\theta_0=\theta_0(k) \in [k,\pi/2]$ is a unique solution to $ \alpha_2  \sin 2\theta   =  \alpha_3 \sin (\theta -k)$ in $[k,\pi/2]$.
Corollary \ref{C:app5}  follows from the following proposition.
\begin{proposition}\label{P:app5}
\begin{enumerate}
\item 
Let $\theta_j$ ($j=0,1,2,3$) be the solution to \eqref{E:equation} with $\rho=\alpha_3/\alpha_2$ and $\tau=\eta$ and define
\[
	\tilde{T}_{c,j}:=\bigcup_{\theta \in \R/2\pi \Z} (e^{i\theta} \cos \tfrac{\theta_j}2, e^{i\theta} \sin \tfrac{\theta_j}2).
\]
The function $g|_{\partial B}$ has the critical points
$\cup_{j=0,1,2,3} \tilde{T}_{c,j}$ if $\alpha_3/\alpha_2 \le \rho_*(\eta)$ and 
 $\cup_{j\in \mathcal{J}(\eta)} \tilde{T}_{c,j}$ if $\alpha_3/\alpha_2 > \rho_*(\eta)$, where
\[
	\mathcal{J}(\eta):= \begin{cases}
	\{ 0,3 \}&\text{if  } \eta<\pi/2 ,\\
	\{ 1,3 \}&\text{if  } \eta=\pi/2 , \\
	\{ 2,3 \}&\text{if  } \eta>\pi/2 .
	\end{cases}
\]

If 
\begin{equation}\label{E:c5pfothercrit}
\alpha_1\neq\alpha_2 \quad \text{and} \quad
\alpha_3^2 \le \tfrac{(\alpha_1 + \alpha_2)^2(\alpha_1 - \alpha_2)^2}{\alpha_1^2 + \alpha_2^2 - 2\alpha_1 \alpha_2 \cos 2\eta}
\end{equation}
then the set
\begin{align*}
	{T}_{c}:= \bigcup_{\theta \in \R/2\pi \Z,\,\sigma \in \{\pm 1\} } \Big\{ \Big(& \sigma e^{i\theta} \sqrt{\tfrac{\alpha_1 + \alpha_2-\alpha_3 \cos \eta }{2(\alpha_1 + \alpha_2)} }, \\
	&	 e^{i\theta}  \sqrt{\tfrac{\alpha_1 + \alpha_2}{2 (\alpha_1 + \alpha_2-\alpha_3 \cos \eta)} }\Big(-\tfrac{\sigma \alpha_3 \sin \eta}{\alpha_1-\alpha_2}+i \sqrt{ 1- \tfrac{\alpha_3^2(\alpha_1^2 + \alpha_2^2 - 2\alpha_1 \alpha_2 \cos 2\eta)}{(\alpha_1 + \alpha_2)^2(\alpha_1 - \alpha_2)^2} }\Big) \Big) \Big\}
\end{align*}
also gives critical points of $g|_{\partial B}$.
These are all the critical points of $g|_{\partial B}$.
\item $g_{\min}$ and $T_0$ is given as follows: 
\begin{itemize}
\item 
If the condition \eqref{E:5cond} is fulfilled then
\begin{equation*}
	g_{\min} = 
	 - \tfrac{\alpha_3^2(\alpha_1 - \alpha_2 \cos 2 \eta)}{\alpha_1^2-\alpha_2^2}  + 2\alpha_1+r   
\end{equation*}
and $T_0=T_c$.
\item If \eqref{E:5cond} is not satisfied then
\[
	g_{\min} = \alpha_2  \cos 2\theta_3
+2\alpha_3 \cos (\theta_3 -  \eta ) +3\alpha_1+\alpha_2 + r
\]
and $T_0=\tilde{T}_{c,3}$.
\end{itemize}
\end{enumerate}
\end{proposition}

\begin{remark}
In \eqref{E:nls4} case, $\eta=0$, the function $g$ has the symmetry $g(-u_1,u_2)=
g(u_1,-u_2)=g(u_1,u_2)$.
This shows that the set $T_0$ of minimum points of $g|_{\partial B}$ is symmetric with respect to $\{z_1=0\}$ and $\{z_2=0\}$. 
Indeed, as is observed in Proposition \ref{P:app4}, the set $T_0$ actually has this symmetry.
However, this symmetry is broken in \eqref{E:nls5} case, $\eta>0$.
We see that $g(-u_1,u_2)=g(u_1,-u_2)\neq g(u_1,u_2)$ unless $u_1=0$ or $u_2=0$.
Consequently, the set $T_0$ does not have such symmetry any more.
\end{remark}

\subsubsection{Preliminaries}
Let us first prove Lemma \ref{L:pzero2}.

\begin{proof}[Proof of Lemma \ref{L:pzero2}]
Recall that
$f_{\rho,\tau} (\theta):=\sin 2\theta + \rho \sin (\theta -\tau)$.

(1)  Fix $\tau \in (0,\pi/2)$.
It is obvious that $f_{0,\tau}(\theta)=0$ has four solutions $\theta = \pi/2, \pi , 3\pi/2, 2\pi$. Hence, we let $\rho>0$ in the rest of the proof.
Since $f_{\rho,\tau}$ solves the ODE $y''=-y - 3 \sin 2\theta$,
it follows from the Duhamel principle that
\begin{equation}\label{E:pzero2pf1}
	f_{\rho,\tau} (\theta) = \cos (\theta- \tilde{\theta}) f_{\rho,\tau} (\tilde{\theta})
	+ \sin (\theta - \tilde{\theta})f_{\rho,\tau}' (\tilde{\theta})
	-3 \int_{\tilde{\theta}}^\theta \sin (\theta - \tau) \sin 2\tau d\tau
\end{equation}
for any $\theta,\tilde{\theta} \in \R$.

Let us begin with the study of solutions in the range $[0,\pi/2)$.
Since $f_{\rho,\tau} (0) = -\rho \sin \tau <0$ and $f_{\rho,\tau} (\pi/2) = \rho \cos \tau>0$, there exists a solution $\theta_0 \in (0,\pi/2)$ to $f_{\rho,\tau}(\theta)=0$.
Without loss of generality, we may suppose that $\theta_0$ is the largest solution in this interval.
Then, $f_{\rho,\tau}>0$ on $(\theta_0,\pi/2)$. Further, one sees from \eqref{E:pzero2pf1} with $\tilde\theta=\theta_0$
that
\[
	f_{\rho,\tau} (\theta) = \sin (\theta- \theta_0) f'_{\rho,\tau} (\theta_0)
	-3 \int_{\theta_0}^\theta \sin (\theta - \tau) \sin 2\tau d\tau.
\]
By putting $\theta=\pi/2$ in this formula and using the fact that $f_{\rho,\tau}(\pi/2)>0$,
one obtains $f'_{\rho,\tau} (\theta_0)>0$.
Then, the above formula yields $f_{\rho,\tau} <0$ on $[0,\theta_0)$.
Hence, $\theta_0$ is the only solution in this interval.
Finally, since $f_{\rho,\tau} (\tau) = \sin 2\tau>0$, one has the bound $\theta_0<\tau$.

Let us move on to the solutions in $(\pi, \frac{3\pi}2]$.
Since $f_{\rho,\tau}(\pi)=\rho \sin \tau >0$ and $f_{\rho,\tau} (3\pi/2)= - \rho \cos \tau <0$,
the equation $f_{\rho,\tau} (\theta)=0$ has a solution in $(\pi,3\pi/2)$.
We let $\theta_3 \in (\pi,3\pi/2)$ be the smallest solution in this range.
Then, one has $f_{\rho,\tau}(\theta_3)=0$ and $f_{\rho,\tau}>0$ on $(\pi,\theta_3)$.
One sees from \eqref{E:pzero2pf1} with $\tilde\theta=\theta_3$ that
\[
	f_{\rho,\tau} (\theta) = 
	 \sin (\theta - \theta_3)f_{\rho,\tau}' (\theta_3)
	-3 \int_{\theta_3}^\theta \sin (\theta - \tau) \sin 2\tau d\tau.
\]
By substituting $\theta=\pi$ to this formula and using the fact that $f_{\rho,\tau}(\pi)>0$, one sees that $f_{\rho,\tau}'(\theta_3)<0$.
Together with the above formula, this shows  $f_{\rho,\tau}<0$ on $(\theta_3,3\pi/2)$.
Hence, $\theta_3$ is the unique solution in this interval.
Noting that $f_{\rho,\tau} (\pi + \tau) = \sin 2\tau>0$, we have $\pi+ \tau < \theta_3$.

Let us next consider solutions in $(\frac{3\pi}2,2\pi)$. 
One has $\sin 2\theta<0$.
Since $\tau \in (0,\pi/2)$, we see that $\theta - \tau \in (\pi,2\pi)$
and hence that $\rho\sin (\theta-\tau)<0$.
Therefore, we have $f_{\rho,\tau} (\theta)< 0$ for all $\theta \in (3\pi/2,2\pi)$,
which means that
$f_{\rho,\tau}(\theta)=0$ has no solution in this range.

We finally consider the interval $[\pi/2,\pi]$.
Define
\[
	\rho_* : = \inf \left\{ \rho >0 \ \middle|\ \inf_{\theta \in [{\pi}/2,\pi]} f_{\rho,\tau}(\theta) >0 \right\}.
\]
As $\inf_{\theta \in [\frac{\pi}2,\pi]} \sin (\theta -\tau)=\min (\sin \tau,\cos \tau)>0$, $\rho_*$ is finite. 
Since $f_{\rho,\tau}(\tau+\pi/2) = -\sin 2\tau +\rho\le 0$ for $\rho\le \sin 2\tau$, one has the lower bound $\rho_* \ge \sin 2\tau>0$.

By definition, if $\rho > \rho_*$ then there is no solution to $f_{\rho,\tau}(\theta)=0$
in $[{\pi}/2,\pi]$.

Let us prove if $\rho=\rho_*$  the equation $f_{\rho_*,\tau}(\theta)=0$ has a unique solution
in $[{\pi}/2,\pi]$.
Since $f_{\rho,\tau}$ is continuous in $L^\I ([{\pi}/2,\pi])$ with respect to $\rho$,
one has
\[
	\inf_{\theta \in [{\pi}/2,\pi]} f_{\rho_*,\tau}(\theta) =0.
\]
Since $f_{\rho_*,\tau}$ is continuous in $\theta$ and since $f_{\rho_*,\tau}(\pi/2)>0$
and $f_{\rho_*,\tau}(\pi)>0$,
there exists $\theta_*\in (\pi/2,\pi)$ such that $f_{\rho_*,\tau}(\theta_*)=0$.
It follows that $f_{\rho_*,\tau}'(\theta_*)=0$. Then, we see from \eqref{E:pzero2pf1} with $\tilde\theta=\theta_*$ that
\[
	f_{\rho_*,\tau}(\theta) =
		-3 \int_{\theta_*}^\theta \sin (\theta - \tau) \sin 2\tau d\tau.
\]
Hence, we see that $f_{\rho_*,\tau}\ge0$ on $[\pi/2,\pi]$. Further, in this interval, the equality holds if and only if $\theta=\theta_*$.
Thus, $f_{\rho_*,\tau}(\theta)=0$ has the solution $\theta_*\in (\pi/2,\pi)$ which is unique in $[{\pi}/2,\pi]$. 

We now prove that if $0 \le \rho <\rho_*$ then there exist exactly two solutions to $f_{\rho,\tau}(\theta)=0$ in $[{\pi}/2,\pi]$. 
Fix $\rho\in (0,\rho_*)$.
Since we have $f_{\rho,\tau}(\pi/2)> 0$, $f_{\rho,\tau}(\pi)>0$, and
\[
	f_{\rho,\tau} (\theta_*) = (\rho_*-\rho) \sin (\theta_* -\tau) <0,
\]
there exist two solutions $\theta_1 \in (\pi/2,\theta_*)$ and
$\theta_2 \in (\theta_*,\pi)$.
By Rolle's theorem, there exists $\theta_{**} \in (\theta_2,\theta_3)$ such that
$f_{\rho,\tau}'(\theta_{**})=0$. Then, \eqref{E:pzero2pf1} with $\tilde{\theta}=\theta_{**}$ yields
\[
	f_{\rho,\tau} (\theta) = \cos (\theta-\theta_{**}) f_{\rho,\tau}(\theta_{**})
	-3 \int_{\theta_{**}}^\theta \sin (\theta - \tau) \sin 2\tau d\tau.
\]
If $f(\theta_{**})\ge0$ then $f_{\rho,\tau}(\theta)=0$ has at most one solution in $[{\pi}/2,\pi]$, a contradiction.
Hence, $f_{\rho,\tau}(\theta_{**})<0$. Then, the above formula shows that $f_{\rho,\tau}$
is strictly decreasing in $(\pi/2,\theta_{**})$ and strictly decreasing in $(\theta_{**},\pi)$.
In particular, $\theta_1$ and $\theta_2$ are only solutions to $f_{\rho,\tau}(\theta)=0$
in $[{\pi}/2,\pi]$.
Combining the above, we obtain the first assertion.

(2) is obvious.

(3) follows from (1) by symmetry. If we define $\zeta = \pi - \theta$, $f_{\rho,\tau}(\theta)=0$ is equivalent to
$\sin 2\zeta + \rho \sin (\zeta - (\pi-\tau))=0$. Thus, we have $\rho_{*}(\tau)=\rho_*(\pi-\tau)$ and
$\theta_0(\tau)=\pi - \theta_2(\pi-\tau)$, $\theta_1(\tau)= \pi - \theta_1(\pi-\tau)$,
$\theta_2(\tau)=\pi - \theta_0(\pi-\tau)$, and $\theta_3(\tau)=3\pi - \theta_3(\pi-\tau)$.
\end{proof}

\begin{lemma}\label{L:pmin}
For $\rho>0$ and $\tau \in (0,\pi)$, the function
\[
	p_{\rho,\tau}(\theta):= \cos 2\theta   +  2\rho \cos (\theta -\tau) 
\]
takes its minimum at $\theta = \theta_3$. There is no other minimum point in $[0,2\pi)$.
\end{lemma}
\begin{proof}
We only consider the case $\tau\neq\pi/2$ since $\tau=\pi/2$ is obvious.
Note that $p(\theta) =  \cos 2\theta + 2\rho \cos \tau \cos \theta + 2\rho \sin \tau \sin \theta$.

If $\tau  \in (0,\pi/2)$ then one sees that minimum is attained only in $[\pi,3\pi/2]$.
Indeed, if $\theta \in [0,\pi/2]$ then one has $p_{\rho,\tau}(\theta)> p_{\rho,\tau}(\theta+\pi)$.
If $\theta \in [\pi/2,\pi)$ then $p_{\rho,\tau}(\theta)>p_{\rho,\tau}(2\pi-\theta)$ holds.
Similarly, if $\theta \in (3\pi/2,2\pi]$ then $p_{\rho,\tau}(\theta)>p_{\rho,\tau}(3\pi-\theta)$.
These imply that minimum of $p_{\rho,\tau}$ is not achieved in $[\pi,3\pi/2]^c$.
Since a minimum point is a critical point and since
$\theta_3$ is the only critical point in this interval for $\rho>0$, one sees that the minimum point must be this point.

If $\tau \in (\pi/2,\pi)$ then one sees that minimum is attained only in $[3\pi/2,2\pi]$ by a similar symmetry argument. By the previous lemma, $\theta_3$ is the only critical point in this range
and hence it must be the minimum point.
\end{proof}

\subsubsection{Proof of Proposition \ref{P:app5}}
\begin{proof}
Regarding $g$ as a function $\R^4\to\R$ by 
$z_1=x_1+iy_1$ and $z_2=x_2+iy_2$, we have
\[
	\nabla_{x_1.y_1,x_2,y_2} g =4(\Re F_1(z_1,z_2), \Im F_1(z_1,z_2), \Re F_2(z_1,z_2), \Im F_2(z_1,z_2)).
\]

If $|z_1|=1$ and $z_2=0$ then
\[
	F_1(z_1,0) = ( 3\alpha_1 + \alpha_2 +2\alpha_3\cos \eta + r )z_1, \quad
	F_2(z_1,0) = \alpha_3 \sin \eta z_1.
\]
Since $\alpha_3 \sin \eta\neq0$, $(F_1(z_1,0),F_2(z_1,0))$ is not a constant multiple of $(z_1,z_2)$.
Hence, $(z_1,0)$ is not a critical point of $g|_{\partial B}$. 

If $|z_2|=1$ and $z_1=0$ then
\[
	F_1(0,z_2) = \alpha_3 \sin \eta z_2, \quad
	F_2(0,z_2) = ( 3\alpha_1 + \alpha_2 -2\alpha_3\cos \eta + r )z_2.
\]
Since $\alpha_3 \sin \eta\neq0$, $(0,z_2)$ is not a critical point of $g|_{\partial B}$.

To find other critical points, let us consider
the function $h$ defined in \eqref{D:h}.
It takes the form
\begin{align*}
	h(\nu,\zeta) 
	={}&2^{-1}(\alpha_1 - \alpha_2)  \sin^2 2\nu \cos2\zeta-
	2^{-1}(\alpha_1 + 3\alpha_2)  \sin^2 2\nu \\
	&
	+2\alpha_3 \cos \eta \cos 2\nu +2\alpha_3 \sin \eta \sin 2\nu \cos\zeta +3\alpha_1+\alpha_2+r,
\end{align*}
where $\nu \in (0,\pi/2)\cup  (\pi/2,\pi)$ and $\zeta \in [0,\pi)$.
One has
\[
	\partial_\nu h(\nu,\zeta)=
	 2 \sin 2\nu (\cos 2\nu ((\alpha_1 - \alpha_2) \cos2\zeta
	-(\alpha_1 + 3\alpha_2)) -2\alpha_3\cos \eta)
	+ 4 \alpha_3 \sin \eta \cos 2\nu \cos \zeta
\]
and
\begin{align*}
	\partial_\zeta h(\nu,\zeta) 
	={}&-2 \sin\zeta  \sin 2\nu ( (\alpha_1 - \alpha_2)  \sin 2\nu  \cos \zeta
	+\alpha_3 \sin \eta ).
\end{align*}

Since $\sin 2\nu \neq0$, $\partial_\zeta h=0$ implies that $\zeta=0$ or
\begin{equation}\label{E:app5_mod_pf1}
	(\alpha_1 - \alpha_2)  \sin 2\nu  \cos \zeta
	+\alpha_3 \sin \eta=0.
\end{equation}

Let us first consider the case $\zeta =0$. One has
\[
	 h(\nu,0)=\alpha_2 \cos 4 \nu + 2\alpha_3 \cos (2\nu-\eta) + 3\alpha_1  +r
	 =\alpha_2 p_{{\alpha_3}/{\alpha_2}, \eta}(2\nu) + 3\alpha_1  +r.
\]
Hence, by means of Lemma \ref{L:pzero2}, the solutions $\partial_\nu h(\nu,0)$ in $(0,\pi/2)\cup (\pi/2,\pi)$
are $\nu =\tfrac12 \theta_j(\alpha_3/\alpha_2)$ for $j=0,1,2,3$ if $\alpha_3/\alpha_2 \le \rho_*(\eta)$ and $\nu =\tfrac12 \theta_j(\alpha_3/\alpha_2)$ for $j \in \mathcal{J}(\eta)$ if $\alpha_3/\alpha_2 > \rho_*(\eta)$.
They are critical points of $h$ and correspond to the points given by $\tilde{T}_{c,j}$.

If $\alpha_1=\alpha_2$ there is no more critical points since \eqref{E:app5_mod_pf1} has no solution.
Suppose $\alpha_1\neq \alpha_2$. Then, \eqref{E:app5_mod_pf1} is equivalent to
\[
	 \cos \zeta
	=- \tfrac{\alpha_3 \sin \eta}{(\alpha_1 - \alpha_2)  \sin 2\nu }.
\]
Using this relation, we see that $\partial_\zeta h=0$ is equivalent to 
\[
	\cos 2\nu = - \tfrac{\alpha_3 \cos \eta }{\alpha_1 + \alpha_2}.
\]
To obtain a solution $(\nu,\zeta)$ such that $\sin 2 \nu \neq 0$ to these identities, we need
\[
	\tfrac{\alpha_3^2 \sin^2 \eta}{(\alpha_1 - \alpha_2)^2} \le  1- \tfrac{\alpha_3^2 \cos^2 \eta}{(\alpha_1+\alpha_2)^2}.
\]
Let us denote the set of the critical points $\mathcal{C} \subset ((0.\pi/2)\cup (\pi/2,\pi))\times [0,\pi)$ of $h$.
The condition for the existence of this critical point is summarized as \eqref{E:c5pfothercrit}.
Further, for any $(\nu,\zeta)\in \mathcal{C}$, we have
\[
	h(\nu,\zeta) = 	 - \tfrac{\alpha_3^2(\alpha_1 - \alpha_2 \cos 2 \eta)}{\alpha_1^2-\alpha_2^2}  + 2\alpha_1+r  .
\]
We denote this value $h_c$. The set of the corresponding critical points is $T_c$.

Let us find $g_{\min}$. We compare critical values.
By virtue of Lemma \ref{L:pmin}, 
\[
	\min_{j=0,1,2,3} h(\tfrac12 \theta_j (\alpha_3/\alpha_2),0) = h(\tfrac12 \theta_3 (\alpha_3/\alpha_2),0).
\]
One sees that this value is equal to
\[
	\tilde{h}_c:=  \alpha_2 \cos ( 2\theta_3 (\tfrac{\alpha_3}{\alpha_2})) + 2\alpha_3 \cos (\theta_3 (\tfrac{\alpha_3}{\alpha_2})-\eta) 
	+ 3\alpha_1 + \alpha_2 + r.
\]

If \eqref{E:c5pfothercrit} fails
then there is no other  critical point and hence $g_{\min}=\tilde{h}_c$ and $T_0=\tilde{T}_{c,3}$. 

We consider the case \eqref{E:c5pfothercrit} holds.
Let us first consider the subcase $\alpha_1<\alpha_2$.
If $(\nu,\zeta) \in \mathcal{C}$ then
\[
	\partial_\zeta^2 h = 2(\alpha_1-\alpha_2) (1- \tfrac{ 4\alpha_3^2 \sin^2 \eta \alpha_1\alpha_2 }{ (\alpha_1 + \alpha_2)^2(\alpha_1 - \alpha_2)^2 }).
\]
This has the same sign as $\alpha_1-\alpha_2$ since \eqref{E:c5pfothercrit}  yields
\[
1- \tfrac{ 4\alpha_3^2 \sin^2 \eta \alpha_1\alpha_2 }{ (\alpha_1 + \alpha_2)^2(\alpha_1 - \alpha_2)^2 }
\ge  \tfrac{(\alpha_1-\alpha_2)^2}{\alpha_1^2 + \alpha_2^2 - 2\alpha_1 \alpha_2 \cos 2\eta} >0.
\]
Hence, if $\alpha_1<\alpha_2$ then $\partial_\zeta^2 h<0$ holds and hence $h_c$ is not the minimum value.
Thus, $g_{\min}= \tilde{h}_c$ and $T_0=\tilde{T}_{c,3}$.

Let us consider the subcase $\alpha_1>\alpha_2$ and
\begin{equation}\label{E:c5pfothercrit2}
	\alpha_3^2 < \tfrac{(\alpha_1 + \alpha_2)^2(\alpha_1 - \alpha_2)^2}{\alpha_1^2 + \alpha_2^2 - 2\alpha_1 \alpha_2 \cos 2\eta}.
\end{equation}
One has
\[
	\partial_\zeta^2 h(\nu,0) =
	-2 \sin 2\nu ( \alpha_3 \sin \eta + (\alpha_1-\alpha_2) \sin 2\nu ).
\]
Since $\theta_3 \in (\pi,2\pi)$, $\partial_\zeta^2 h(\tfrac12\theta_3,0)<0$ holds if and only if
$
	\sin  \theta_3 <-\tfrac{\alpha_3 \sin \eta}{ \alpha_1-\alpha_2}.
$
If $\eta<\pi/2$ then the condition reads as $f_{{\alpha_3}/{\alpha_2}, \eta}(\arcsin \tfrac{\alpha_3 \sin \eta}{ \alpha_1-\alpha_2})>0$.
This is equivalent to \eqref{E:c5pfothercrit2}.
A similar argument shows that the condition for $\partial_\zeta^2 h(\tfrac12\theta_3,0)<0$
in the case $\eta \in [\pi/2.\pi)$ is also \eqref{E:c5pfothercrit2}.
In this case $\tilde{h}_c$ is not the minimum point.
Hence, we have $g_{\min}=h_c$ and $T_0=T_c$.

Finally, in the subcase $\alpha_1-\alpha_2>0$ and
\[
	\alpha_3^2 = \tfrac{(\alpha_1 + \alpha_2)^2(\alpha_1 - \alpha_2)^2}{\alpha_1^2 + \alpha_2^2 - 2\alpha_1 \alpha_2 \cos 2\eta},
\]
we see that $(\frac12 \theta_3,0)\in \mathcal{C}$. Hence, we have $g_{\min}=h_c = \tilde{h}_c$ and
$T_0=T_c = \tilde{T}_{c,3}$.
\end{proof}

\section{Proof of Theorem \ref{T:CO}}\label{S:CO}

\begin{proof}
In view of Theorems \ref{T:main}, \ref{T:excited}, \ref{T:stability}, and \ref{T:instability}, it suffices to investigate
critical points of
\[
	g(z_1,z_2) = - \kappa |z_1|^3 - |z_2|^3 - \tfrac{3\gamma}2 \Re (\overline{z_1}^2 z_2).
\]
Regarding $g$ as a function $\R^4\to\R$ by 
$z_1=x_1+iy_1$ and $z_2=x_2+iy_2$, we have
\[
	\nabla_{x_1.y_1,x_2,y_2} g =4(\Re F_1(z_1,z_2), \Im F_1(z_1,z_2), \Re F_2(z_1,z_2), \Im F_2(z_1,z_2)).
\]

If $|z_1|=1$ and $z_2=0$ then
$F_1(z_1,0) = \kappa z_1$ and $F_2(z_1,0) = \tfrac\gamma2 z_1^2$.
Since $(F_1(z_1,0),F_2(z_1,0))$ is not a constant multiple of $(z_1,0)$, one sees from Lagrange's multiplier theorem that
$(z_1,0)$ is not a critical point of $g|_{\partial B}$. 

If $|z_2|=1$ and $z_1=0$ then
$(F_1(0,z_2),F_2(0,z_2)) = (0, z_2)$.
This implies that $(0,z_2)$ is a critical point of $g|_{\partial B}$. The critical value is $g(0,z_2) = -1$.
Hence,
\[
	A_{0,\omega}:= \bigcup_{\theta \in \R/2\pi \Z} \mathcal{R}( (0, e^{2i\theta}) , \omega, 1) \subset \mathcal{A}_\omega
\]
for any $\omega>0$.

To find other critical points, let us introduce
\[
	h(\nu,\zeta) = g(\cos \nu  , e^{i\zeta}\sin \nu ) = - \kappa \cos^3 \nu - \sin^3 \nu -\tfrac{3\gamma}2 \cos^2 \nu \sin \nu \cos \zeta
\]
By the symmetry $g(-z_1,z_2)=g(z_1,z_2)$, we consider $h$ on $(\nu,\zeta) \in ((-\pi/2,0)\cup (0,\pi/2))\times (\R/2\pi\Z)$.
As in Lemma \ref{L:h}, critical points of $h$ in this domain corresponds to those of $g|_{\partial B}$.
We see that $\partial_\zeta h(\nu,\zeta) =\tfrac{3\gamma}2 \cos^2 \nu \sin \nu \sin \zeta =0$
implies $\zeta =0$ or $\pi$.

We first consider the case $\zeta=0$.
Since
\[
	\partial_\nu h(\nu,0)= \tfrac{3}2\cos^3 \nu( 2\kappa \tan \nu  + 2(\gamma-1) \tan^2 \nu
	-\gamma ),
\]
one sees that $\partial_\nu h(\nu,0)=0$ holds if and only if  $2\kappa \tan \nu  + 2(\gamma-1) \tan^2 \nu
	-\gamma =0$.
If $\gamma\neq1$ and $\kappa^2 + 2\gamma (\gamma-1)>0$ then it has two solutions
\[
	\nu_1 = \arctan\tfrac{ \gamma }{ \kappa + \sqrt{\kappa^2 + 2\gamma (\gamma-1) } }, \quad
	\nu_2 = \arctan\tfrac{ \gamma }{ \kappa - \sqrt{\kappa^2 + 2\gamma (\gamma-1) } }.
\]
If $\gamma\in(0,1)$ and $\kappa^2 + 2\gamma (\gamma-1)=0$ then the two solutions are the same.
If $\gamma=1$ and $\kappa>0$ then it has a unique solution $\nu_1=\gamma/2\kappa$.

Let us see $h(\nu_j,0)<0$ if and only if $(\gamma,\kappa) \in {J}_j$ for $j=1,2$. Looking at the sign of $\partial_\nu h(\nu,0)$, one sees that
$h(\nu_1)<h(\pi/2,0)=-1$ holds for $\gamma>1$ or if $\gamma=1$ and $\kappa>0$.
If $\gamma<1$ and $\kappa \ge \sqrt{2\gamma(1-\gamma)}$ then one obtains
$h(\nu_1,0)<h(0,0)=-\kappa<0$.
If $\gamma<1$ and $\kappa \le -\sqrt{2\gamma(1-\gamma)}$ then one obtains
$h(\nu_1)>h(-\pi/2)=1$.
Thus, $h(\nu_1,0)<0$ if and only if $(\gamma,\kappa) \in {J}_1$.
To see the sign of $h(\nu_2)$, we put $b_1=\kappa - \sqrt{\kappa^2 + 2\gamma (\gamma-1)}\neq0$.
Then, $\cos \nu_2 = |b_1|/(\gamma^2+b_1^2)^{1/2}$, $\sin \nu_2 = b_1 \gamma/|b_1|(\gamma^2+b_1^2)^{1/2}$, and $\kappa=(b_1^2+2\gamma(1-\gamma))/2b_1$.
Hence,
\[
	h(\nu_2,0) =  -\tfrac{b_1}{2|b_1|}(\gamma^2+b_1^2)^{-\frac12} ( b_1^2 + 2\gamma ).
\]
We see that $h(\nu_2,0)<0$ if and only if $(\gamma,\kappa) \in {J}_2$.

We next consider the case $\zeta =\pi$.
Since
\[
	\partial_\nu h(\nu,\pi)= \tfrac{3}2\cos^3 \nu( 2\kappa \tan \nu  - 2(\gamma+1) \tan^2 \nu
	+\gamma ),
\]
one sees that $\partial_\nu h(\nu,\pi)=0$ holds if and only if  $-2\kappa \tan \nu  + 2(\gamma+1) \tan^2 \nu
-\gamma =0$. It has two solutions
\[
	\nu_3 = \arctan\tfrac{ \gamma }{  \sqrt{\kappa^2 + 2\gamma (\gamma+1) } -\kappa  }, \quad
	\nu_4 = -\arctan\tfrac{ \gamma }{  \sqrt{\kappa^2 + 2\gamma (\gamma+1) }+\kappa  }.
\]
One deduces from the sign of $\partial_\nu h(\nu,\pi)$ that $h(\nu_4,\pi)>h(-\pi/2,\pi)=1$.
Let us observe that $h(\nu_3,\pi)<0$ if and only if $(\gamma,\kappa) \in J_3$.
To this end, out $b_2=\sqrt{\kappa^2 + 2\gamma (\gamma+1) } -\kappa\in (0,\sqrt{2\gamma (\gamma+1) })$.
Then, since $\cos \nu_2 = b_2/(\gamma^2+b_2^2)^{1/2}$, $\sin \nu_2 =  \gamma/(\gamma^2+b_2^2)^{1/2}$, and $\kappa=(-b_2^2+2\gamma(\gamma+1))/2b_2$, one has
\[
	h(\nu_3,\pi) = \tfrac12 (b_2^2+\gamma^2)^{-\frac12} (b_2^2-2\gamma).
\]
Hence, $h(\nu_3,\pi)<0$ if and only if $b_2<\sqrt{2\gamma}$. The condition is equivalent to $(\gamma,\kappa) \in J_3$.

Let us find $\mathcal{G}_\omega$. 
Recall that $g(0,e^{i\theta})=-1$. 
Hence, $g_{\min}=-1$ and  $\mathcal{G}_\omega = A_{0,\omega}$
for $(\gamma,\kappa) \in (\R_+\times \R)\setminus (J_1\cup J_2 \cup J_3)$.

Note that $J_1 \supset J_2 \cup J_3$.
For $(\gamma,\kappa) \in J_2$, we have $h(\nu_2,0) \ge h(\nu_1,0)$. Further, the equality holds if $\gamma<1$ and $\kappa=\sqrt{2\gamma(1-\gamma)}$.
For $(\gamma,\kappa) \in J_2$, recalling that $\nu_3>0$, we have $h(\nu_3,\pi) > h(\nu_3, 0) \ge h(\nu_1,0)$.
Hence, $(\nu_3,\pi)$ is not the minimum point.
These imply that
\[
	g_{\min} = \min (-1,h(\nu_1,0))
\]
for $(\gamma,\kappa) \in J_1\cup J_2 \cup J_3 = J_1$.

Let us compare this value with $h(\nu_1,0)$ for $(\gamma,\kappa) \in J_1$.
If $\gamma>1$ or if $\gamma=1$ and $\kappa>0$ then we see from the sign of $\partial_\nu h(\nu,0)$ that $h(\nu_1,0)<h(\pi/2,0)=-1$. Hence, $g_{\min}=h(\nu_1,0)$ and hence $\mathcal{G}_\omega = A_{1,\omega}$.
Let us consider the region $0<\gamma<1$ and $\kappa \ge \sqrt{2\gamma (1-\gamma)}$.
Set $b_3=\kappa + \sqrt{\kappa^2 + 2\gamma (\gamma-1)} \ge \sqrt{2\gamma (1-\gamma)}$.
Then, since $\cos \nu_1 = b_3/(\gamma^2+b_3^2)^{1/2}$, $\sin \nu_1 =  \gamma/(\gamma^2+b_3^2)^{1/2}$, and $\kappa=(b_3^2-2\gamma(\gamma-1))/2b_3$, one has
\[
	h(\nu_1,0) = -\tfrac12 (\gamma^2+b_3^2)^{-1/2}(b_3^2+2\gamma).
\]
Hence, $h(\nu_1,0)<-1$ if and only if $b_3> 2\sqrt{1-\gamma}$. 
This condition reads as $\kappa > \kappa_c(\gamma)$.
In this case, we see $g_{\min}=h(\nu_1,0)$ and $\mathcal{G}_\omega = A_{1,\omega}$.
A similar argument shows that if $\kappa = \kappa_c(\gamma)$ then $h(\nu_1,0)=-1$ and hence
$g_{\min}=-1$ and $\mathcal{G}_\omega = A_{0,\omega} \cup A_{1,\omega}$.
Similarly, if $\kappa \in [ \sqrt{2\gamma(1-\gamma)}, \kappa_c(\gamma))$ then
$g_{\min}=-1$ and $\mathcal{G}_\omega = A_{0,\omega}$.
\end{proof}

\appendix

\section{Derivation of the standard systems}\label{S:A}

In this appendix, we show that the systems \eqref{E:nls1}--\eqref{E:nls5} exhaust all standard form
of the systems which have a conserved energy which has a coercive kinetic-energy part.
The precise statement is as follows:
\begin{theorem}\label{T:reduction}
Consider  a system of the form 
\begin{equation}\label{E:gNLS2}
	\left\{
	\begin{aligned}
	(	i \partial_t +\Delta) u_1 
		&= \lambda_1 |u_1|^2 u_1 + \lambda_2 |u_1|^2 u_2+ \lambda_3 u_1^2 \overline{u_2}+ \lambda_4 |u_2|^2 u_1+ \lambda_5 u_2^2 \overline{u_1} + \lambda_6 |u_2|^2 u_2 ,
		\\
	(	i \partial_t +\Delta) u_2
		&=\lambda_7 |u_1|^2 u_1 + \lambda_8 |u_1|^2 u_2+ \lambda_9 u_1^2 \overline{u_2}+ \lambda_{10} |u_2|^2 u_1+ \lambda_{11} u_2^2 \overline{u_1} + \lambda_{12} |u_2|^2 u_2 ,
	\end{aligned}
	\right.
\end{equation}
where $(t,x)\in \R \times \R^3$ and $\lambda_j \in \R$.
Suppose that there exist $a,b,c \in \R$ with $b^2-ac<0$ and a real-valued quartic polynomial $g=g(z_1,z_2)$ satisfying $g(e^{i\theta}z_1,e^{i\theta}z_2)=g(z_1,z_2)$ for any $(z_1,z_2,\theta) \in \C\times \C \times \R$ such that
\begin{equation}\label{E:tmpenergy}
	\int_{\R^3} (a|\nabla u_1|^2 + 2b\Re \overline{\nabla u_1} \cdot\nabla u_2 + c |\nabla u_2|^2 + g(u_1,u_2)) dx
\end{equation}
becomes a time-independent quantity for any $H^1(\R^3) \times H^1(\R^3)$-solution $(u_1(t),u_2(t))$ to the system.
Then, there exists $M \in GL_2(\R)$ such that the system which
\[
		\begin{pmatrix} \tilde{u}_1 \\ \tilde{u}_2 \end{pmatrix}
		=M	\begin{pmatrix} u_1 \\ u_2 \end{pmatrix}
\]
solves becomes either one of \eqref{E:nls1}--\eqref{E:nls5}.
\end{theorem}

\begin{remark}
The condition $b^2-ac<0$ corresponds to the coercivity of quadratic part
the quantity defined in \eqref{E:tmpenergy}.
Here, we restrict ourselves to the three dimensional case. However, this is purely for the sake of simplicity.
Note that the action of $M\in GL(\R)$ on the coefficients $\{\lambda_j\}_{j=1,\dots,12}$ is independent of the
dimensions, and more generally, of the choice of the linear part of the equations.
\end{remark}

\begin{proof}
A  system of the form \eqref{E:gNLS2} is identified with a vector
$(\lambda_j)_{j=1}^{12} \in \R^{12}$.
For a vector $(\lambda_j)_{j=1}^{12} \in \R^{12}$,
we define a matrix $C\in M_3 (\R)$ and a vector $V \in \R^3$ by
\[
C=
\begin{pmatrix}
\lambda_2-\lambda_3 & -\lambda_1+\lambda_8-\lambda_9 & -\lambda_7 \\
\lambda_5 & -\lambda_3+\lambda_{11} & -\lambda_9 \\
\lambda_6 & -\lambda_4+\lambda_5+\lambda_{12} & -\lambda_{10}+\lambda_{11}
\end{pmatrix}  ,\quad
	V=
	\begin{pmatrix} \lambda_8-2\lambda_9 \\ \tfrac12 (-\lambda_2+2\lambda_3-\lambda_{10}+ 2\lambda_{11}) \\  \lambda_4-2\lambda_5 \end{pmatrix}. 
\]
Note that the inverse map is given as follows:
For given $((c_{jk})_{1\le j,k \le 3},(v_\ell)_{1\le \ell \le 3})\in M_3(\R) \times \R^3$, the corresponding nonlinearities are given by
\begin{equation*}
\begin{aligned}
	F_{1} (u_1,u_2) :={}&  -(c_{12}+c_{23}) |u_1|^2 u_1+c_{11}(2|u_1|^2 u_2 + u_1^2\ol{u_2})+c_{21} (2u_1|u_2|^2 + \ol{u_1}u_2^2) + c_{31} |u_2|^2u_2 \\
	& - (\tr{C}) \Re (\overline{u_1} u_2) u_1 +(v_1 |u_1|^2 + 2v_2\Re (\overline{u_1}u_2) + v_3 |u_2|^2)u_1, \\ 
	F_{2} (u_1,u_2) :={}&  -c_{13} |u_1|^2 u_1-c_{23}(2|u_1|^2 u_2 + u_1^2\ol{u_2})-c_{33} (2u_1|u_2|^2 + \ol{u_1}u_2^2) + (c_{21}+c_{32}) |u_2|^2u_2\\
	& + (\tr{C}) \Re (\overline{u_1} u_2) u_2 + (v_1 |u_1|^2 + 2v_2\Re (\overline{u_1}u_2) + v_3 |u_2|^2)u_2.
\end{aligned}
\end{equation*}
In the sequel, we will freely use the identification between a system of the form \eqref{E:gNLS2} and a pair $(C,V)\in M_3(\R) \times \R^3$.

By \cite{MSU2}*{Proposition A.9}, there exists $g$ such that
the quantity \eqref{E:tmpenergy} becomes a conserved quantity if and only if the pair $(C,V)$ satisfies
\begin{equation}\label{E:tmpecriterion}
	C \begin{pmatrix} a \\ b \\ c \end{pmatrix} = 0, \quad
	\begin{pmatrix} \tr C - 2v_2 & 2v_1 & 0 \\ -v_3 & \tr C & v_1  \\ 0 & - 2v_3 & \tr C + 2v_2 \end{pmatrix} \begin{pmatrix} a \\ b \\ c \end{pmatrix}=0,
\end{equation}
where $v_j$ ($j=1,2,3$) is the $j$-th component of the vector $V$.
In particular, we have $\rank C \le 2$.
\smallskip

\noindent{\bf Case 1}.
We first consider the case $\rank C=2$. The assumption $b^2-ac<0$ implies that $C$ belongs to 
$Z_{\mathrm{e}}$ given in \cite{M}*{Section 1.4}.
Then,  thanks to \cite{M}*{Theorem 1.10}, there exists $M \in GL_2(\R)$ such that
the matrix $C'$ corresponding to the transformed system becomes either one of the element in $T_{\mathrm{e}}$
given as follows:
$T_{\mathrm{e}}:= T_{\mathrm{e},1} \cup  T_{\mathrm{e},2}\cup T_{\mathrm{e},3}
$, where
\[
	A_1 := \begin{pmatrix} 1 & 0 & -1  \\ 0 & 2 & 0 \\ -1 & 0 & 1 \end{pmatrix},
	\quad
	A_2:= \begin{pmatrix} 0 & - 2 & 0 \\ 1 & 0 & - 1 \\ 0 & 2 & 0 \end{pmatrix},
	\quad
	A_3:= \begin{pmatrix} 0 & - 2 & 0 \\ -1 & 0 & 1 \\ 0 & 2 & 0 \end{pmatrix}
\]
and
\[
	A_4(\eta):= \begin{pmatrix} \sin \eta& -2\cos \eta & -\sin \eta  \\ 0&0&0 \\ \sin \eta & -2\cos \eta & -\sin \eta \end{pmatrix}
\]
and
\begin{align*}
	T_{\mathrm{e},1} :={}& 
	\left\{ p_1 A_1 + p_2 A_2 + p_3 A_3 + p_4 A_4(0) \ \middle| \ (p_1,p_2,p_3,p_4) \in \mathcal{T}_1 \right\},\\
	T_{\mathrm{e},2} :={}&
	\bigcup_{\eta \in [0,\frac{\pi}2]}  \left\{ p_1 A_1 + p_2 A_2 + p_3 A_3 + p_4 A_4(\eta) \ \middle| \ (p_1,p_2,p_3,p_4) \in \mathcal{T}_2(\zeta) \right\}, \\
	T_{\mathrm{e},3} := {}& \bigcup_{\zeta \in (\frac{\pi}2,\pi)}  \left\{ p_1 A_1 + p_2 A_2 + p_3 A_3 + p_4 A_4(\eta) \ \middle| \ (p_1,p_2,p_3,p_4) \in \mathcal{T}_3 \right\} 
\end{align*}
with $S^3_+ := \{ (w,x,y,z) \in \R^4 \ |\ w^2+x^2+y^2+z^2=1,\, w \ge 0, \, y \ge 0,\, 0 \le z<1 \}$ and
\begin{align*}
	\mathcal{T}_1 :={}&\{ (w,x,y,z) \in S^3_+ \ |\ y=0 \} \cup \{ (w,x,y,z) \in S^3_+ \ |\ y \neq 1/\sqrt{2},\, z=0 \}, \\
	\mathcal{T}_2(\eta) :={}&\{ (w,x,y,z) \in S^3_+ \ |\ y>0,\, z>0 \} \\
	&\setminus \{ (w,x,y,z) \in S^3_+ \ |\ (w,x,y) = k (\sin 2\eta, \cos 2\eta,1),\, \exists k \in \R \} , \\
	\mathcal{T}_3 :={}&\{ (w,x,y,z) \in S^3_+ \ |\ x>0,\, y>0,\, z>0 \}. 
\end{align*}
Since the original system has a conserved quantity of the form \eqref{E:tmpenergy}, so does the transformed system.
Further, since $\ker C'$ is spanned by $\ltrans{(1,0,1)}$, we see from the criterion \eqref{E:tmpecriterion} that
the conserved energy
of the transformed system takes the form
\[
	\int_{\R^3} (|\nabla \tilde{u}_1|^2 +  |\nabla \tilde{u}_2|^2 + \tilde{g}(\tilde{u}_1,\tilde{u}_2)) dx.
\]
Let $\tilde{v}_j$ ($j=1,2,3$) be the $j$-th component of the transformed vector $\tilde{V}$.
Then we see from the criterion \eqref{E:tmpecriterion} that
$\tr C' - 2\tilde{v}_2 = \tr C' + 2\tilde{v}_2 =0$ and $\tilde{v}_1=\tilde{v}_3$.
In particular, $\tr C'=0$ and $\tilde{v}_2=0$. Thus, we see from the formula $C'=p_1 A_1 + p_2 A_2 + p_3 A_3 + p_4 A_4(\eta)$ that $p_1=\frac14 \tr C'=0$.

If $C' = \alpha_1 A_2 + \alpha_2 A_3 \in T_{\mathrm{e},1}$ ($(0,\alpha_1,\alpha_2,0) \in \mathcal{T}_1$)
then the transformed system is \eqref{E:nls3}.
In the case $C' = \alpha_1 A_2 + \alpha_3 A_4(0) \in T_{\mathrm{e},1}$ ($(0,\alpha_1,0,\alpha_3) \in \mathcal{T}_1$)
and the case $C' = \alpha_1 A_2 + \alpha_2 A_3 + \alpha_3 A_4(0) \in T_{\mathrm{e},2}$ ($(0,\alpha_1,\alpha_2,\alpha_3) \in \mathcal{T}_2(0)$), we have \eqref{E:nls4}.
Finally, \eqref{E:nls5} is obtained if $C' \in T_{\mathrm{e},3}$ or if
$C' = \alpha_1 A_2 + \alpha_2 A_3 + \alpha_3 A_4(\zeta) \in T_{\mathrm{e},2}$ ($(0,\alpha_1,\alpha_2,\alpha_3) \in \mathcal{T}_2(\zeta)$)) for some $\zeta \in (0,\pi/2]$.
This completes the proof in the case $\rank C=2$.
\smallskip

\noindent{\bf Case 2}.
Let us proceed to the case $\rank C=1$. In this case, by means of \cite{MSU2}*{Theorem A.16},
there exist $M \in GL_2(\R)$ and
$(k_1,k_2,k_3) \in \R^3\setminus\{(0,0,0)\}$ such that
the matrix $C'$ corresponding to the transformed system becomes either one of the following three forms:
\[
	\begin{pmatrix}
	0 & k_1 & 0 \\
	0 & k_2 & 0 \\
	0 & k_3 & 0 
	\end{pmatrix}, \quad
	\begin{pmatrix}
	0 & 0 & k_1 \\
	0 & 0 & k_2 \\
	0 & 0 & k_3 
	\end{pmatrix}, \quad
	\begin{pmatrix}
	 k_1 & 0 & k_1\\
	 k_2 & 0 & k_2\\
	 k_3 & 0 & k_3
	\end{pmatrix}.
\]
The assumption $b^2-ac<0$ implies that $\ker C'$ contains a vector $\ltrans{(a',b',c')} \in \R^3$ such that $(b')^2-a'c'<0$. This implies that $C'$ takes the first form and that $\ker C'$ is spanned by $\ltrans{(1,0,0)}$ and $\ltrans{(0,0,1)}$.
Hence, by means of \eqref{E:tmpecriterion}, there exists a pair of positive numbers $a'$ and $c'$
such that
\[
	\begin{pmatrix} \tr C' - 2\tilde{v}_2 & 2\tilde{v}_1 & 0 \\ -\tilde{v}_3 & \tr C' & \tilde{v}_1  \\ 0 & - 2\tilde{v}_3 & \tr C' + 2\tilde{v}_2 \end{pmatrix} \begin{pmatrix} a' \\ 0 \\ c' \end{pmatrix}=0,
\]
where $\tilde{v}_j$ ($j=1,2,3$) is the $j$-th component of the vector part of the transformed system.
One then sees that $\tr C'=\tilde{v}_2=0$.
Further, either $(\tilde{v}_1,\tilde{v}_3)=(0,0)$ or $\tilde{v}_1\tilde{v}_3>0$ holds.
 
Hence, by virtue of \cite{MSU2}*{Theorem A.13}, the triplet $(C',\tilde{v}_1,\tilde{v}_3)$ can be chosen one of the following three forms:
\begin{equation}\label{e:Z_+1}
	C'  = \begin{pmatrix} 0 & 1  & 0 \\ 0 &0 & 0 \\ 0 & 1 & 0  \end{pmatrix},
	\quad \tilde{v}_1,\tilde{v}_3 \in \R,
\end{equation}
\begin{equation}\label{e:Z_+2}
	C'  = \pm \begin{pmatrix} 0 & 1  & 0 \\ 0 &0 & 0 \\ 0 & -1 & 0  \end{pmatrix},
	\quad \tilde{v}_1 \ge \tilde{v}_3,
\end{equation}
or
\begin{equation}\label{e:Z_+3}
	C'  =\pm \begin{pmatrix} 0 &  1  & 0 \\ 0 &0 & 0 \\ 0 & 0 & 0  \end{pmatrix},
	\quad  \tilde{v}_1\in \R,\, \tilde{v}_3\in \{-1,0,1\}.
\end{equation}

If $C'$, $\tilde{v}_1$, and $\tilde{v}_3$ are as in \eqref{e:Z_+1}, the transformed system takes the form
\begin{equation*}
	\left\{
	\begin{aligned}
	(	i \partial_t +\Delta) \tilde{u}_1 
		&= -|\tilde{u}_1|^2\tilde{u}_1 +(\tilde{v}_1 |\tilde{u}_1|^2+\tilde{v}_3 |\tilde{u}_2|^2)\tilde{u}_1,\\
	(	i \partial_t + \Delta) \tilde{u}_2
		&= |\tilde{u}_2|^2 \tilde{u}_2+(\tilde{v}_1 |\tilde{u}_1|^2+ \tilde{v}_3 |\tilde{u}_2|^2)\tilde{u}_2 .
	\end{aligned}
	\right.
\end{equation*}
If $\tilde{v}_1=\tilde{v}_3=0$ then it is \eqref{E:nls1}.
If $\tilde{v}_1\tilde{v}_3>0$ then
by introducing a new variable $(u^\dagger_1, u^\dagger_2)=( \sqrt{|\tilde{v}_3|}\tilde{u}_2,\sqrt{|\tilde{v}_1|}\tilde{u}_1)$,
we obtain
\begin{equation*}
	\left\{
	\begin{aligned}
	(	i \partial_t +\Delta) u^\dagger_1 
		&= \tfrac1{|\tilde{v}_3|} |u^\dagger_1|^2u^\dagger_1 +(\sign \tilde{v}_1) (|u^\dagger_1|^2+ |u^\dagger_2|^2)u^\dagger_1,\\
	(	i \partial_t + \Delta) u^\dagger_2
		&= -\tfrac1{|\tilde{v}_1|}|u^\dagger_2|^2 u^\dagger_2+(\sign \tilde{v}_1) ( |u^\dagger_1|^2+|u^\dagger_2|^2)u^\dagger_2 ,
	\end{aligned}
	\right.
\end{equation*}
which is \eqref{E:nls2} with $\alpha\beta<0$.

If $C'$, $\tilde{v}_1$, and $\tilde{v}_3$ are as in \eqref{e:Z_+2}, the system takes the form
\begin{equation*}
	\left\{
	\begin{aligned}
	(	i \partial_t +\Delta) \tilde{u}_1 
		&= \tilde{\sigma} |\tilde{u}_1|^2\tilde{u}_1 +(\tilde{v}_1|\tilde{u}_1|^2+\tilde{v}_3|\tilde{u}_2|^2)\tilde{u}_1,\\
	(	i \partial_t + \Delta) \tilde{u}_2
		&= \tilde{\sigma} |\tilde{u}_2|^2 \tilde{u}_2+(\tilde{v}_1|\tilde{u}_1|^2+\tilde{v}_3|\tilde{u}_2|^2)\tilde{u}_2 
	\end{aligned}
	\right.
\end{equation*}
with $\tilde{\sigma} \in \{\pm1\}$.
If $\tilde{v}_1=\tilde{v}_3=0$ then it is \eqref{E:nls1}.
If $\tilde{v}_1\tilde{v_3}$ then
by introducing a new variable $(u^\dagger_1, u^\dagger_2)=(\sqrt{|\tilde{v}_1|}\tilde{u}_1, \sqrt{|\tilde{v}_3|}\tilde{u}_2)$,
we obtain
\begin{equation*}
	\left\{
	\begin{aligned}
	(	i \partial_t +\Delta) u^\dagger_1 
		&= \tfrac{\tilde{\sigma}}{|\tilde{v}_1|} |u^\dagger_1|^2u^\dagger_1 +(\sign \tilde{v}_1)(|u^\dagger_1|^2+  |u^\dagger_2|^2)u^\dagger_1,\\
	(	i \partial_t + \Delta) u^\dagger_2
		&= \tfrac{\tilde{\sigma}}{|\tilde{v}_3|}|u^\dagger_2|^2 u^\dagger_2+(\sign \tilde{v}_1)( |u^\dagger_1|^2+  |u^\dagger_2|^2)u^\dagger_2 .
	\end{aligned}
	\right.
\end{equation*}
By swapping $u^\dagger_1 $ and $u^\dagger_2$ if necessary, we obtain \eqref{E:nls2} with $\alpha \beta >0$.

If $C'$, $\tilde{v}_1$, and $\tilde{v}_3$ are as in \eqref{e:Z_+3}, the system takes the form
\begin{equation*}
	\left\{
	\begin{aligned}
	(	i \partial_t +\Delta) \tilde{u}_1 
		&= \tilde{\sigma} |\tilde{u}_1|^2\tilde{u}_1 +(\tilde{v}_1 |\tilde{u}_1|^2+\tilde{v}_3 |\tilde{u}_2|^2)\tilde{u}_1,\\
	(	i \partial_t + \Delta) \tilde{u}_2
		&= (\tilde{v}_1|\tilde{u}_1|^2+ \tilde{v}_3 |\tilde{u}_2|^2)\tilde{u}_2 
	\end{aligned}
	\right.
\end{equation*}
with $\tilde{\sigma} \in \{\pm1\}$.
If $\tilde{v}_1=\tilde{v}_3=0$ then it is \eqref{E:nls1}.
If $\tilde{v}_1\neq0$ and $\tilde{v}_3= \sign \tilde{v}_1$ then
by introducing a new variable $(u^\dagger_1, u^\dagger_2)=(\sqrt{|\tilde{v}_1|}\tilde{u}_1, \tilde{u}_2)$,
we obtain
\begin{equation*}
	\left\{
	\begin{aligned}
	(	i \partial_t +\Delta) u^\dagger_1 
		&= -\tfrac{\tilde{\sigma}}{|\tilde{v}_1|} |u^\dagger_1|^2u^\dagger_1 +(\sign \tilde{v}_1)( |u^\dagger_1|^2+  |u^\dagger_2|^2)u^\dagger_1,\\
	(	i \partial_t + \Delta) u^\dagger_2
		&= (\sign \tilde{v}_1)( |u^\dagger_1|^2+  |u^\dagger_2|^2)u^\dagger_2 .
	\end{aligned}
	\right.
\end{equation*}
By swapping $u^\dagger_1 $ and $u^\dagger_2$ if $\tilde{\sigma}=1$, we obtain \eqref{E:nls2} with $\alpha \beta =0$ but $(\alpha,\beta)\neq (0,0)$.
\smallskip

\noindent{\bf Case 3}.
Finally, we consider the case $\rank C=0$. In this case $C=0$. 
By the presence of a conserved energy with a coercive kinetic-energy part, after a change of variable if necessary,
there exists a pair of positive numbers $a'$ and $c'$
such that
\[
	\begin{pmatrix}  - 2\tilde{v}_2 & 2\tilde{v}_1 & 0 \\ -\tilde{v}_3 & 0 & \tilde{v}_1  \\ 0 & - 2\tilde{v}_3 &  2\tilde{v}_2 \end{pmatrix} \begin{pmatrix} a' \\ 0 \\ c' \end{pmatrix}=0,
\]
where $\tilde{v}_j$ ($j=1,2,3$) is the $j$-th component of the vector part of the transformed system.
One then sees that $\tilde{v}_2=0$ and that either $(\tilde{v}_1,\tilde{v}_3)=(0,0)$ or $\tilde{v}_1\tilde{v}_3>0$ holds, as in the $\rank C=1$ case.
Then, the transformed system takes the form
\begin{equation*}
	\left\{
	\begin{aligned}
	(	i \partial_t +\Delta) \tilde{u}_1 
		&= (\tilde{v}_1 |\tilde{u}_1|^2+\tilde{v}_3 |\tilde{u}_2|^2)\tilde{u}_1,\\
	(	i \partial_t + \Delta) \tilde{u}_2
		&= (\tilde{v}_1 |\tilde{u}_1|^2+ \tilde{v}_3 |\tilde{u}_2|^2)\tilde{u}_2 .
	\end{aligned}
	\right.
\end{equation*}
If $\tilde{v}_1 = \tilde{v_3}=0$ then it is \eqref{E:nls1}. 
If $\tilde{v}_1 \neq0$ then
the system fir $(u^\dagger_1, u^\dagger_2)=(\sqrt{|\tilde{v}_1|}\tilde{u}_1, \sqrt{|\tilde{v}_3|}\tilde{u}_2)$
becomes
\begin{equation*}
	\left\{
	\begin{aligned}
	(	i \partial_t +\Delta) u^\dagger_1 
		&= (\sign \tilde{v}_1)(|u^\dagger_1|^2+  |u^\dagger_2|^2)u^\dagger_1,\\
	(	i \partial_t + \Delta) u^\dagger_2
		&= (\sign \tilde{v}_1)( |u^\dagger_1|^2+  |u^\dagger_2|^2)u^\dagger_2 .
	\end{aligned}
	\right.
\end{equation*}
which is the remaining case $\alpha=\beta=0$ of \eqref{E:nls2}.
\end{proof}

\subsection*{Acknowledgements} 
The author would like to express the most profound appreciation to professors Masahito Ohta, Noriyoshi Fukaya, and Yoshinori Nishii for their fruitful discussions and for giving constructive suggestions on the preliminary version of the results.
The author would like to thank professor Norihisa Ikoma for the helpful discussions and for the information on references.
Deep gratitude goes to professor Sim\~{a}o Correia for drawing the author's attention to the reference \cite{Co,Co2}.
The author was supported by JSPS KAKENHI Grant Numbers JP21H00991 and JP21H00993.

\end{document}